\documentclass{amsart}
\usepackage[utf8]{inputenc}
\usepackage{amsthm}
\usepackage{amsmath}
\usepackage{amssymb} 
\usepackage{comment}
\usepackage{graphicx}
\usepackage{mathrsfs}
\usepackage{fullpage}
\usepackage{etex}
\usepackage{pinlabel}

\usepackage{colonequals}
\usepackage{soul}
\usepackage[all]{xy}
\usepackage[usenames,dvipsnames]{color}

\numberwithin{equation}{subsection} 
\numberwithin{figure}{subsection} 

\makeatletter
\let\c@equation\c@figure
\makeatother

\DeclareMathOperator{\Hom}{Hom}
\DeclareMathOperator{\Aut}{Aut}
\DeclareMathOperator{\im}{im}
\DeclareMathOperator{\Mod}{Mod}
\DeclareMathOperator{\Pic}{Pic}
\DeclareMathOperator{\Hur}{Hur}

\DeclareMathOperator{\Gal}{Gal}

\newtheorem{theorem}[equation]{Theorem}
\newtheorem{corollary}[equation]{Corollary}
\newtheorem{lemma}[equation]{Lemma}
\newtheorem{proposition}[equation]{Proposition}
\newtheorem{conjecture}[equation]{Conjecture}

\theoremstyle{definition}
\newtheorem{remark}[equation]{Remark}

\newtheorem{question}[equation]{Question}
\newtheorem{construction}[equation]{Construction}

\theoremstyle{definition}
\newtheorem{definition}[equation]{Definition}

\newcommand\nc{\newcommand}
\nc\on{\operatorname}

\title{Surface bundles and the section conjecture}
\author{Wanlin Li, Daniel Litt, Nick Salter, Padmavathi Srinivasan}
\date{\today}

\begin{document}
\begin{abstract}
    We formulate a tropical analogue of Grothendieck's section conjecture: that for every stable graph $\Gamma$ of genus $g>2$, and every field $k$, the generic curve with reduction type $\Gamma$ over $k$ satisfies the section conjecture. We prove many cases of this conjecture. In so doing we show the existence of many examples of curves with no rational points satisfying the section conjecture over fields of geometric interest, and then over $p$-adic fields and number fields via a Chebotarev argument.
    
    We construct two Galois cohomology classes $o_1$ and $\widetilde{o_2}$, which obstruct the existence of $\pi_1$-sections and hence of rational points. The first is an abelian obstruction, closely related to the period of a curve and to a cohomology class on the moduli space of curves $\mathscr{M}_g$ studied by Morita. The second is a $2$-nilpotent obstruction and appears to be new. We study the degeneration of these classes via topological techniques, and we produce examples of surface bundles over surfaces where these classes obstruct sections. We then use these constructions to show the existence of curves over $p$-adic fields and number fields where each class obstructs $\pi_1$-sections and hence rational points.
    
    Among our geometric results are a new proof of the section conjecture for the generic curve of genus $g\geq 3$, and a proof of the section conjecture for the generic curve of even genus with a rational divisor class of degree one (where the obstruction to the existence of a section is genuinely non-abelian).
\end{abstract}

\maketitle

\section{Introduction}

The goal of this paper is to give a systematic way to prove the existence of many examples of curves for which Grothendieck's section conjecture holds. We begin by formulating a geometric analogue of the section conjecture ``at the boundary of $\mathscr{M}_g$," which we refer to as the \emph{tropical section conjecture}. We prove many cases of this conjecture, giving \emph{geometric} examples where the section conjecture holds, by an analysis of the degeneration of certain cohomology classes on the moduli space of curves, $\mathscr{M}_g$, and on the moduli space of degree one divisor classes on the universal curve,  $\mathbf{Pic}^1_{\mathscr{C}_g/\mathscr{M}_g}$.  We then use a Chebotarev density argument to prove the existence of (many) examples of curves over $p$-adic fields for which the section conjecture holds, which yields by standard approximation techniques (which we omit) the existence of many examples over number fields. Our methods are inspired by ``arithmetic topology," and indeed, a key step in our construction is to pass from the existence of certain (topological) surface bundles over surfaces to curves over $p$-adic fields with analogous properties. In our view the main interest in this paper arises from its fusion of topological, geometric, and arithmetic techniques.

The main technical innovation of the paper is the study of the Gysin images of two cohomology classes $o_1, \widetilde{o_2}$ which obstruct $\pi_1$-sections and hence rational points, as well as the construction of $\widetilde{o_2}$ itself (which appears to be new, though it is inspired by work of Jordan Ellenberg \cite{ellenberg2nilpotent} and Wickelgren \cite{wickelgren2009lower}). 

\subsection{The section conjecture}
Let $k$ be a field and let $X$ be a smooth projective $k$-curve (that is, a smooth, projective, separated, geometrically connected $k$-scheme of dimension $1$). Let $\overline{k}$ be a separable closure of $k$ and let $\bar x\in X(\overline{k})$ be a geometric point of $X$. Then there is a short exact sequence \begin{equation}\label{fundamental-exact-sequence} 1\to \pi_1^{\text{\'et}}(X_{\overline k}, \bar x)\to \pi_1^{\text{\'et}}(X, \bar x)\to \on{Gal}(\overline k/k)\to 1,\end{equation} where $\pi_1^{\text{\'et}}$ denotes the \'etale fundamental group. To each rational point $x\in X(k)$ one may associate a canonical conjugacy class of splittings $[s_x]$ of this exact sequence. We call splittings of sequence (\ref{fundamental-exact-sequence}) $\pi_1$-sections. The starting point for this work is Grothendieck's section conjecture, which suggests that in many cases the sequence above encodes all of the arithmetic of $X$: 
\begin{conjecture}[The section conjecture \cite{grothendieck1997letter}]
Suppose $k$ is a finitely-generated field of characteristic $0$ and that the genus of $X$ is at least $2$. Then the map $$X(k)\to \{\text{splittings of sequence (\ref{fundamental-exact-sequence})}\}/\text{conjugacy}$$
$$x\mapsto [s_x]$$ is a bijection.
\end{conjecture}
Though Grothendieck originally made his conjecture only over finitely-generated fields of characteristic $0$, it is widely believed to hold true in more generality --- for example, over $p$-adic fields. 

Following Stix \cite{stix2010period}, we say that a smooth projective curve $X/k$ \emph{trivially} satisfies the section conjecture if sequence (\ref{fundamental-exact-sequence}) has no sections. As any map with target the empty set is a bijection, such a curve evidently does in fact satisfy the section conjecture. (Of course it is in general by no means trivial to show that a curve does indeed trivially satisfy the section conjecture.) While there is now a fair amount of evidence for the section conjecture (see e.g.~\cite{stix2012rational}), all known examples of curves $X/k$ satisfying the section conjecture do so trivially, at least to the authors' knowledge. Though this state of affairs is disappointing, it is perhaps worth noting that proving the section conjecture in the case of $k$-curves $X$ with $X(k)=\emptyset$ would in fact imply the section conjecture for all $X/k$ \cite[Appendix C]{stix2010period}.

\subsection{The tropical section conjecture}\label{subsec:tropical-section-conjecture}
Let $g\geq 2$ be an integer. Recall that the moduli space $\mathscr{M}_g$ of smooth projective curves of genus $g$ has a Deligne-Mumford compactification $\overline{\mathscr{M}_g}$, parametrizing stable curves of genus $g$. The boundary strata of $\overline{\mathscr{M}_g}$ are indexed by stable graphs (see Section \ref{sec:Mg-preliminaries} for more details). For a stable graph $\Gamma$, we denote the corresponding boundary stratum by $Z_\Gamma$.

Now suppose $g>2$, and let $k$ be a field. For each boundary stratum $Z_\Gamma$ of $\overline{\mathscr{M}_g}$, let $\widehat{K_\Gamma}$ be the fraction field of the complete local ring $\mathscr{O}_{\widehat{K_\Gamma}}:=\widehat{\mathscr{O}_{\overline{\mathscr{M}_{g, k}}, Z_\Gamma}}$ of $\overline{\mathscr{M}_{g, k}}$ at the generic point of $Z_\Gamma$, and let $\mathscr{C}_{\widehat{K_{\Gamma}}}$ be the fiber of the universal curve $\mathscr{C}_g$ over $\widehat{K_\Gamma}$. 
\begin{conjecture}[Tropical section conjecture]\label{conj:tropical-section-conjecture}
For every field $k$ and stable graph $\Gamma$, the curve $\mathscr{C}_{\widehat{K_\Gamma}}$ \emph{trivially} satisfies the section conjecture.
\end{conjecture}
That is, we conjecture that the sequence (\ref{fundamental-exact-sequence}) has no sections when we set $X=\mathscr{C}_{\widehat{K_\Gamma}}$. We think of $\mathscr{C}_{\widehat{K_\Gamma}}$ as ``the generic curve with reduction type $\Gamma$." In our view this conjecture is interesting because it aims to capture the \emph{local, geometric} reasons for the truth of the section conjecture. Indeed, the methods we use to prove special cases of this conjecture can also be used to show the existence of arithmetic examples of curves over local fields satisfying the section conjecture for geometric reasons, as we explain later on in this introduction.
\begin{remark}
One can verify that the curves $\mathscr{C}_{\widehat{K_\Gamma}}$ have no rational points, at least in characteristic $0$, consistent with Conjecture \ref{conj:tropical-section-conjecture}. Indeed, it follows from the main result of Hubbard's thesis \cite{hubbard1972non} that if $\Gamma$ is the trivial graph, consisting only of a single vertex of genus $g$, then $\mathscr{C}_{\widehat{K_\Gamma}}$ has no $\widehat{K_\Gamma}$-rational points; this is the case of the generic curve. For non-trivial $\Gamma$, the analysis of rational points of the generic $n$-pointed curve by Earle and Kra \cite{earle1976sections} implies that the only rational points on the special fiber of the canonical curve over $\mathscr{O}_{\widehat{K_\Gamma}}$ are nodes.  Now an analysis of the deformation theory of these nodes shows that none of them lift to $\widehat{K_\Gamma}$-rational points of $\mathscr{C}_{\widehat{K_\Gamma}}$.
\end{remark}
\begin{remark}
The assumption $g>2$ is necessary so that $\widehat{K_\Gamma}$ is a field, rather than a gerbe over the spectrum of a field. This condition can be relaxed to $g\geq 2$ if one allows such objects into the formulation of the section conjecture, but we felt doing so would introduce unnecessary clutter.
\end{remark}

One of the main purposes of this paper is to prove several special cases of this conjecture, and indeed to identify precise obstructions to the splitting of (\ref{fundamental-exact-sequence}) in these cases. All of the cases of this conjecture that we verify are of a combinatorial nature, as we now explain.

Let $\widetilde{\Gamma}$ be a connected \emph{tropical curve}, i.e.~a connected metric graph such that the underlying graph $\Gamma$ is stable in the sense of Section \ref{sec:Mg-preliminaries}.   One may associate (non-functorially) a compact orientable surface $\Sigma_\Gamma$ to $\Gamma$, with marked loops $\gamma_e$ on $\Sigma_\Gamma$ for each edge $e$ of $\Gamma$ by ``inflating $\Gamma$" (see Figure \ref{fig:graph-surface} for an illustration, and Section~\ref{sec:Mg-preliminaries} for a precise description). We denote the Dehn twist around $\gamma_e$ by $T_e$.
\begin{figure}
    \centering
    \includegraphics[scale=.3]{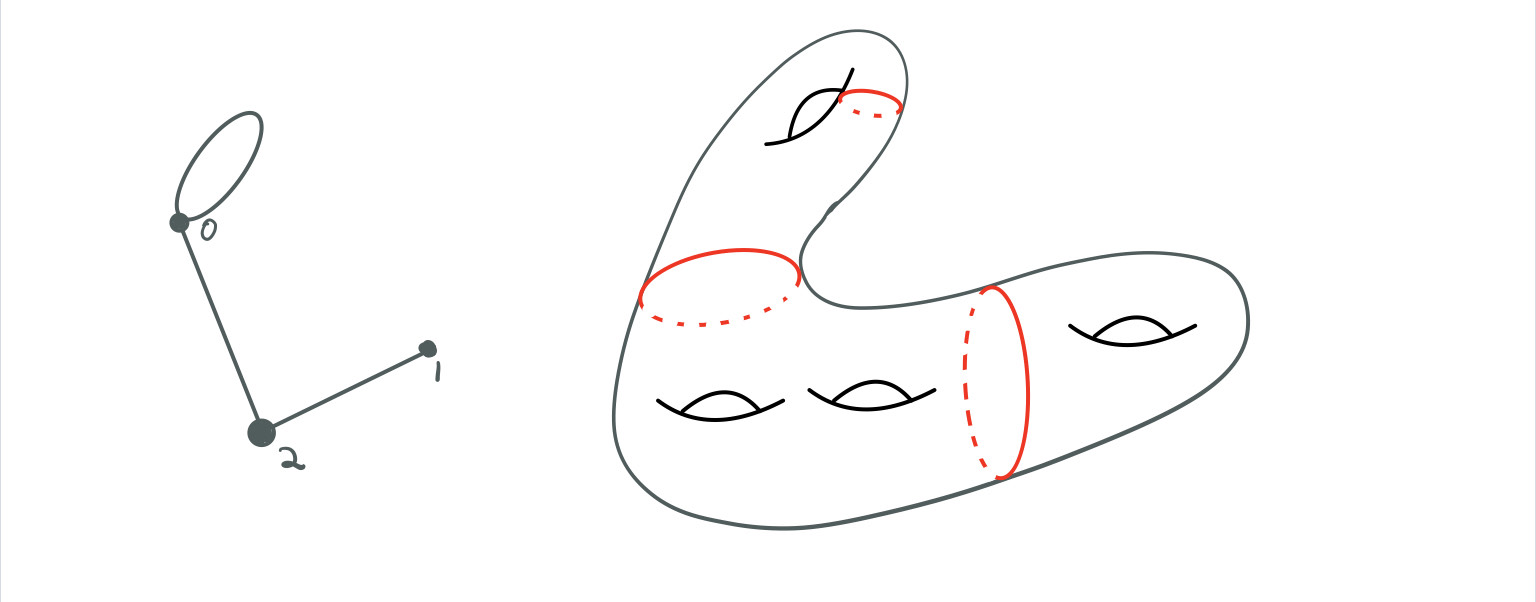}
    \caption{The surface $\Sigma_\Gamma$ associated to a stable graph $\Gamma$. Here each edge of $\Gamma$ corresponds to a marked loop in $\Sigma_\Gamma$, and a vertex of $\Gamma$ labeled by an integer $g$ corresponds to a component of the complement of the marked loops in $\Sigma_\Gamma$ of genus $g$.}
    \label{fig:graph-surface}
\end{figure}
Now assume that the edge lengths $\ell(e)$ of the metric graph $\widetilde{\Gamma}$ are positive integers. Let $T_\Gamma$ be the Dehn multitwist along the marked curves in $\Sigma_\Gamma$, that is, $$T_\Gamma=\prod_{e\in E} T_e^{\ell(e)}.$$ Let $G$ be a group of orientation-preserving mapping classes acting on $\Sigma_\Gamma$ and permuting the loops $\gamma_e$ up to isotopy, such that if $g(\gamma_e)=\gamma_{e'}$ for some $g\in G$, then $\ell(e)=\ell(e')$. Then $G$ commutes with $T_\Gamma$ up to isotopy and so we obtain an action (up to isotopy) of $\langle T_\Gamma \rangle\times G$ on $\Sigma_\Gamma$, where $\langle T_\Gamma \rangle$ is the subgroup of the mapping class group generated by $T_\Gamma$. Hence we have a fibration $$W=\Sigma_\Gamma \times_{(\langle T_\Gamma \rangle \times G)} E(\langle T_\Gamma \rangle \times G)\to B(\langle T_\Gamma \rangle\times G)$$ with fiber $\Sigma_\Gamma$, where here $E(\langle T_\Gamma \rangle \times G)$ is a contractible space with free $\langle T_\Gamma \rangle\times G$-action and $$B(\langle T_\Gamma \rangle \times G)=E(\langle T_\Gamma \rangle \times G)/(\langle T_\Gamma \rangle \times G)$$ denotes the classifying space. The long exact sequence in homotopy groups gives:
\begin{equation} \label{fundamental-tropical-sequence} 1 \to \pi_1(\Sigma_\Gamma)\to \pi_1(W)\to \langle T_\Gamma \rangle \times G\to 1.\end{equation}
\begin{question}\label{tropical-question}
For which $\Gamma, G$ does sequence (\ref{fundamental-tropical-sequence}) split?
\end{question}
One observation of this paper is that in some cases, answering this purely topological question for certain $G,\Gamma$, allows us to (non-trivially) deduce Conjecture \ref{conj:tropical-section-conjecture} for $\Gamma$ and certain $k$. 
\subsection{Main results}
\subsubsection{Geometric results}\label{subsubsec:geometric-results}
Let $\Gamma, \Gamma'$ be stable graphs. We say that $\Gamma$ \emph{specializes} to $\Gamma'$ if $\Gamma$ can be obtained from $\Gamma'$ by contracting edges (or equivalently, if $Z_{\Gamma'}$ is in the closure of $Z_\Gamma$). So, for example the graph consisting only of a single vertex of genus $g$ specializes to every stable graph of genus $g$. 

Our first result is a verification of Conjecture \ref{conj:tropical-section-conjecture} in many cases. A simple-to-state special case of this result is:

\begin{theorem}\label{thm:main-geometric-thm}
 Let $k$ be a field of characteristic $0$.
 \begin{enumerate}
\item Let $g> 2$ be an integer. Let $C_{g-1}$ be the stable graph consisting of a $(g-1)$-cycle all of whose vertices have genus $1$. Then, if $\Gamma$ is any graph which specializes to $C_{g-1}$, the section conjecture is trivially true for $\mathscr{C}_{\widehat{K_\Gamma}}$.
\item  Let $g>2$ be an even integer, and let $T_g$ be any stable tree of genus $g$ admitting an involution that fixes no vertices and stabilizes a unique edge.  Then, if $\Gamma$ is any graph which specializes to $T_g$, the section conjecture is trivially true for $\mathscr{C}_{\widehat{K_\Gamma}}$.
\end{enumerate}
\end{theorem}
 See Figures \ref{Fig:Genus4Max}, \ref{Fig:C34} for other examples of graphs for which our method succeeds, including some graphs corresponding to boundary components of $\overline{\mathscr{M}_g}$ of maximal codimension. See Corollary \ref{cor:o1-tropical-section-conjecture} for a strenghtening of (1) and Corollary \ref{cor:tropical-section-on-pic} for a strengthening of (2).

In fact we prove a substantially stronger result---in case (1) of Theorem \ref{thm:main-geometric-thm}, we show that there is an obstruction to splitting sequence (\ref{fundamental-exact-sequence}) arising from the abelianization of the geometric étale fundamental group, and in case (2) we show that, while there is no such abelian obstruction, there is an obstruction arising from the second nilpotent quotient of the geometric étale fundamental group.

As a consequence of our argument in case (2), we find: 
\begin{theorem}\label{thm:generic-point-of-pic}
Let $k$ be a field of characteristic $0$. Let $g>2$ be even, and let $Q:=k(\mathbf{Pic}^1_{\mathscr{C}_g/\mathscr{M}_g})$ be the function field of the moduli space of degree $1$ divisors on the universal curve over $k$. Then the section conjecture is trivially true for the base change of the universal curve $\mathscr{C}_g$ to $Q$. 
\end{theorem}
See Corollary \ref{cor:generic-point-of-pic} for a more precise statement---we show that while there is no abelian obstruction to sections, there is in fact a $2$-nilpotent obstruction. That is, we show the sequence $$1\to \pi_1^{\text{\'et}}(\mathscr{C}_{g,\overline{ Q}})/L^3\pi_1^{\text{\'et}}(\mathscr{C}_{g,\overline{Q}})\to \pi_1^{\text{\'et}}(\mathscr{C}_{g,Q})/L^3 \pi_1^{\text{\'et}}(\mathscr{C}_{g,\overline{Q}})\to \on{Gal}(\overline{Q}/ {Q})\to 1$$ does not split, where $L^i\pi_1^{\text{\'et}}(\mathscr{C}_{g,\overline{ Q}})$ denotes the lower central series of $\pi_1^{\text{\'et}}(\mathscr{C}_{g,\overline{ Q}})$. As far as we know even the following simple corollary is new:
\begin{corollary}
Let $k$ be a field of characteristic $0$ and let $g>2$ be even. The base change of the universal curve to $k(\mathbf{Pic}^1_{\mathscr{C}_g/\mathscr{M}_g})$ has no rational points.
\end{corollary}
The following consequence of our methods partially strengthens a result of Hain \cite{hain2011rational}:
\begin{corollary}\label{cor:generic-point-of-Mg}
Let $g>2$ be an integer and $k$ a field of characteristic $0$. Let $L$ be an extension $L$ of the function field $k(\mathscr{M}_g)$ of $\mathscr{M}_{g, k}$ of degree not divisible by $g-1$. Then the base change $\mathscr{C}_{g,L}$ of the universal curve $\mathscr{C}_g$ of genus $g$ to  $L$ trivially satisfies the section conjecture. Indeed, the exact sequence $$1\to \pi_1^{\text{\'et}}(\mathscr{C}_{g,\overline{ k(\mathscr{M}_g)}})^\text{ab}\to \pi_1^{\text{\'et}}(\mathscr{C}_{g,{L}})/[\pi_1^{\text{\'et}}(\mathscr{C}_{g,\overline{ k(\mathscr{M}_g)}}), \pi_1^{\text{\'et}}(\mathscr{C}_{g,\overline{ k(\mathscr{M}_g)}})]\to \on{Gal}(\overline{ k(\mathscr{M}_g)}/ { L})\to 1$$ does not split.
\end{corollary}
See Corollary \ref{cor:generic-point-o1-nonvanishing} for a stronger statement. Hain proves that the generic curve of genus $g$ over a field of characteristic zero satisfies the section conjecture if $g\geq 5$. Our proof strengthens his statement in some ways by showing that the obstruction to splitting is in some sense \emph{abelian}. Hain proves similar non-splitting results for the level covers of $\mathscr{M}_g$, about which we say nothing. While we state the results here in characteristic $0$ for simplicity, they in fact hold in any sufficiently large (in terms of $g$) finite characteristic as well, as explained in Corollary \ref{cor:generic-point-o1-nonvanishing}.

In both cases of Theorem \ref{thm:main-geometric-thm}, we proceed by first answering its topological variant, Question \ref{tropical-question}, for $(\Gamma, G)=(C_{g-1}, \mathbb{Z}/(g-1)\mathbb{Z})$ and $(\Gamma, G)=(T_g, \mathbb{Z}/2\mathbb{Z})$ in parts (1) and (2) respectively. 
\subsubsection{Topological results and the cohomology of $\mathscr{M}_g$}\label{subsubsec:topological-results}
All of these results follow from an analysis of certain torsion cohomology classes on $\mathscr{M}_g$ and $\mathbf{Pic}^1_{\mathscr{C}_g/\mathscr{M}_g}$ over a field $k$ of characteristic $0$. We describe the situation over $\mathbb{C}$ now. We study a class $$o_1\in H^2(\mathscr{M}_g, \mathbb{V}_1)$$ (which we call the Morita class, as it was previously studied by Morita \cite{Morita86}) and a class $$\widetilde{o_2}\in M(\mathbf{Pic}^1_{\mathscr{C}_g/\mathscr{M}_g},  \mathbb{V}_2),$$ where the $\mathbb{V}_i$ are certain local systems on $\mathscr{M}_g$ and $\mathbf{Pic}^1_{\mathscr{C}_g/\mathscr{M}_g}$, respectively, and  $M(\mathbf{Pic}^1_{\mathscr{C}_g/\mathscr{M}_g},  \mathbb{V}_2)$ is a certain functorial quotient of $H^2(\mathbf{Pic}^1_{\mathscr{C}_g/\mathscr{M}_g},  \mathbb{V}_2)$. These classes obstruct splittings of sequence (\ref{fundamental-exact-sequence}), in a sense which we now explain.

Let $\pi: E\to B$ be a surface bundle with fiber the orientable surface $\Sigma_g$ of genus $g$, and let $f: B\to \mathscr{M}_g$ the associated map. Let $L^i\pi_1(\Sigma_g)$ be the lower central series filtration on $\pi_1(\Sigma_g)$. Then $$f^*o_1\in H^2(B, f^*\mathbb{V}_1)$$ is zero if and only if the exact sequence $$0\to \pi_1(\Sigma_g)/L^2\pi_1(\Sigma_g)\to \pi_1(E)/L^2\pi_1(\Sigma_g)\to \pi_1(B)\to 1$$ splits. If the sequence does split, $f$ admits a lift $\tilde f: B\to \mathbf{Pic}^1_{\mathscr{C}_g/\mathscr{M}_g}$. And if the sequence $$1\to \pi_1(\Sigma_g)/L^3\pi_1(\Sigma_g)\to \pi_1(E)/L^3\pi_1(\Sigma_g)\to \pi_1(B)\to 1$$ splits, then if ${\tilde f}^*\widetilde{o_2}$ vanishes (though the converse need not hold).

Morita shows \cite[Corollary 3, Proposition 4]{Morita86} that $o_1$ is non-zero for $B=\mathscr{M}_g$ with $g \geq 9$. We are able to extend his non-vanishing result to all $g \geq 3$ by constructing certain surface bundles over the two-torus $T^2$, as we now explain.
We construct maps $f_g: T^2\to \mathscr{M}_g$ such that the pullbacks $f_g^*o_1$  are non-trivial. Hence in particular the associated surface bundles have no section. Similarly, for all $g\geq 2$, we construct maps $h_g: T^2\to \mathbf{Pic}^1_{\mathscr{C}_g/\mathscr{M}_g}$ such that $h_g^*\widetilde{o_2}$ has order exactly $2$. Hence again the associated surface bundles have no section. These constructions answer a question of Hillman \cite[end of Section 10]{hillman2015sections}, who asked if there are surface bundles over tori with (1) hyperbolic fiber and (2) no continuous section, for all $g\geq 2$. See Corollary \ref{cor:main-corollary-o1-pinwheel} and Theorem \ref{thm:tree-o2-nonvanishing} for these topological constructions.

The construction of these bundles in fact gives substantially finer information about the behavior of the classes $o_1, \widetilde{o_2}$, near the boundary of the Deligne-Mumford compactification of $\mathscr{M}_g$ and the Caporaso compactification of $\mathbf{Pic}^1_{\mathscr{C}_g/\mathscr{M}_g}$ \cite{caporaso1994compactification}. Let $\Gamma_1=C_{g-1}, \Gamma_2=T_g$ be the stable graphs described in Section \ref{subsubsec:geometric-results}, and let $E_{\Gamma_i}$ be the exceptional divisor of the blowup of $\overline{\mathscr{M}_g}$ at the stratum corresponding to $\Gamma_i$. We show the following (see Corollaries \ref{cor:o1-gysin-corollary} and \ref{cor:o2-gysin-nonvanishing} for the precise statements):
\begin{theorem}
The Gysin image of $o_1$ is non-zero in the cohomology of a Zariski-open subset $E_{\Gamma_1}^\circ$ of $E_{\Gamma_1}$, and the Gysin image of $\widetilde{o}_2$ is non-zero in (an appropriate quotient of) the cohomology of a Zariski-open subset of the preimage of $E_{\Gamma_2}$ in a blowup of the Caporaso compactification of $\mathbf{Pic}^1_{\mathscr{C}_g/\mathscr{M}_g}$. 
\end{theorem} 
The non-vanishing of these Gysin images is crucial for our arithmetic applications, and for the proof of Theorem \ref{thm:main-geometric-thm}. In particular it more or less immediately shows that these classes do not vanish at the generic points of the respective moduli spaces on which they live.
\subsubsection{Arithmetic results}
Our main arithmetic results are arithmetic analogues of those described in Section \ref{subsubsec:topological-results}. We define classes $o_{1,\text{\'et}}, \widetilde{o_{2, \text{\'et}}}$ in the \'etale cohomology of $\mathscr{M}_g$, $\mathbf{Pic}^1_{\mathscr{C}_g/\mathscr{M}_g}$,  with coefficients in certain $\widehat{\mathbb{Z}}$-local systems over any field of characteristic $0$. These classes obstruct $\pi_1$-sections (and hence rational points!) in a sense which we now explain. Let $C/k$ be a smooth projective curve over a field $k$ of characteristic $0$, and let $[C]:\on{Spec}(k)\to\mathscr{M}_g$ be the associated map. Then the Galois cohomology class $[C]^*o_{1, \text{\'et}}$ is zero if and only if the exact sequence $$0\to \pi_1^{\text{\'et}}(C_{\bar k})/L^2\pi_1^{\text{\'et}}(C_{\bar k})\to \pi_1^{\text{\'et}}(C)/L^2\pi_1^{\text{\'et}}(C_{\bar k})\to \on{Gal}(\bar k/k)\to 1$$ splits. This invariant is closely related to the \emph{period} of $C$, as we describe in Section \ref{subsubsec:Pic-construction-o1} and Remark~\ref{R:periodindex}. Similarly if $[\tilde C]: \on{Spec}(k)\to \mathbf{Pic}^1_{\mathscr{C}_g/\mathscr{M}_g}$ is a morphism, the non-vanishing of $[\tilde C]^*\widetilde{o_{2, \text{\'et}}}$ implies that the sequence $$1\to \pi_1^{\text{\'et}}(C_{\bar k})/L^3\pi_1^{\text{\'et}}(C_{\bar k})\to \pi_1^{\text{\'et}}(C)/L^3\pi_1^{\text{\'et}}(C_{\bar k})\to \on{Gal}(\bar k/k)\to 1$$
does not split (though the converse need not hold).

We show that there are many examples of curves $C, C'$ for which $[C]^*o_{1,{\text{\'et}}}$, $[\tilde{C'}]^*\widetilde{o_{2, \text{\'et}}}$ are non-vanishing. For instance, the geometric examples discussed in Section \ref{subsubsec:geometric-results} (e.g.~the cases in which we prove the tropical section conjecture) have this property. But we also use a Chebotarev argument to show the existence of many curves over $p$-adic fields (and hence number fields) for which these classes do not vanish. In particular, these curves trivially satisfy the section conjecture. For example, we show:
\begin{theorem}
Let $\Gamma$ be a graph as in Theorem \ref{thm:main-geometric-thm}. Then there exists a Zariski-dense set $S$ of closed points of $Z_{\Gamma, \mathbb{Z}}$ such that for each $s\in S$, there exists a $\on{Frac}(W(\kappa(s)))$-point $s'$ of $\mathscr{M}_g$ specializing to $s$, such that the section conjecture is true for the curve $\mathscr{C}_{g, s'}$ (that is, the $\on{Frac}(W(\kappa(s)))$-curve corresponding to $s'$).
\end{theorem}
Here $\kappa(s)$ is the residue field of a closed point $s$ and $W(\kappa(s))$ is the ring of Witt vectors of $\kappa(s)$. See Theorems \ref{thm:examples-o1-p-adic} and \ref{thm:example-o2-p-adic} for a stronger and more general statement. In principle our method gives us quantitative control (in the sense of the Chebotarev density theorem) over the Dirichlet density of the set $S$ in the theorem.

We believe that our examples of curves $C$ exhibiting the non-vanishing of $\widetilde{o_{2, \text{\'et}}}$ are particularly interesting, as they show that this class is a genuine (non-abelian) obstruction to the existence of rational points. Unfortunately all of our examples are of a local nature. It would be interesting to find an example of a curve over a number field which has points everywhere locally, but which exhibits the non-vanishing of $\widetilde{o_{2, \text{\'et}}}$. As far as we are aware there is no known example of a curve over a number field which has points everywhere locally and is known to satisfy the section conjecture in a genuinely non-abelian way.
\subsection{Relation to previous work}
The main precursor to our geometric work is the paper \cite{hain2011rational}, which, as remarked earlier, proves a form of Corollary \ref{cor:generic-point-of-Mg} for $g\geq 5$ over fields of characteristic $0$.  See also related work of Watanabe in positive characteristic \cite{watanabe}.

The methods we use to prove this Corollary are also closely related to work on the Franchetta conjecture, in particular \cite{schroer2003strong}. Indeed, the class $o_{1, \text{\'et}}$ is the image of the class $[\mathbf{Pic}^1_{\mathscr{C}_g/\mathscr{M}_g}]\in H^1(\mathscr{M}_g, \mathbf{Pic}^0_{\mathscr{C}_g/\mathscr{M}_g})$ under the Kummer map (see section \ref{subsubsec:Pic-construction-o1} and Remark~\ref{R:periodindex}), and so it is closely related to the period of a (relative) curve. This class has been well-studied in the complex-analytic setting (see e.g.~\cite{Morita86, Morita89}). The class $\widetilde{o_2}$ has not, to our knowledge, been studied before, but it is related to unpublished work of Ellenberg \cite{ellenberg2nilpotent} and work of Wickelgren \cite{wickelgren2009lower}.

Our arithmetic examples of curves for which the section conjecture holds also seem related to those constructed by other authors, though we do not know how to make this precise. The examples constructed in \cite{stix2010period} and \cite[6.2]{stix2011brauer} have the same reduction type as the curves constructed in special cases of Theorems \ref{thm:examples-o1-p-adic} and \ref{thm:example-o2-p-adic}.  It would be interesting to understand how $o_{1, \text{\'et}}$ and especially $\widetilde{o_{2, \text{\'et}}}$ is related to the inequality for period and index for curves exploited by Stix. Harari and Szamuely \cite{harari2009galois} study curves for which the abelianized fundamental exact sequence does not split---implicitly, analyzing $o_{1,\text{\'et}}$---and construct examples where the obstruction to splitting is fundamentally global, rather than local.  It would be extremely interesting to construct such an example with $\widetilde{o_{2, \text{\'et}}}$.

\subsection{Structure of the paper}
In Section \ref{sec:moduli-preliminaries} we recall various preliminaries and notation for the moduli spaces we use, as well as their relevant compactifications, coarse spaces, and boundary strata. Primarily, we use the Deligne-Mumford compactification $\overline{\mathscr{M}_g}$ of the moduli space of curves $\mathscr{M}_g$, as well as the moduli space of degree one divisor classes on the universal curve, $\mathbf{Pic}^1_{\mathscr{C}_g/\mathscr{M}_g}$, and its Caporaso compactification \cite{caporaso1994compactification}. In Section \ref{sec:cohomological-preliminaries}, we recall the various versions of the Gysin maps we will use (in group-theoretic, topological, Galois-cohomological, and \'etale-cohomological contexts) and prove a useful variant of the Chebotarev density theorem (Theorem \ref{modified-chebotarev}) which may be of independent interest. In Section \ref{sec:primary-and-secondary-morita}, we define the classes $o_1, \widetilde{o_2}$, which are key to our analysis, and study their basic properties, in both topological and \'etale cohomological settings. Because the non-abelian group cohomology machinery required for the construction of $\widetilde{o_2}$ is quite involved, we have banished this construction and the ensuing calculations to the appendix, Section \ref{sec:Non-abelian-cohomology}, to which we refer frequently throughout the paper. 

In Section \ref{SMorSurfGp}, we begin proving the main results of the paper. We construct various surface bundles over surfaces (with no sections) exhibiting the non-vanishing of the classes $o_1, \widetilde{o_2}$. From these constructions, we deduce the non-vanishing of certain Gysin images of the classes $o_1, \widetilde{o_2}$ in the cohomology of boundary components of $\mathscr{M}_g$ and  $\mathbf{Pic}^1_{\mathscr{C}_g/\mathscr{M}_g}$ (over the complex numbers). Again, we banish certain involved cocycle computations with surface groups to the appendix, Section \ref{sec:Non-abelian-cohomology}. In Section \ref{SSecConjApp}, we use standard comparison results to pass from these topological computations to results in \'etale cohomology, and then in Galois cohomology. This is where we provide geometric examples of curves for which the section conjecture is trivially true and prove many special cases of Conjecture \ref{conj:tropical-section-conjecture}. We then use our modified Chebotarev density theorem (Theorem \ref{modified-chebotarev}) to show the existence of arithmetic examples over $p$-adic fields, and hence over number fields. 
\subsection{Acknowledgments}
Li is funded by the Simons Collaboration on Arithmetic Geometry, Number Theory and Computation.
Litt is supported by NSF Grant DMS-2001196. This paper benefited from private correspondence with Jordan Ellenberg, Richard Hain, and Qixiao Ma. We would also like to thank the referee for their extremely useful report. On behalf of all authors, the corresponding author states that there is no conflict of interest. Data sharing not applicable to this article as no datasets were generated or analysed during the current study.

We would like to thank Jordan Ellenberg and Alexander Smith for pointing out a mistake in an earlier version of Theorem 3.1.1 with a counterexample from the work of Wang on Grunwald--Wang theorem.
\section{The moduli space of curves and boundary strata}\label{sec:moduli-preliminaries}
We begin by indicating our conventions regarding the various moduli spaces we will use (primarily the moduli space of curves $\mathscr{M}_g$ and its Deligne-Mumford compactification $\overline{\mathscr{M}_g}$), and recalling the combinatorics of their boundary strata.
\subsection{$\mathscr{M}_g$ and its boundary}\label{sec:Mg-preliminaries}
Let $g>1$ be an integer and let $\mathscr{M}_g$ be the moduli space of algebraic curves; recall that $\mathscr{M}_g$ is a smooth Deligne-Mumford stack over $\on{Spec}(\mathbb{Z})$. Let $\overline{\mathscr{M}_{g}}$ be the Deligne-Mumford compactification of $\mathscr{M}_g$; recall that it is a smooth and proper Deligne-Mumford stack over $\on{Spec}(\mathbb{Z})$. When there is no chance of confusion, we will also use the notation $\mathscr{M}_g$ to denote the complex-analytic moduli stack of genus $g$ Riemann surfaces; otherwise we will denote it by $\mathscr{M}_{g, \mathbb{C}}^{\text{an}}$ (and similarly with the analytification of its Deligne-Mumford compactification).

We briefly recall the combinatorics of the boundary strata of $\overline{\mathscr{M}_g}$, and we interpret the inertia about boundary components group-theoretically, in terms of Dehn twists.

Let $\Gamma$ be a stable graph of genus $g$, i.e.,  ~a collection $V$ of vertices and $E$ of (undirected) edges between pairs of vertices and a labeling $h: V\to \mathbb{Z}_{\geq 0}$ such that:
\begin{enumerate}
    \item The Euler characteristic $$\chi(\Gamma)+\sum_{v\in V} (2-2h(v))=2-2g,$$
    \item If $h(v)=0$, then the degree of $v$ is at least 3, and
    \item If $h(v)=1$ then the degree of $v$ is at least $1$, and
    \item $\Gamma$ is connected.
\end{enumerate}
Given a vertex $v$ of $\Gamma$, we say that $h(v)$ is its \emph{genus}.

We let $$Z_\Gamma\subset \overline{\mathscr{M}_g}$$ be the locally closed substack parametrizing stable curves of \emph{type $\Gamma$}, i.e.,~stable curves whose dual graph is isomorphic to $\Gamma$. The codimension of $\overline{Z_\Gamma}$ in $\overline{\mathscr{M}_g}$ is equal to the number of edges of $\Gamma$. We say that $\Gamma$ \emph{specializes} to $\Gamma'$ if $Z_{\Gamma'}$ is contained the closure of $Z_{\Gamma}$. Combinatorially this means that $\Gamma$ can be obtained from $\Gamma'$ by contracting edges and redistributing weights accordingly.

We briefly discuss the topology of a neighborhood of $Z_\Gamma$ in terms of the graph $\Gamma$. We first (non-functorially) associate a surface $\Sigma_\Gamma$ to $\Gamma$, with a marked loop $\gamma_e$ for each edge of $\Gamma$, as illustrated in Figure \ref{fig:graph-surface}. Explicitly, to each vertex $v$ of genus $g$ and degree $d$, associate an oriented surface $\Sigma_v$ of genus $g$, with $\deg(v)$ distinguished disjoint closed discs $\iota_{v, e}:\Delta\to \Sigma_v$ for each edge $e$ adjacent to $v$. Then $\Sigma_\Gamma$ is obtained by gluing $\Sigma_{v_i}\setminus\bigcup_{e\in \text{nbd}(v_i)}\iota_{v_i, e}(\Delta^\circ)$ to $\Sigma_{v_j}\setminus \bigcup_{e\in \text{nbd}(v_j)}\iota_{v_j, e}(\Delta^\circ)$ along the circles $\iota_{v_i, e}(\delta\Delta), \iota_{v_j, e}(\delta\Delta)$ if $e$ is an edge between $v_i$ and $v_j$. That is, $$\Sigma_\Gamma=\on{colim}\left(\bigsqcup_{e\in E}\delta\Delta \rightrightarrows \bigsqcup_{v\in V} \Sigma_v\setminus\bigcup_{e\in \text{nbd}(v)}\iota_{v, e}(\Delta^\circ)\right),$$ where $V,E$ are the set of vertices and edges of $\Gamma$, respectively. The boundaries of the discs $\iota_{v,e}(\Delta)$ are the distinguished curves $\gamma_e$ in $\Sigma_\Gamma$. We say that $\Gamma$ is the \emph{dual graph} of the marked surface $\Sigma_\Gamma$.
\begin{lemma}\label{lem:dehn-twist-inertia}
Let $E_\Gamma\subset \on{Bl}_{\overline{Z_{\Gamma}}} \overline{\mathscr{M}_g}$ be the exceptional divisor. Then any inertia subgroup $I_\Gamma \subset \pi_1(\mathscr{M}_g)=\on{Mod}(g)$ 
corresponding to $E_\Gamma$ is conjugate to the group generated by the Dehn multitwist about the curves $\gamma_e$ (corresponding to the edges of $\Gamma$).
\end{lemma}
Here $\on{Mod}(g)$ is the mapping class group of a genus $g$ surface $\Sigma_g$.
\begin{proof}
The stratum $\overline{Z_{\Gamma}} \subset \overline{\mathcal{M}_g}$ is (locally) the intersection of the boundary divisors of $\overline{\mathscr{M}_g}$ which contain it, with normal crossings. The monodromy about these boundary divisors is worked out in \cite[Theorem 2.2]{AMO}. The lemma now follows from a local computation (of the monodromy around the blowup of an intersection of divisors with normal crossings), contained in, for example \cite{MO-speyer}. 
\end{proof}
We will at some points be forced to work with the coarse spaces $M_g$ of $\mathscr{M}_g$ and $\overline{M_g}$ of $\overline{\mathscr{M}_g}$, and the sublocus $M_g^0$ of $M_g$ and $\overline{M_g}^0$ of $\overline{M}_g$ consisting of curves with trivial automorphism group.
\subsection{$\mathbf{Pic}^1_{\mathscr{C}_g/\mathscr{M}_g}$ and its boundary} \label{subsec:pic-preliminaries} Let $\mathscr{C}_g$ be the universal curve over $\mathscr{M}_g$. We will denote by $\mathbf{Pic}^d_{\mathscr{C}_g/\mathscr{M}_g}$ the $\mathbb{G}_m$-rigidification of the moduli stack whose $T$-points are families of smooth proper genus $g$ curves over $T$ with a line bundle of relative degree $d$ (see e.g.,~\cite[Section 2]{melo2014picard} for a precise definition). We will also use the space $P_{d,g}$ which coarsely represents the Picard functor of degree $d$ line bundles over $M_g^0$.  Caporaso constructs \cite{caporaso1994compactification} a Cohen-Macaulay compactification $\overline{P_{d,g}}$ of $P_{d,g}$, equipped with a proper map $\overline{P_{d,g}}\to \overline{M}_g$ which will also be used. 

Ebert and Randal-Williams \cite{ebert2010stable} also consider analytic analogues of these moduli stacks. Melo and Viviani \cite{melo2014picard} construct a map from these analytic stacks to the analytifications of those described in the paragraph above. Both Ebert--Randal-Williams and Melo--Viviani speculate that this map is an equivalence, but do not check it explicitly. That said,  by considering the fibers over points of $\mathscr{M}_g$, it is easy to see that this map induces an isomorphism on fundamental groups. Ultimately all of the work we do here is group-theoretic in nature, so one might equally well work with the stacks considered by Ebert and Randal-Williams. For us, the fundamental relevant observation about these analytic stacks is that their fundamental groups are given by $$\pi_1(\mathbf{Pic}^1_{\mathscr{C}_g/\mathscr{M}_g})=\pi_1(\mathscr{C}_g)/L^2\pi_1(\Sigma_g).$$
See e.g.~\cite[Section 2, Proposition 2.1, and Theorem 4.6]{ebert2010stable} for a discussion of the homotopy type of these stacks.

At various points in the text we will phrase things in terms of the (somewhat complicated) stacks $\mathbf{Pic}^d_{\mathscr{C}_g/\mathscr{M}_g}$; for the reader uncomfortable with stacks, we indicate now that this usage only leads to cleaner statements. Indeed, all of our main results could be formulated entirely in terms of the schemes $M_g^0$ and $P_{d,g}$ and their (scheme-theoretic) compactifications.

We require the following fact from Caporaso \cite[Section 7.2 and footnote at the bottom of page 594]{caporaso1994compactification}:
\begin{proposition}\label{prop:irreducible-fibers}
Let $\Gamma$ be a stable tree and $Z_\Gamma$ the corresponding stratum of the boundary of $\overline{M_g}$ (the coarse space of $\overline{\mathscr{M}_g}$). Then the fibers of the canonical projection $\overline{P_{d,g}}\to \overline{M_g}$ over points of $Z_\Gamma$ are irreducible.
\end{proposition}

We will also make use of certain blowups of $\overline{P_{d,g}}$; we will require the fact that they are also Cohen-Macaulay. For this purpose we record the following:
\begin{lemma}\label{lem:cohen-macaulay}
Let $X$ be a Cohen-Macaulay scheme and $V\subset X$ an lci subscheme. Then $\on{Bl}_V X$ is Cohen-Macaulay.
\end{lemma}
\begin{proof}
This is \cite[Proposition 5.5(1)]{kovacs2017rational}.
\end{proof}
Let $D$ be a smooth connected divisor in a smooth complex variety $X$. By a \emph{deleted neighborhood} of $D$ in $X$ we will mean the complement of $D$ in an $\epsilon$-neighborhood of $D$ in $X$ (in the Euclidean topology) of which $D$ is a deformation retract; this is a punctured disc bundle over $D$. 

We will also require the following in Section \ref{M_g-consequences}, during our analysis of Gysin images of certain cohomology classes on $\mathscr{M}_g$, $P_{1,g}$:
\begin{lemma}\label{lem:pi1-factoring}
Let $X$ be a smooth complex variety and let $D\subset X$ be smooth and connected of codimension one in $X$. Let $\widetilde{D}$ be a deleted neighborhood of $D$ in $X$ and let $$\pi: \widetilde{D}\to D$$ the associated punctured disc bundle. Let $x\in \widetilde{D}$ be a point and $y=\pi(x)$. Suppose we have $a\in \pi_1(D, y), b\in \pi_1(X, x)$ such that $b$ is a generator of the local inertia around $D$ (i.e.~it generates the kernel of the map $\pi_1(\widetilde{D}, x)\to \pi_1(D, y)$ induced by $\pi$).

Then for any Zariski-open $U\subset D$ containing $y$, there exist commuting elements $a', b'\in \pi_1(\pi^{-1}(U), x)$ such that
\begin{enumerate}
\item $\pi_*(a')\in \pi_1(U, y)$ maps to $a$ in $\pi_1(D, y)$, and 
\item $b'$ maps to $b$ in $\pi_1(\widetilde{D})$ 
\end{enumerate}
\end{lemma}
\begin{proof} For any inclusion $U \subset D$ of a Zariski-open set containing $y$, the natural map $\pi_1(U) \rightarrow \pi_1(D)$ is surjective (as $D$ is normal). Hence we may may lift $a$ to $\pi_1(U, y)$ and then to $a'\in \pi_1(\pi^{-1}(U), x)$. Now choosing $b'$ to be any lift of $b$ contained in the local inertia group of $U$ in $\pi^{-1}(U)$ gives the result.
\end{proof}
An essentially identical proof gives:
\begin{lemma}\label{lem:galois-factoring}
Let $X$ be a normal variety over an algebraically closed field $k$ of characteristic zero and $D\subset X$ normal of pure codimension one in $X$. Let $\bar y:\on{Spec}(k)\to  D$ be a $k$-point such that $X,D$ are both non-singular at $\bar y$ (such a point exists by normality). Let $a\in  \pi_1^{\text{\'et}}(D, \bar y)$ be any element. Let $\Gamma_D\subset \on{Gal}(\overline{k(X)}/k(X))$ be a decomposition group associated to $D$ and let $b$ be a generator of the inertia of $\Gamma_D$.

Then there exists $a'\in \Gamma_D$ commuting with $b$.
\end{lemma}
\begin{proof}
As $\Gamma_D$ surjects onto $\pi_1^{\text{\'et}}(D, \bar y)$ by normality, we may let $a'$ be any lift of $a$; it automatically commutes with $b$ as the inertia subgroup of $\Gamma_D$ is central.
\end{proof}
\section{Cohomological preliminaries}\label{sec:cohomological-preliminaries}
\subsection{A form of the Chebotarev density theorem}
One of the arithmetic goals of this paper is to provide an abundance of points of $\mathscr{M}_g$ where certain cohomology classes do not vanish. These classes are in cohomological degree $2$, where we do not know how to directly prove the existence of such points. However, the following variant of the Chebotarev density theorem gives such points for classes in cohomological degree $1$, and will be crucial for our applications:

\begin{theorem}\label{modified-chebotarev}
    Let $X$ be a finite-type, integral, normal $\mathbb{Z}$-scheme of dimension at least
one, and let $\mathscr{F}$ be a locally constant constructible sheaf of abelian groups on $X_{\text{\'et}}$. Let $\bar x \in X$ be a geometric point. For $\gamma\in\pi_1^{\text{\'et}}(X, \bar x)$, let $f_\gamma: \widehat{\mathbb{Z}}\to \pi_1^{\text{\'et}}(X, \bar x)$ be the map sending $1$ to $\gamma$. For  $\alpha \in H^1(X_{\text{\'et}}, \mathscr{F})$, suppose there exists $\gamma$ such that 
$$f_\gamma^*\alpha \in H^1(\widehat{\mathbb{Z}}, \mathscr{F})$$ 
is non-zero. Then the set of closed points $x$ in $X$ such that $\alpha|_x$ is nonzero is Zariski-dense.
\end{theorem}
\begin{proof}
    For ease of notation write $\Gamma=\pi_1^{\text{\'et}}(X, \bar x)$  and $A=\mathscr{F}_{\bar x}$. By the continuity of the $\Gamma$-action on $A$, there exists an open neighborhood $U$, $\gamma\in U\subset \Gamma$, such that for $\gamma'\in U$, the action of $\gamma'$ on $A$ is the same as the  action of $\gamma$ on $A$, giving a canonical identification of $f^*_\gamma A$ with $f^*_{\gamma'}A$ as $\widehat{\mathbb{Z}}$-modules. Let $$\Phi: U\to H^1(\widehat{\mathbb{Z}}, f_\gamma^*A)$$ be the map $$\gamma'\mapsto f_{\gamma'}(\alpha).$$ It is a matter of unwinding definitions to check that $\Phi$ is continuous. In particular there exists an open neighborhood $V\subset U$ such that for all $\gamma'\in V$, $\Phi(\gamma')$ is non-zero. Let $W$ be the orbit of $V$ under conjugation by $\Gamma$; we also have that for all $\eta\in W$, $f_\eta^*\alpha$ is nonzero.

    Now the Chebotarev density theorem \cite[Theorem 9.11]{serre_lectures_2016} yields that the conjugacy classes of Frobenii at a Zariski-closed subset of points of $X$ lie in $W$, completing the proof.
\end{proof}

\begin{remark}
Evidently there is no analogue of Theorem \ref{modified-chebotarev}  for classes $\alpha\in H^i(X_{\text{\'et}}, \mathscr{F})$, $i>1$, as the cohomological dimension of a finite field is $1$. One might ask if there is an analogous result for classes in higher cohomological degree on varieties over fields with higher cohomological dimension. We are unaware of any such result, with the exception of \cite[Theorem 2.5]{fein1992brauer} (which inspired in part the arithmetic results of this paper).
\end{remark}
\subsection{Gysin sequences and comparison results}\label{subsec:gysin}
Having provided in Theorem \ref{modified-chebotarev} a mechanism for finding non-zero specializations of a class in cohomological degree $1$, we now describe our mechanism for shifting the classes we will study---namely $o_1, \widetilde{o_2}$, and their \'etale-cohomological variants---from cohomological degree $2$ to cohomological degree $1$. We will use various versions of the Gysin map, in étale, singular, group, and Galois cohomology for this purpose; we recall these maps and the relationships between them now.
\subsubsection{Gysin sequences in topology}
Let $\pi: E\to B$ be a circle bundle, and let $\mathscr{F}$ be a locally constant sheaf of abelian groups on $E$.
\begin{proposition}[Thom-Gysin sequence]\label{gysin-topology}
There is a long exact sequence $$\cdots \to H^i(B, \pi_*\mathscr{F})\to H^i(E, \mathscr{F})\to H^{i-1}(B, R^1\pi_*\mathscr{F})\to H^{i+1}(B, \pi_*\mathscr{F})\to H^{i+1}(E, \mathscr{F})\to \cdots$$
\end{proposition}
\begin{proof}
This is immediate from the Leray spectral sequence associated to $\pi$.
\end{proof}
\begin{corollary}[Group-theoretic Thom-Gysin sequence]\label{gysin-group-theory}
Let $$G\to H$$ be a surjection of groups with kernel $K$ isomorphic to $\mathbb{Z}$. Then for any $G$-module $A$, there is a long exact sequence $$\cdots\to H^i(H, A^K) \to H^i(G, A)\to H^{i-1}(H, A_K)\to H^{i+1}(H, A^K)\to H^i(G, A)\to \cdots$$
\end{corollary}
\begin{proof}
Apply Proposition \ref{gysin-topology} to the fibration $K(G, 1)\to K(H,1)$.
\end{proof}
We will typically apply Proposition \ref{gysin-topology} in the following setting. Let $X$ be a smooth complex variety and $D\subset X$ a smooth irreducible divisor. Then a deleted neighborhood $\tilde D$ of $D$ is (homotopic to) a circle bundle over $D$ via a map $$\pi: \tilde D\to D.$$ Hence for any $i\geq 0$ and any locally constant sheaf of Abelian groups $\mathscr{F}$ on $U=X\setminus D$, there is a natural map $$g_D: H^i(U, \mathscr{F})\to H^i(\tilde D, \mathscr{F}|_{\tilde D})\to H^{i-1}(D, R^1\pi_*\mathscr{F}|_{\tilde D}),$$ which we refer to as a Gysin map.

One may alternately view the map above as follows. Let $\iota: U\to X$ be the natural inclusion, and $j: D\to X$ the inclusion of its complement. Then the Leray spectral sequence for $\iota$ yields $$E_2^{pq}=H^p(X, R^q\iota_*\mathscr{F})\implies H^{p+q}(U, \mathscr{F}).$$ A local computation (see e.g. \cite[Theorem 16.11]{milne1998lectures} and the surrounding references) yields a canonical isomorphism $$j^*R^1\iota_*\mathscr{F}\simeq R^1\pi_*\mathscr{F}|_{\tilde D},$$ and under this identification the map $g_D$ agrees with the map $$H^2(U, \mathscr{F})\to H^1(X, R^1\iota_*\mathscr{F})\simeq H^1(D, j^*R^1\iota_*\mathscr{F})\simeq H^1(D, R^1\pi_*\mathscr{F}|_{\tilde D}),$$ where the first map arises from the Leray spectral sequence for $\iota$.
\subsubsection{Gysin sequences in \'etale cohomology} There is an analogous story in \'etale cohomology. Let $R$ be a complete discrete valuation ring with residue field $k$ and fraction field $K$. Let $G_K$ be the absolute Galois group of $K$ and let $I\subset G_K$ be the inertia subgroup. Let $A$ be a finite discrete $G_K$-module of order prime to $\on{char}(k)$.
\begin{proposition}\label{gysin-etale}
There is a long exact sequence $$\cdots \to H^i(k, A^I)\to H^i(K, A)\to H^{i-1}(k, A(-1)_I)\to H^{i+1}(k, A^I)\to H^{i+1}(K, A)\to \cdots$$
\end{proposition}
\begin{proof}
This follows from \cite[Lemma 2.18]{Milne06}.
\end{proof}
The Gysin map $H^i(K, A)\to H^{i-1}(k, A(-1)_I)$ globalizes as follows. Let $X$ be a regular scheme  and let $D\subset X$ be a regular subscheme of codimension one. Let $U=X\setminus D$ and let $\mathscr{F}$ be an lcc sheaf of abelian groups on $U$, tame along $D$, whose order is invertible on $X$. Let $\iota: U\to X$ be the natural inclusion. Then a local computation shows that $R^i\iota_*\mathscr{F}=0$ for $i\not=0,1$, and that $R^1\iota_*\mathscr{F}$ is supported on $D$; hence the hypercohomology spectral sequence for $R\iota_*\mathscr{F}$ becomes a long exact sequence $$\cdots \to H^i(X, \iota_*\mathscr{F})\to H^i(U, \mathscr{F})\to H^{i-1}(D, R^1\iota_*\mathscr{F}|_D)\to H^{i+1}(X, \iota_*\mathscr{F})\to H^i(U, \mathscr{F})\to\cdots.$$

As before, let $X$ be a smooth complex variety and $D\subset X$ is a divisor, and let $$\pi: \tilde D\to D$$ be the projection from a deleted neighborhood of $D$ to $D$ (recall the definition of a deleted neighborhood above Lemma \ref{lem:pi1-factoring}). The complete local ring $\widehat{\mathscr{O}_{X, D}}$ has residue field $\mathbb{C}(D)$ and fraction field $\widehat{\mathbb{C}(X)}$. Let $I\subset G_{\widehat{\mathbb{C}(X)}}$ be the inertia subgroup, and let $U=X\setminus D$. Let $\iota: U\hookrightarrow X$ be the inclusion of $U$ into $X$ and $j: D\to X$ the inclusion of the complement. We record the evident compatibilities between the various Gysin maps described above in this setting in the following proposition; we only sketch the proof, which is a matter of unwinding the objects in question. Unfortunately we do not know a precise reference.
\begin{proposition}\label{EtaleGysinCompleted}
Let $\mathscr{F}$ be a locally constant sheaf of finite Abelian groups on $U(\mathbb{C})^{\text{an}}$, and let $\mathscr{F}^{\text{\'et}}$ be the associated sheaf on $U_{\text{\'et}}$. Then the diagram 
$$\xymatrix{
H^2(U(\mathbb{C})^{\text{an}}, \mathscr{F}) \ar[r]^-{g_D} \ar[d]^{\sim} & H^{1}(D(\mathbb{C})^{\text{an}}, R^1\pi_*(\mathscr{F}|_{\tilde D}))\ar[d]^{\sim} \\
H^2(U_{\text{\'et}}, \mathscr{F}^{\text{\'et}}) \ar[r]  \ar[d] & H^{1}(D_{\text{\'et}}, j^*R^1\iota_*\mathscr{F}^{\text{\'et}}) \ar[d]\\
H^2(\widehat{\mathbb{C}(X)}, \mathscr{F}^{\text{\'et}}|_{\widehat{\mathbb{C}(X)}}) \ar[r]  & H^{1}(\mathbb{C}(D), (\mathscr{F}^{\text{\'et}}|_{\widehat{\mathbb{C}(X)}}(-1))_I)
}$$
commutes.
\end{proposition}
\begin{proof}[Proof sketch]
The primary issue is to observe that there is a canonical isomorphism $$(R^1\pi_*(\mathscr{F}|_{\tilde D}))^{\text{\'et}}\overset{\sim}{\to}j^*R^1\iota_*\mathscr{F}^{\text{\'et}};$$ once this isomorphism is estabilished, the commutativity of the top square follows from the functoriality properties of the comparison between \'etale and singular cohomology; the commutativity of the bottom square follows by the compatibility between \'etale and Galois cohomology and explicit computation of the restriction of $j^*R^1\iota_*\mathscr{F}^{\text{\'et}}$ to the generic point of $D$. 

To construct the desired isomorphism of sheaves, it suffices to do so in the topological setting, i.e.~we wish to construct an isomorphism $$R^1\pi_*(\mathscr{F}|_{\tilde D})\overset{\sim}{\to}j^*R^1\iota_*\mathscr{F}.$$ Thus it suffices to construct a canonical such isomorphism locally. So we consider the case where $\Delta$ is the open unit ball in $\mathbb{C}^n$, $D\subset \Delta$ is the vanishing locus of $x_0$, and $\mathscr{F}$ is a locally constant sheaf on $\Delta\setminus D$. 

Let $1\in \Delta\setminus D$ be a point. As $\Delta\setminus D$ has fundamental group $\mathbb{Z}$, $\mathscr{F}$ is defined by an automorphism $\tau: \mathscr{F}_1\to \mathscr{F}_1$ (given by the monodromy about $D$), and direct computations shows that both $R^1\pi_*(\mathscr{F}|_{\tilde D}), j^*R^1\iota_*\mathscr{F}$ are canonically isomorphic to the constant sheaf with value $\on{coker}(\tau-\on{id})$, completing the proof.
\end{proof}

\section{The primary and secondary Morita classes}\label{sec:primary-and-secondary-morita}
We now construct the classes $o_1, o_{1, \text{\'et}}$ described in the introduction, which we refer to as the primary Morita classes. Recall that these classes will obstruct splittings of an \emph{abelianized} version of the fundamental exact sequence (\ref{fundamental-exact-sequence}). Explicitly, in the topological setting, $o_1$ will obstruct the splitting of sequence (\ref{seq:abelianized-Birman}) below, and $o_{1, \text{\'et}}$ will obstruct the splitting of its profinite analogue.
\subsection{The class $o_1$ in the topological setting} We will give two constructions of the Morita class in the topological setting.
\subsubsection{A group-theoretic construction} Let $g>1$ and let $\Sigma_g$ be a compact orientable surface of genus $g$,  and let $\text{Mod}(g)$ be the mapping class group of $\Sigma_g$. Let $\text{Mod}(g, 1)$ be the mapping class group of a pointed genus $g$ surface. The Birman exact sequence $$1\to \pi_1(\Sigma_g)\to \text{Mod}(g, 1)\to \text{Mod}(g)\to 1$$ is the  exact sequence of fundamental groups associated to the fibration $$\mathscr{C}_g\to \mathscr{M}_g,$$ where $\mathscr{M}_{g}$ is the complex-analytic moduli stack of $g$ curves, and $\mathscr{C}_g$ is the universal curve over $\mathscr{M}_g$.

Pushing out along the Hurewicz map $$\Hur: \pi_1(\Sigma_g)\to H_1(\Sigma_g, \mathbb{Z}),$$ we obtain a short exact sequence \begin{equation}\label{seq:abelianized-Birman} 0\to H_1(\Sigma_g, \mathbb{Z})\to \text{Mod}(g,1)/[\pi_1(\Sigma_g), \pi_1(\Sigma_g)]\to \text{Mod}(g)\to 1.\end{equation}
\begin{definition}[Topological Morita class]
Let $$o_1\in H^2(\text{Mod}(g), H_1(\Sigma_g,\mathbb{Z}))$$ be the cohomology class associated to this extension.
\end{definition}
We may equivalently view $o_1$ as an element of $H^2(\mathscr{M}_g, \mathbb{V}_1)$, where $\mathbb{V}_1$ is the local system on $\mathscr{M}_g$ associated to the $\text{Mod}(g)$-representation $H_1(\Sigma_g, \mathbb{Z})$, as $\mathscr{M}_g$ is a $K(\on{Mod}(g), 1)$.

Morita announced in \cite{Morita86} with proof in \cite{Morita89} the following theorem. For a geometric interpretation and proof, see \cite[Section 7]{HainReed01}.
\begin{theorem}[Morita, \cite{Morita86,Morita89}]
For $g\geq 9,$
$$H^2(\on{Mod}(g), H_1(\Sigma_g,\mathbb{Z}))=(\mathbb{Z}/(2g-2)\mathbb{Z})o_1.$$
\end{theorem}

\subsubsection{An analytic construction}\label{Pic-construction-analytic}
We briefly give another description of the Morita class, in terms of the universal Picard variety. Let $\mathbf{Pic}^d_{\mathscr{C}_g/\mathscr{M}_g}$ be the (rigidified) moduli space of degree $d$ line bundles on the universal curve, as discussed in Section \ref{subsec:pic-preliminaries}. Then $\mathbf{Pic}^1_{\mathscr{C}_g/\mathscr{M}_g}$ is a torsor for $$J_{\mathscr{C}_g/\mathscr{M}_g}:=\mathbf{Pic}^0_{\mathscr{C}_g/\mathscr{M}_g},$$ the universal Jacobian, so we can think of it as an element $$[\mathbf{Pic}^1_{\mathscr{C}_g/\mathscr{M}_g}]\in H^1(\mathscr{M}_g, J_{\mathscr{C}_g/\mathscr{M}_g}).$$ Let $\omega:=\Omega^1_{\mathscr{C}_g/\mathscr{M}_g}$ be the relative differentials, and let $$\pi: \mathscr{C}_g\to \mathscr{M}_g$$ be the projection. Then the ``exponential" short exact sequence of sheaves on $\mathscr{M}_g$ $$0\to \mathbb{V}_1:=(R^1\pi_*\underline{\mathbb{Z}})^\vee  \to (R^0\pi_*\omega)^\vee\to J_{\mathscr{C}_g/\mathscr{M}_g}\to 0$$ induces a boundary map $$\delta: H^1(\mathscr{M}_g, J_{\mathscr{C}_g/\mathscr{M}_g})\to H^2(\mathscr{M}_g, \mathbb{V}_1).$$
The following proposition (which we will not use) explains how to interpret the construction in the previous section in terms of this data.
\begin{proposition}
Under the natural identification $$H^2(\on{Mod}(g), H_1(\Sigma_g, \mathbb{Z}))\simeq H^2(\mathscr{M}_g, \mathbb{V}_1),$$ the Morita class $o_{\text{univ}}$ maps to $\delta([\mathbf{Pic}^1_{\mathscr{C}_g/\mathscr{M}_g}])$. 
\end{proposition}
\begin{proof}[Proof sketch]
This is an immediate consequence of the fact that the Abel-Jacobi map $$\mathscr{C}_g\to \mathbf{Pic}^1_{\mathscr{C}_g/\mathscr{M}_g}$$ induces an isomorphism on $\pi_1^{\text{ab}}$ on the fiber over every point of $\mathscr{M}_g$, combined with the fact that for any Riemann surface $C$, $\mathbf{Pic}^1_C$ is canonically a $K(\pi_1(C)^{\text{ab}}, 1)$.
\end{proof}
See \cite[Section 2]{ebert2010stable} for a discussion of the $J_{\mathscr{C}_g/\mathscr{M}_g}$-torsors $[\mathbf{Pic}^d_{\mathscr{C}_g/\mathscr{M}_g}]$.
\subsection{The \'etale Morita class} There is some subtlety in defining the analogue of the Morita class in \'etale cohomology, because it is not known that the mapping class group $\on{Mod}(g)$ is a  ``good group" in the sense of Serre (see e.g.~\cite[3.4]{farb2006problems} for a brief discussion)---in particular, it is not immediately clear that the Birman exact sequence remains exact upon profinite completion (though this is in fact true, and has been used in existing literature, e.g. in \cite[Section 3.1]{HM}). Nonetheless we give two (equivalent) constructions of an \'etale-cohomological analogue of the Morita class. As before, we fix an integer $g>1$ throughout. 
\subsubsection{A group-theoretic construction}
\begin{proposition}\label{prop:profinite-birman-exact}
Let $p$ be a prime. The profinite (resp.~prime-to-$p$) completion of the Birman exact sequence is exact.
\end{proposition}
\begin{proof}
This follows immediately from \cite[Proposition 3]{anderson1974exactness}, as the profinite (resp.~prime-to-$p$) completion of a surface group of genus $g>1$ has trivial center.
\end{proof}
Now let $k$ be a field and $X$ a Deligne-Mumford stack over $k$; let $\bar x$ be a geometric point of $X$. We denote by $\pi_1^{\text{\'et}}(X, \bar x)$ the \'etale fundamental group of $X$, and by $\pi_1^{(p)}(X, \bar x)$ the group obtained via the following pushout:
$$\xymatrix{
\pi_1^{\text{\'et}}(X_{\bar k}, \bar x) \ar[r] \ar[d] & \pi_1^{\text{\'et}}(X, \bar x)\ar[d]\\
\widehat{\pi_1^{\text{\'et}}(X_{\bar k}, \bar x)}^{(p)} \ar[r] & \pi_1^{(p)}(X, \bar x).
}$$
Here the object in the lower left is the prime-to-$p$ completion of $\pi_1^{\text{\'et}}(X_{\bar k}, \bar x)$.

Now for $\bar x$ a geometric point of $\mathscr{C}_{g, k}$, $[\overline C]$ the corresponding geometric point of $\mathscr{M}_{g,k}$, and $\overline C$ the corresponding curve over $\bar k$, let ($B$) (resp.~($B_p$)) be the following sequences of profinite groups:
\begin{equation}\tag{$B$}
 1\to  \pi_1^{\text{\'et}}(\overline C, \bar x)\to \pi_1^{\text{\'et}}(\mathscr{C}_{g,k}, \bar x)\to \pi_1^{\text{\'et}}(\mathscr{M}_{g, k}, [\overline C])\to 1
\end{equation}
\begin{equation}\tag{$B_p$}
 1\to  \pi_1^{(p)}(\overline C, \bar x)\to \pi_1^{(p)}(\mathscr{C}_{g, k}, \bar x)\to \pi_1^{(p)}(\mathscr{M}_{g, k}, [\overline C])\to 1
\end{equation}
By Proposition \ref{prop:profinite-birman-exact}, sequence ($B$) is exact if $k$ has characteristic $0$, and sequence ($B_p$) is exact if $k$ has characteristic $p$.

Now taking the quotient by the derived subgroup of the group on the left gives short exact sequences 
\begin{equation}\tag{$B^{\text{ab}}$}\label{seq:abelian-char-zero}
 1\to  \pi_1^{\text{\'et}}(\overline C, \bar x)^{\text{ab}}\to \pi_1^{\text{\'et}}(\mathscr{C}_{g,k}, \bar x)/[\pi_1^{\text{\'et}}(\overline C, \bar x),\pi_1^{\text{\'et}}(\overline C, \bar x)]\to \pi_1^{\text{\'et}}(\mathscr{M}_{g, k}, [\overline C])\to 1
\end{equation}
\begin{equation}\tag{$B_p^{\text{ab}}$}\label{seq:abelian-char-p}
 1\to  \pi_1^{(p)}(\overline C, \bar x)^{\text{ab}}\to \pi_1^{(p)}(\mathscr{C}_{g, k}, \bar x)/[\pi_1^{(p)}(\overline C, \bar x),\pi_1^{(p)}(\overline C, \bar x)]\to \pi_1^{(p)}(\mathscr{M}_{g, k}, [\overline C])\to 1
\end{equation}
If $k$ is a field of characteristic $0$, let $\widehat{\mathbb{V}_1}$ be the lisse $\widehat{\mathbb{Z}}$-sheaf on $\mathscr{M}_{g,k}$ associated to the $\pi_1^{\text{\'et}}$-representation $\pi_1^{\text{\'et}}(\overline C, \bar x)^{\text{ab}}$ (equivalently, $\widehat{\mathbb{V}_1}=(R^1\pi_*\widehat{\mathbb{Z}})^\vee$). If $k$ has characteristic $p$, let $\widehat{\mathbb{V}_1}^{(p)}$ be the lisse $\widehat{\mathbb{Z}}^{(p)}$-sheaf associated to $\pi_1^{(p)}(\overline C, \bar x)^{\text{ab}}$ (equivalently, $\widehat{\mathbb{V}_1}^{(p)}=(R^1\pi_*\widehat{\mathbb{Z}}^{(p)})^\vee$).

If $k$ is a field of characteristic $0$, sequence (\ref{seq:abelian-char-zero}) gives rise to a class in $H^2(\pi_1^{\text{\'et}}(\mathscr{M}_{g, k}, [\overline C]), \pi_1^{\text{\'et}}(\overline C, \bar x)^{\text{ab}})$.  We let $o_{1, \text{\'et}}$ be the image of this class in $H^2(\mathscr{M}_{g, k}, \mathbb{V}_1)$ under the natural map $$H^2(\pi_1^{\text{\'et}}(\mathscr{M}_{g, k}, [\overline C]), \pi_1^{\text{\'et}}(\overline C, \bar x)^{\text{ab}})\to H^2(\mathscr{M}_{g, k}, \widehat{\mathbb{V}_1}).$$ Similarly, if $k$ is a field of characteristic $p>0$, sequence (\ref{seq:abelian-char-p}) gives rise to a class in $H^2(\pi_1^{(p)}(\mathscr{M}_{g, k}, [\overline C]), \pi_1^{(p)}(\overline C, \bar x)^{\text{ab}})$; we let $o_{1, \text{\'et}}^{(p)}$ be its image in $H^2(\mathscr{M}_{g,k}, \widehat{\mathbb{V}_1}^{(p)})$.
\subsubsection{A construction using the Picard variety}\label{subsubsec:Pic-construction-o1} We can also imitate the construction in section \ref{Pic-construction-analytic}, giving a construction which works over an arbitrary base $S$. As before, we may consider the torsor $$[\mathbf{Pic}^1_{\mathscr{C}_g/\mathscr{M}_g}]\in H^1(\mathscr{M}_{g,S,\text{\'et}}, J_{\mathscr{C}_g/\mathscr{M}_g, S}).$$ 
\begin{definition}
Let $S$ be a scheme and let $P$ be the set of primes invertible in $S$, and define $$\widehat{\mathbb{Z}}_S:=\prod_{p\in P} \mathbb{Z}_p.$$ Let $$\kappa: H^1(\mathscr{M}_{g,S,\text{\'et}}, J_{\mathscr{C}_g/\mathscr{M}_g, S})\to H^2(\mathscr{M}_{g,S,\text{\'et}}, (R^1\pi_*\widehat{\mathbb{Z}}_S)^\vee)$$ be the Kummer map. Then $$o^S_{1,\text{\'et}}:=\kappa([\mathbf{Pic}^1_{\mathscr{C}_g/\mathscr{M}_g}]).$$
\end{definition}
One can check (using e.g.~geometric class field theory) that for $S=\text{Spec}(k)$, these classes agree with those defined in the previous section. We will not use this fact. 
\subsection{The secondary Morita classes}
We now describe the classes $\widetilde{o_2}, \widetilde{o_{2, \text{\'et}}}$, which obstruct splittings of a \emph{$2$-nilpotent} version of sequence (\ref{fundamental-exact-sequence}). For simplicity of presentation we have relegated the involved group cohomology computations required to define these classes to the Appendix (Section \ref{sec:Non-abelian-cohomology}), but we briefly summarize them here and discuss how they are applied to our situation.

\subsubsection{The topological setting} Fix an integer $g>1$ and let $\Sigma_g$ be a compact orientable surface of genus $g$. Let $L^k\pi_1(\Sigma_g)$ denote the lower central series of $\pi_1(\Sigma_g)$, i.e. $L^{k+1}\pi_1(\Sigma_g) \colonequals [\pi_1(\Sigma_g) ,L^k\pi_1(\Sigma_g)]$ with $L^1 \pi_1(\Sigma_g) = \pi_1(\Sigma_g)$. 
Recall that from the Birman exact sequence
\begin{equation}\label{equ:Birman}
    1 \to \pi_1(\Sigma_g) \to \text{Mod}(g,1) \to \text{Mod}(g) \to 1
\end{equation}
we obtained the sequence
\[ \xymatrix{1 \ar[r] & 
\pi_1(\Sigma_g)/L^2\pi_1(\Sigma_g) \ar[r] &
\text{Mod}(g,1)/L^2\pi_1(\Sigma_g) \ar[r]^-{p} &
\text{Mod}(g) \ar[r]&
1,} \]
which corresponded to the class $o_1 \in H^2(\Mod(g),\pi_1(\Sigma_g)/L^2\pi_1(\Sigma_g))$.

We now observe that $\Mod(g,1)/L^2\pi_1(\Sigma_g)$ is canonically isomorphic (via the Abel-Jacobi map) to $\pi_1(\mathbf{Pic}^1_{\mathscr{C}_g/\mathscr{M}_g})$, as for any Riemann surface $C$ the Abel-Jacobi map induces an isomorphism $\pi_1(C)^{\text{ab}}\simeq \pi_1(\on{Pic}^1_C)$. 
Now let the group 
$$\tilde{\pi}_g = \Mod(g,1) \times_{\Mod(g)} \pi_1(\mathbf{Pic}^1_{\mathscr{C}_g/\mathscr{M}_g})$$ be the fiber product via the natural map $\Mod(g,1)\to \Mod(g)$ and the map $p$. The group $\tilde\pi_g$ is the fundamental group of the base change of the universal curve $\mathscr{C}_g$ to $\mathbf{Pic}^1_{\mathscr{C}_g/\mathscr{M}_g}.$

We have the following surjection of short exact sequences, where the bottom row is the Birman sequence and the rightmost square is Cartesian:
$$\xymatrix{
1 \ar[r]  & 
\pi_1(\Sigma_g) \ar[r]\ar@{=}[d] &
\tilde{\pi}_g \ar[r]\ar[d] &
 \pi_1(\mathbf{Pic}^1_{\mathscr{C}_g/\mathscr{M}_g}) \ar[r]\ar[d]^{p}& 1 \\
1 \ar[r]  & 
\pi_1(\Sigma_g) \ar[r]&
\text{Mod}(g,1) \ar[r] &
\text{Mod}(g) \ar[r]& 1}.$$
We construct a cohomology class which obstructs the splitting of a \emph{$2$-nilpotent} version of the top sequence above.

First observe that the pullback $p^*o_1$ of the primary Morita class to $\pi_1(\mathbf{Pic}^1_{\mathscr{C}_g/\mathscr{M}_g})$ is trivial. To see this triviality, note that this class classifies the sequence
\begin{equation}\label{pi2toModsequence}
1 \to \pi_1(\Sigma_g)/L^2\pi_1(\Sigma_g) \to \tilde{\pi}_g/L^2\pi_1(\Sigma_g) \to  \pi_1(\mathbf{Pic}^1_{\mathscr{C}_g/\mathscr{M}_g}) \to 1,
\end{equation}
which we claim splits.
Indeed, the map $$\tilde{\pi}_g/L^2\pi_1(\Sigma_g) \simeq  \pi_1(\mathbf{Pic}^1_{\mathscr{C}_g/\mathscr{M}_g}) \times_{\Mod(g)} \pi_1(\mathbf{Pic}^1_{\mathscr{C}_g/\mathscr{M}_g}) \to  \pi_1(\mathbf{Pic}^1_{\mathscr{C}_g/\mathscr{M}_g})$$ has a natural section, given by the diagonal map $\Delta$. We may now apply the construction in the Appendix, Section~\ref{sec:Non-abelian-cohomology}, taking $\pi=\pi_1(\Sigma_g),$ $\tilde{\pi}=\tilde\pi_g$ and $G=\pi_1(\mathbf{Pic}^1_{\mathscr{C}_g/\mathscr{M}_g})$. We briefly introduce the content of this construction, leaving the proofs to Section \ref{sec:Non-abelian-cohomology}.

Consider the sequences 
\begin{equation}\label{pisequence}
0 \to L^2\pi_1(\Sigma_g)/L^3\pi_1(\Sigma_g) \to \pi_1(\Sigma_g)/L^3\pi_1(\Sigma_g) \to \pi_1(\Sigma_g)/L^2\pi_1(\Sigma_g) \to 0,
\end{equation}
\begin{equation}\label{pi2sequence}
0 \to L^2\pi_1(\Sigma_g)/L^3\pi_1(\Sigma_g) \to \tilde{\pi}_g/L^3\pi_1(\Sigma_g) \to \tilde{\pi}_g/L^2\pi_1(\Sigma_g) \to 1.
\end{equation}
Note that the sequences \ref{pi2toModsequence}, \ref{pisequence}, \ref{pi2sequence} correspond to Sequences \ref{seq:abelian_seq},  \ref{seq:nilpotent_seq}, \ref{seq:l2modl3extension}, respectively. Sequence (\ref{pi2sequence})  is classified by a class $b\in H^2(\tilde{\pi}_g/L^2\pi_1(\Sigma_g), L^2\pi_1(\Sigma_g)/L^3\pi_1(\Sigma_g))$. We define 
\begin{equation}\label{eq:o2}
o_2:=\Delta^*b\in H^2(\pi_1(\mathbf{Pic}^1_{\mathscr{C}_g/\mathscr{M}_g}), L^2\pi_1(\Sigma_g)/L^3\pi_1(\Sigma_g)).
\end{equation}
Now there are maps 
\begin{align*}m: H^1(\pi_1(\mathbf{Pic}^1_{\mathscr{C}_g/\mathscr{M}_g}), \pi_1(\Sigma_g)^{\text{ab}})^{\otimes 2} &\overset{\cup}{\longrightarrow}H^2(\pi_1(\mathbf{Pic}^1_{\mathscr{C}_g/\mathscr{M}_g}), (\pi_1(\Sigma_g)^{\text{ab}})^{\otimes 2})\\
&\overset{[-,-]}{\longrightarrow}H^2(\pi_1(\mathbf{Pic}^1_{\mathscr{C}_g/\mathscr{M}_g}), L^2\pi_1(\Sigma_g)/L^3\pi_1(\Sigma_g))
\end{align*}

(given by the composition of the commutator map with the cup product), and $$\delta_{\Delta}: H^1(\pi_1(\mathbf{Pic}^1_{\mathscr{C}_g/\mathscr{M}_g}), \pi_1(\Sigma_g)^{\text{ab}})\to H^2(\pi_1(\mathbf{Pic}^1_{\mathscr{C}_g/\mathscr{M}_g}), L^2\pi_1(\Sigma_g)/L^3\pi_1(\Sigma_g))$$
given by the long exact sequence in non-abelian cohomology arising from sequence \ref{pisequence} (and using the $\pi_1(\mathbf{Pic}^1_{\mathscr{C}_g/\mathscr{M}_g})$-action on $\pi_1(\Sigma_g)/L^3\pi_1(\Sigma_g)$ arising from $\Delta$). We define $$\overline{H^2(\pi_1(\mathbf{Pic}^1_{\mathscr{C}_g/\mathscr{M}_g}), L^2\pi_1(\Sigma_g)/L^3\pi_1(\Sigma_g))}:=\on{coker}(m).$$
By Proposition \ref{prop:non-linear-bdry-map}, the composite map 
\begin{align*}
\delta: H^1(\pi_1(\mathbf{Pic}^1_{\mathscr{C}_g/\mathscr{M}_g}), \pi_1(\Sigma_g)^{\text{ab}})&\overset{\delta_{\Delta}}{\longrightarrow} H^2(\pi_1(\mathbf{Pic}^1_{\mathscr{C}_g/\mathscr{M}_g}), L^2\pi_1(\Sigma_g)/L^3\pi_1(\Sigma_g))\\
&\to \overline{H^2(\pi_1(\mathbf{Pic}^1_{\mathscr{C}_g/\mathscr{M}_g}), L^2\pi_1(\Sigma_g)/L^3\pi_1(\Sigma_g))}
\end{align*}
is linear (and by Proposition \ref{prop: bdry-independence-of-section}, it is independent of the $\pi_1(\mathbf{Pic}^1_{\mathscr{C}_g/\mathscr{M}_g})$-action on $\pi_1(\Sigma_g)/L^3\pi_1(\Sigma_g)$). We define $$M(\pi_1(\mathbf{Pic}^1_{\mathscr{C}_g/\mathscr{M}_g}), L^2\pi_1(\Sigma_g)/L^3\pi_1(\Sigma_g)):=\on{coker}(\delta).$$
\begin{definition}\label{def:SecondaryMoritaClass}
We define the secondary Morita class $\tilde{o}_2 \in M(\pi_1(\mathbf{Pic}^1_{\mathscr{C}_g/\mathscr{M}_g}), L^2\pi_1(\Sigma_g)/L^3\pi_1(\Sigma_g))$ to be the image of $o_2$ in this quotient group. (Compare to Definition \ref{def:obs-independent-of-section}.) 
\end{definition}
For the functoriality properties of this class, and details of the claims above, see Section \ref{sec:Non-abelian-cohomology}.

As $\mathbf{Pic}^1_{\mathscr{C}_g/\mathscr{M}_g}$ is a $K(\pi_1(\mathbf{Pic}^1_{\mathscr{C}_g/\mathscr{M}_g}), 1)$, we may just as well think of the class $\widetilde{o_2}$ as living in a quotient of $H^2(\mathbf{Pic}^1_{\mathscr{C}_g/\mathscr{M}_g}, \mathbb{V}_2)$, denoted $M(\mathbf{Pic}^1_{\mathscr{C}_g/\mathscr{M}_g}, \mathbb{V}_2)$, where $\mathbb{V}_2$ is the local system corresponding to the $\pi_1$-representation $L^2\pi_1(\Sigma_g)/L^3\pi_1(\Sigma_g)$. We may describe the local system $\mathbb{V}_2$ more explicitly as follows. There is a natural map of local systems on $\mathscr{M}_g$ $$\mathbb{Z}(1)\to \bigwedge^2 \mathbb{V}_1,$$ given by the intersection pairing on $\mathbb{V}_1^\vee$. The local system $\mathbb{V}_2$ is the pullback of the cokernel of this map to $\mathbf{Pic}^1_{\mathscr{C}_g/\mathscr{M}_g}$.

\subsubsection{The \'etale setting}\label{subsubsec:etale-o2}
We only sketch the construction in the \'etale-cohomological setting, as it is essentially identical to the construction in the previous section. Again, by Proposition \ref{prop:profinite-birman-exact}, the profinite (resp.~prime-to-$p$) completion of the Birman sequence remains exact; hence, one may construct profinite (resp.~prime-to-$p$) analogues of all of the exact sequences above, as we now explain.  

Let $k$ be a field of characteristic $0$. As we have already shown in Proposition \ref{prop:profinite-birman-exact}, the profinite completion of the Birman sequence is exact. Let $\bar x$ be a geometric point of $\mathscr{C}_g$, $[\overline{C}]$ the corresponding point of $\mathscr{M}_g$, and $\overline{C}$ the corresponding curve. We have (by geometric class field theory) a canonical isomorphism $$\pi_1^{\text{\'et}}(\mathscr{C}_{g, k}, \bar x)/L^2\pi_1^{\text{\'et}}(\overline{C}, \bar x)\simeq \pi_1^{\text{\'et}}(\mathbf{Pic}^1_{\mathscr{C}_g/\mathscr{M}_g, k}, \bar x),$$ where on the right we view $\bar x$ as a point of $\mathbf{Pic}^1_{\mathscr{C}_g/\mathscr{M}_g, k}$ via the Abel-Jacobi map. If $k$ is a field of characteristic $p>0$, we have an analogous isomorphism on $\pi_1^{(p)}$. Thus we may construct profinite (resp.~prime-to-$p$) analogues of sequences (\ref{pi2toModsequence}), (\ref{pisequence}), (\ref{pi2sequence}), as well as all of the other diagrams above, giving rise to classes $$o_{2, \text{\'et}}\in H^2(\pi_1^{\text{\'et}}(\mathbf{Pic}^1_{\mathscr{C}_g/\mathscr{M}_g, k}, \bar x), L^2\pi_1^{\text{\'et}}(\overline{C}, \bar x)/L^3\pi_1^{\text{\'et}}(\overline{C}, \bar x)),$$ $$\widetilde{o_{2, \text{\'et}}}\in M(\pi_1^{\text{\'et}}(\mathbf{Pic}^1_{\mathscr{C}_g/\mathscr{M}_g, k}, \bar x), L^2\pi_1^{\text{\'et}}(\overline{C}, \bar x)/L^3\pi_1^{\text{\'et}}(\overline{C}, \bar x))$$ in characteristic $0$ and 
$$o_{2, \text{\'et}}^{(p)}\in H^2(\pi_1^{(p)}(\mathbf{Pic}^1_{\mathscr{C}_g/\mathscr{M}_g, k}, \bar x), L^2\pi_1^{(p)}(\overline{C}, \bar x)/L^3\pi_1^{(p)}(\overline{C}, \bar x)),$$ $$\widetilde{o_{2, \text{\'et}}}^{(p)}\in M(\pi_1^{(p)}(\mathbf{Pic}^1_{\mathscr{C}_g/\mathscr{M}_g, k}, \bar x), L^2\pi_1^{(p)}(\overline{C}, \bar x)/L^3\pi_1^{(p)}(\overline{C}, \bar x))$$
in characteristic $p>0$. We denote the lisse $\widehat{\mathbb{Z}}$-sheaf (resp.~$\widehat{\mathbb{Z}}^{(p)}$-sheaf) on $\mathbf{Pic}^1_{\mathscr{C}_g/\mathscr{M}_g, k, \text{\'et}}$ corresponding to the coefficients in the cohomology groups above via $\widehat{\mathbb{V}_2},$ (resp.~$\widehat{\mathbb{V}_2}^{(p)}$). As before we may explicitly describe $\widehat{\mathbb{V}_2}$ as the pullback to $\mathbf{Pic}^1_{\mathscr{C}_g/\mathscr{M}_g}$ of the cokernel of the map of lisse sheaves on $\mathscr{M}_g$ $$\widehat{\mathbb{Z}}(1)\to \bigwedge^2\widehat{\mathbb{V}_1}$$ arising from the intersection form, and similarly with $\widehat{\mathbb{V}_2}^{(p)}$.

Pulling back along the natural map $$H^2(\pi_1^{\text{\'et}}(\mathbf{Pic}^1_{\mathscr{C}_g/\mathscr{M}_g, k}, \bar x), L^2\pi_1^{\text{\'et}}(\overline{C}, \bar x)/L^3\pi_1^{\text{\'et}}(\overline{C}, \bar x))\to H^2(\mathbf{Pic}^1_{\mathscr{C}_g/\mathscr{M}_g, k, \text{\'et}}, \widehat{\mathbb{V}_2})$$
we obtain a class in $H^2(\mathbf{Pic}^1_{\mathscr{C}_g/\mathscr{M}_g, k, \text{\'et}}, \widehat{\mathbb{V}_2})$ which we also call $o_{2, \text{\'et}}$, in an abuse of notation, and similarly with $\widetilde{o_{2, \text{\'et}}}, o_{2,{\text{\'et}}}^{(p)}, \widetilde{o_{2,_{\text{\'et}}}}^{(p)}$. Note that we do not know if these pullback maps are isomorphisms, as it is not clear if $\mathbf{Pic}^1_{\mathscr{C}_g/\mathscr{M}_g, k}$ is an \'etale $K(\pi, 1)$.
\subsection{Vanishing of $o_1$}
We now discuss certain situations where the classes $o_1, o_1^{\text{\'et}}, o_{1, \text{\'et}}^{(p)}$ vanish for geometric reasons, and some consequences of this vanishing for $o_2$. We first record the following fundamental property of $o_1$:
\begin{proposition}\label{prop:o1-vanishing-tautological-consequence}
Let $g>2$ be an integer. 
\begin{enumerate}
    \item Let $E\to B$ be a fibration with fiber $\Sigma_g$, associated to a map $f: B\to \mathscr{M}_g$. Then $f^*o_1=0$ if and only if central extension obtained by pulling back sequence (\ref{seq:abelianized-Birman}) along the map $f_*: \pi_1(B)\to \text{Mod}(g)$ splits.
    \item Let $C$ be a smooth projective curve of genus $g$ over a field $k$ of characteristic $0$, associated to a map $\on{Spec}(k)\to \mathscr{M}_{g,k}$. Then the natural exact sequence $$0\to \pi_1^{\text{\'et}}(C_{\bar k})^{\text{ab}}\to \pi_1^{\text{\'et}}(C)/L^2\pi_1^{\text{\'et}}(C_{\bar k})\to \Gal(\bar k/k)\to 1$$ splits if and only if $o_{1, \text{\'et}}|_k$ vanishes.
    \item Let $C$ be a smooth projective curve of genus $g$ over a field $k$ of characteristic $p>0$, associated to a map $\on{Spec}(k)\to \mathscr{M}_{g,k}$. Then the natural exact sequence $$0\to \pi_1^{(p)}(C_{\bar k})^{\text{ab}}\to \pi_1^{(p)}(C)/L^2\pi_1^{(p)}(C_{\bar k})\to \Gal(\bar k/k)\to 1$$ splits if and only if $o^{(p)}_{1, \text{\'et}}|_k$ vanishes.
\end{enumerate} 
\end{proposition}
\begin{proof}
Immediate from the definition and functoriality of the sequences above.
\end{proof}
We now deduce some vanishing properties of $o_1$.
\begin{proposition}\label{prop:o1-vanishing-cpct-type}
Let $R$ be a Noetherian complete local ring with maximal ideal $\mathfrak{m}$, residue field $R/\mathfrak{m}=k$ and fraction field $K$. Let $C/R$ be a projective curve with $C_K$ smooth of genus $g$ and $C_k$ semistable of compact type. Suppose $C_k(k)$ is non-empty. If the residue characteristic of $R$ is zero, then $o_{1, \text{\'et}}|_K$ is zero. 
\end{proposition}
\begin{proof}
We give the proof if the characteristic of $k$ is zero; the proof in positive residue characteristic is essentially identical, though more notationally involved. Let $\overline{K}$ be an algebraic closure of $K$ and $\overline{k}$ an algebraic closure of $k$. Let $\bar x$ be a $\overline{K}$-point of $C_K$ and $\bar\xi$ a specialization of $\bar x$. Then there is a commutative diagram of specialization maps
$$\xymatrix{
\pi_1^{\text{\'et}}(C_K, \bar x) \ar[r]^{\on{sp}}\ar[d]^{\pi_K} & \pi_1^{\text{\'et}}(C_k, \bar \xi)\ar[d]^{\pi_k} \\
\on{Gal}(\overline{K}/K)\ar[r] & \on{Gal}(\overline{k}/k).
}$$
Now consider the quotient diagram 
$$\xymatrix{
\pi_1^{\text{\'et}}(C_K, \bar x)/L^2\pi_1^{\text{\'et}}(C_{\overline{K}}, \bar x) \ar[rr]^{\on{sp}}\ar[d]^{\pi_K^{\text{ab}}}& & \pi_1^{\text{\'et}}(C_k, \bar \xi)/L^2\pi_1^{\text{\'et}}(C_{\overline{k}}, \bar \xi)\ar[d]^{\pi_k^{\text{ab}}} \\
\on{Gal}(\overline{K}/K)\ar[rr]& & \on{Gal}(\overline{k}/k).
}$$
As $C_{\bar k}$ has compact type, this quotient diagram is Cartesian. But $\pi_k$ has a section (as $C_k$ has a rational point), hence the same is true for $\pi_k^{\text{ab}}$. Thus $\pi_K^{\text{ab}}$ has a section, using the Cartesian-ness of the diagram above.
\end{proof}
\begin{corollary}\label{cor:o1-vanishing}
 Let $\Gamma$ be a stable graph of genus $g>2$ and let $k$ be a field. Suppose that
\begin{enumerate}
    \item The underlying graph of $\Gamma$ is a tree, and
    \item $\on{Aut}(\Gamma)$ stabilizes some edge of $\Gamma$.
\end{enumerate}
Then if $k$ has characteristic zero, $o_{1, \text{\'et}}|_{\widehat{K_\Gamma}}$ vanishes.
Here $\widehat{K_\Gamma}$ is defined as in Section \ref{subsec:tropical-section-conjecture}.
\end{corollary}
\begin{proof}
Let $R_\Gamma$ be the complete local ring of $\overline{\mathscr{M}_g}$ at the generic point of $Z_\Gamma$, the boundary stratum of $\overline{\mathscr{M}_g}$ corresponding to stable curves with dual graph $\Gamma$. Let $\mathscr{C}_{\Gamma}$ be the pullback of the universal curve to $R_\Gamma$. The special fiber of $\mathscr{C}_\Gamma$ has compact type by (1), and has a rational point by (2) (namely, the node corresponding to the stabilized edge must be rational). Hence we may conclude by Proposition \ref{prop:o1-vanishing-cpct-type}.
\end{proof}
\begin{remark}\label{rem:o1-vanishing}
Let $\Gamma$ be a graph as in Corollary \ref{cor:o1-vanishing}, and let $W_\Gamma$ be the union of all the boundary divisors of $\overline{\mathscr{M}_{g, \mathbb{C}}}^{\text{an}}$ not containing the stratum $Z_\Gamma$. Let $U_\Gamma$ be a deleted neighborhood of $Z_\Gamma$ in $\overline{\mathscr{M}_{g, \mathbb{C}}}^{\text{an}}\setminus W_\Gamma$. Then an argument essentially identical to the proof of Corollary \ref{cor:o1-vanishing} shows that $o_1|_U$ vanishes. One may also make a rigid-analytic version of this statement, but doing so is beyond the scope of this paper.
\end{remark}
\begin{remark}\label{rem:o2-defn-remark}
In any setting where $o_1$ or $o_{1, \text{\'et}}$ vanishes (as in Corollary \ref{cor:o1-vanishing} or Remark \ref{rem:o1-vanishing}), one may define a version of $\widetilde{o_2}$ or $\widetilde{o_{2, \text{\'et}}}$. Indeed, the vanishing of these classes imply that the sequences arising in Proposition \ref{prop:o1-vanishing-tautological-consequence} split, which suffices to apply the construction in Section \ref{subsec:nonabelian-cohomology}. For example, there is an analogue of $\widetilde{o_{2, \text{\'et}}}$ defined in $M(\widehat{K_\Gamma}, L^2\Pi_\Gamma/L^3\Pi_\Gamma)$, where $\Gamma, \widehat{K}_\Gamma$ are as defined in Corollary \ref{cor:o1-vanishing} and $\Pi_\Gamma$ is the geometric \'etale fundamental group of $\mathscr{C}_\Gamma$. Likewise, there is an analogous class in $M(U_\Gamma, \mathbb{V}_2)$, where $U_\Gamma$ is as defined in \ref{rem:o1-vanishing}.
\end{remark}
\section{The Morita classes for surface groups}\label{SMorSurfGp}
Let $h>0$ be an integer and $\Sigma_h$ a compact orientable surface of genus $h$. In this section, we study the pull-back of the Morita classes along maps $$\pi_1(\Sigma_h)\to \text{Mod}(g)$$ or $$\pi_1(\Sigma_h)\to \pi_1(\mathbf{Pic}^1_{\mathscr{C}_g/\mathscr{M}_g}).$$ In other words, given a fibration $$\Sigma_g\to E\to \Sigma_h,$$ we study the obstruction to splitting the ``abelianized" and ``2-nilpotent" analogues of the exact sequence of fundamental groups: $$1\to \pi_1(\Sigma_g)/L^2\pi_1(\Sigma_g)\to \pi_1(E)/L^2\pi_1(\Sigma_g)\to \pi_1(\Sigma_h)\to 1$$ and $$1\to \pi_1(\Sigma_g)/L^3\pi_1(\Sigma_g)\to \pi_1(E)/L^3\pi_1(\Sigma_g)\to \pi_1(\Sigma_h)\to 1.$$ The calculations in this section are key to the degeneration arguments in the applications to the section conjecture in Section~\ref{SSecConjApp}. We produce explicit examples (Theorem~\ref{T:nontrivialMorita} through Corollary \ref{cor:main-corollary-o1-pinwheel}) where the pull-back of the primary Morita class $o_1$ is nontrivial. When the pullback of the Morita class $o_1$ is trivial, we analyze the pullback of the secondary Morita class $\widetilde{o_2}$ (Definition~\ref{def:SecondaryMoritaClass}). We produce explicit examples where this secondary Morita class has exact order $2$
(Theorems~\ref{P:nontrivialsecondaryMorita} and \ref{thm:tree-o2-nonvanishing}).

Our results will require substantial direct computation with cocycles; we will delay these computations to Section \ref{sec:Non-abelian-cohomology} wherever possible.
\subsection{Computing the primary Morita class for surface groups}\label{S:MorSurfGrp}
We write $$\pi_1(\Sigma_h)=\left\langle a_1, \cdots, a_h, b_1, \cdots, b_h \middle\vert \prod_{i=1}^h [a_i, b_i]\right\rangle$$ for the standard presentation of $\pi_1(\Sigma_h)$.

Fix a homomorphism $$\gamma: \pi_1(\Sigma_h)\to \text{Mod}(g).$$ Choose lifts $\widetilde{\gamma(a_1)},\widetilde{\gamma(b_1)}\cdots, \widetilde{\gamma(a_h)}, \widetilde{\gamma(b_h)}$ of $\gamma(a_i), \gamma(b_i)$ from $\text{Mod}(g) \subset \text{Out}(\pi_1(\Sigma_g))$ to $\text{Mod}(g,1) \subset \text{Aut}(\pi_1(\Sigma_g))$. Then $$\widetilde{R}= \prod_{i=1}^h [\widetilde{\gamma(a_i)}, \widetilde{\gamma(b_i)}]$$ is an inner automorphism of $\pi_1(\Sigma_g)$, and hence can be written as conjugation by some element $\tilde r\in \pi_1(\Sigma_g)$. Let $\Hur \colon \pi_1(\Sigma_g){\rightarrow} H_1(\Sigma_g, \mathbb{Z})$ be the Hurewicz (abelianization) map.

\begin{proposition}[The Morita class for surface groups]\label{P:MoritaSurfaceGroup}
Under the identification $$H^2(\pi_1(\Sigma_h), H_1(\Sigma_g, \mathbb{Z}))\simeq H_0(\pi_1(\Sigma_h), H_1(\Sigma_g, \mathbb{Z})) \simeq H_1(\Sigma_g, \mathbb{Z})_{\pi_1(\Sigma_h)},$$ from Corollary \ref{h2-corollary}, the pull-back $\gamma^*o_{\text{1}}$ is identified with the image of $\tilde r$ under the composition $$\pi_1(\Sigma_g)\overset{\Hur}{\longrightarrow} H_1(\Sigma_g, \mathbb{Z})\twoheadrightarrow H_1(\Sigma_g, \mathbb{Z})_{\pi_1(\Sigma_h)}.$$
\end{proposition}
\begin{proof}
Let $F^{2h} \colonequals \langle a_1,b_1,\ldots,a_h,b_h\rangle$ be the free group with $2h$ generators, and let $R \colonequals \prod_{i=1}^h[a_i,b_i]$. Let $\langle R \rangle \lhd F^{2h}$ be the normal subgroup generated by $R$, and let $\langle R \rangle^{ab}$ denote its abelianization.

By \cite[Section 3.1, Theorem 2, Proposition 3]{Gruenberg}, there is a resolution of $\mathbb{Z}$ as a $\pi_1(\Sigma_h)$ module of the form
\begin{equation}\label{seq:free-resoln-surface-gp}
0 \rightarrow \langle R \rangle^{ab} \rightarrow \mathbb{Z}[\pi_1(\Sigma_h)]^{2h}\rightarrow \mathbb{Z}[\pi_1(\Sigma_h)]
\xrightarrow{\epsilon}
 \mathbb{Z} \rightarrow 0.
\end{equation}
 Comparing it with the exact sequence (\ref{equ:resolution-of-Z}), we get the isomorphism
\begin{align}\label{Eq:isoab}
    \begin{split}
        \langle R \rangle^{ab} &\simeq \mathbb{Z}[\pi_1(\Sigma_h)], \\
        R &\mapsto 1,
    \end{split}
\end{align}
hence this resolution is free.

Consider the following commutative diagram of exact sequences, where the second row is the Birman exact sequence, the third row is the pushout of the second row along the Hurewicz map $\Hur \colon \pi_1(\Sigma_g) \rightarrow H_1(\Sigma_g,\mathbb{Z})$, and the map from the first row to the second row comes from lifting the map $\gamma \colon \pi_1(\Sigma_h) \rightarrow \text{Mod}(g)$ to the map $\phi_2 \colon F^{2h} \rightarrow \text{Mod}(g,1)$ defined by $\phi_2(a_i) = \widetilde{\gamma(a_i)}$ and $\phi_2(b_i) = \widetilde{\gamma(b_i)}$ for all $i$. 



$$\xymatrix{
1 \ar[r]  & 
\langle R \rangle \ar[r]^{\iota_1}\ar[d]^{\phi_1} &
F^{2h} \ar[r]\ar[d]^{\phi_2} &
\pi_1(\Sigma_h)\ar[r]\ar[d]^{\gamma}& 1 \\
1 \ar[r]  & 
\pi_1(\Sigma_g) \ar[r]^{\iota_2}\ar[d]^{H} &
\text{Mod}(g,1) \ar[r]\ar[d] &
\text{Mod}(g) \ar[r]\ar[d]^{\simeq}& 1 \\
1 \ar[r] & 
H_1(\Sigma_g,\mathbb{Z}) \ar[r] &
\text{Mod}(g,1)/L^2\pi \ar[r] &
\text{Mod}(g) \ar[r]&
1}$$

By \cite[Section 5.3, Theorem 1]{Gruenberg}, the extension class $\gamma^*o_{\text{1}} \in H^2(\pi_1(\Sigma_h),H_1(\Sigma_g,\mathbb{Z}))$ is represented by the vertical map $H\circ\phi_1$, where we compute $H^2$ using the free resolution (\ref{seq:free-resoln-surface-gp}). Unwinding the isomorphism given  $H^2(\pi_1(\Sigma_h),H_1(\Sigma_g,\mathbb{Z})) \cong H_1(\Sigma_g,\mathbb{Z})_{\pi_1(\Sigma_h)}$ described in Corollary~\ref{h2-corollary} gives the result.
\end{proof}
\begin{remark}
We will primarily apply this result when $h=1$ --- that is, for surface bundles over a torus. It turns out this suffices for our applications. Indeed, this should be unsurprising; we are interested in proving that the Gysin images of certain cohomology classes do not vanish. These classes live on the boundary of $\mathscr{M}_g$ in cohomological degree $1$. An argument analogous to the proof of Lemma \ref{h1-cyclic-subgroup} shows that one can detect the non-vanishing of such classes by pulling along maps from the circle $S^1$; the preimage of a circle in the deleted neighborhood of a boundary component is (homotopy equivalent to) a surface of genus $1$. 
\end{remark}

Suppose we are given disjoint simple closed curves $l_1,\ldots,l_n$ on $\Sigma_g$ and a homeomorphism $S: \Sigma_g\to \Sigma_g$ which permutes them up to isotopy. For each $i$, let $T_{l_i} \in \text{Mod}(g)$ denote the corresponding Dehn twist.
Let $\pi_1(\Sigma_1) = \langle a,b | [a,b] \rangle$. Define $\gamma \colon \pi_1(\Sigma_1) \rightarrow \text{Mod}(g)$ by specifying the images of $a$ and $b$ as follows. Let $\gamma(a)$ be the Dehn multitwist $$\gamma(a) \colonequals T = \prod_{i=1}^{n}T_{l_i}.$$ Let $\gamma(b) = S \in \text{Mod}(g)$. As $S$ permutes the $l_1, \cdots,l_n$ up to isotopy, $S$ and $T$ commute in $\text{Mod}(g)$, so this defines a genuine homomorphism. Our goal now is to compute $\gamma^*o_1$ for $\gamma$ of this form.

\begin{construction}[Lifts of $S,T \in \Mod(g)$ to $\Mod(g,1)$]\label{remark:liftsST}
To apply Proposition~\ref{P:MoritaSurfaceGroup}, we need lifts $\widetilde{T},\widetilde{S} \in \Aut(\pi_1(\Sigma_g))$, which we now construct.
Fix a point $B \in \Sigma_g$ that does not lie on any of the closed curves $l_1,\ldots,l_n$. Then $B$ is fixed by $T$, and we get an induced map $\widetilde{T}: \pi_1(\Sigma_g,B) \to \pi_1(\Sigma_g,B)$, given by $g \mapsto T(g)$ for each $g$ in $\pi_1(\Sigma_g,B)$.

For any choice of a path $\lambda$ from $B$ to $S(B)$, there is an associated isomorphism \begin{align*}\pi_1(\Sigma_g,S(B)) &\simeq \pi_1(\Sigma_g,B) \\ g &\mapsto \lambda g \lambda^{-1}. \end{align*} 
 Using this isomorphism, we define a lift \begin{align*}\widetilde{S}: \pi_1(\Sigma_g,B) &\to \pi_1(\Sigma_g,B) \\ g &\mapsto \lambda S(g) \lambda^{-1}, \end{align*} 
  with inverse given by $$\widetilde{S}^{-1}(g) = S^{-1}(\lambda^{-1}) S^{-1}(g) S^{-1}(\lambda).$$
\end{construction}
  
\begin{lemma}\label{commutatorcomputation} Let $\gamma, S, T, \lambda$ be as above. Let $h$ be the loop $T(\lambda)\lambda^{-1}$.
Then the commutator $[\widetilde{T},\widetilde{S}] \in \Aut(\pi_1(\Sigma_g,B))$ is the inner automorphism of $\pi_1(\Sigma_g,B)$ given by $$g \to hgh^{-1}.$$
\end{lemma}
\begin{proof}
For any $g \in \pi_1(\Sigma_g,B)$,
\begin{align*}
[\widetilde{T},\widetilde{S}](g) &= \widetilde{T}\circ\widetilde{S}\circ\widetilde{T}^{-1}\circ\widetilde{S}^{-1}(g)
= \widetilde{T}\circ\widetilde{S}\circ\widetilde{T}^{-1}(S^{-1}(\lambda^{-1}g\lambda))\\
&= \widetilde{T}\circ\widetilde{S}(T^{-1}S^{-1}(\lambda^{-1}g\lambda))
= \widetilde{T}(\lambda T^{-1}(\lambda^{-1}g\lambda)\lambda^{-1})\\
&= T(\lambda) \lambda^{-1}g\lambda T(\lambda^{-1}). \qedhere
\end{align*}
\end{proof}

Combining this lemma with Proposition~\ref{P:MoritaSurfaceGroup}, we get the following corollary.
\begin{corollary}\label{cor:reppullback}
The class $\gamma^*o_1$ 
is the image of  $\Hur(T(\lambda)\lambda^{-1})$ in $H_1(\Sigma_g,\mathbb{Z})_{\pi_1(\Sigma_1)}$.
\end{corollary}
We also make the following simple observation:
\begin{corollary}\label{cor:subsurf}
Consider the connected subsurfaces of $\Sigma_g$ with boundary components given by loops $\{l_1,\ldots,l_n\}$. If there exists such a subsurface $\Sigma'$ stabilized by $S$, i.e. $S$ is isotopic to a mapping class which restricts to an automorphism of $\Sigma'$, then $\gamma^*o_\text{1}$ is trivial.
\end{corollary}
\begin{proof}
Pick the base point $B$ in Construction \ref{remark:liftsST} to be on $\Sigma'$, and the path $\lambda$ from $B$ to $S(B)$ to be contained in $\Sigma'$.  Since $\lambda$ does not intersect any $l_i$, the multitwist $T$ acts on $\lambda$ trivially and therefore $T(\lambda)\lambda^{-1} = 1$. Now apply Corollary~\ref{cor:reppullback}.
\end{proof}

\begin{remark}
 Corollary \ref{cor:subsurf} above could also be proven via geometric considerations; associated to the representation $\pi_1(\Sigma_1)\to \text{Mod}(g)$ is a (homotopy class of) fiber bundle over $\Sigma_1$ with fiber $\Sigma_g$, and the hypotheses of Corollary \ref{cor:subsurf} guarantee that this fiber bundle has a continuous section.
\end{remark}

\begin{figure}
\labellist
\small
\pinlabel $l_1$ [t] at 161.56 0.00
\pinlabel $l_2$ [tr] at 22.68 42.52
\pinlabel $l_3$ [br] at 34.01 150.22
\pinlabel $l_4$ [br] at 153.06 184.23
\pinlabel $l_{g-1}$ [tl] at 294.78 31.18
\pinlabel $\lambda_1$ [tr] at 82.20 11.34
\pinlabel $\lambda_2$ [r] at -2.83 90.70
\pinlabel $\lambda_3$ [br] at 73.69 172.90
\pinlabel $\lambda_{g-1}$ [tl] at 263.60 14.17
\pinlabel $b_1$ [bl] at 187.07 11.34
\pinlabel $b_2$ [bl] at 59.52 36.85
\pinlabel $b_3$ [tl] at 42.52 133.22
\pinlabel $b_4$ [tl] at 153.06 170.06
\pinlabel $b_{g-1}$ [br] at 297.61 56.69
\pinlabel $\alpha_1$ [bl] at 133.22 28.34
\pinlabel $\alpha_2$ [tl] at 65.19 76.53
\pinlabel $\alpha_3$ [t] at 107.71 136.05
\pinlabel $\alpha_{g-1}$ [l] at 246.59 42.52
\pinlabel $\beta_1$ [tl] at 124.71 51.02
\pinlabel $\beta_2$ [t] at 76.53 96.37
\pinlabel $\beta_3$ [bl] at 133.22 136.05
\pinlabel $\beta_{g-1}$ [bl] at 221.08 53.85
\endlabellist
\includegraphics[scale=1]{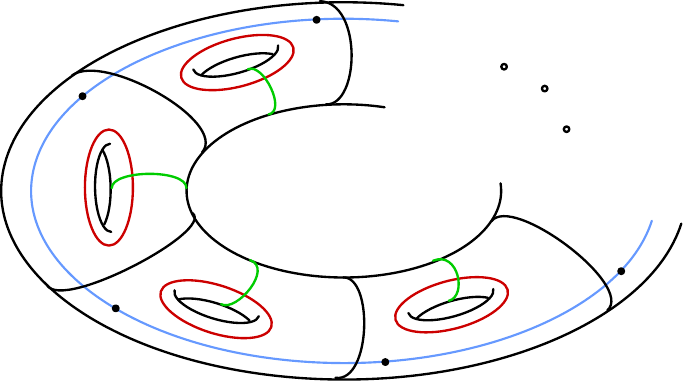}
\caption{The marked surface $\Sigma_g$ with dual graph $C_{g-1}$}\label{fig:marked-surface-g-1}
\end{figure}
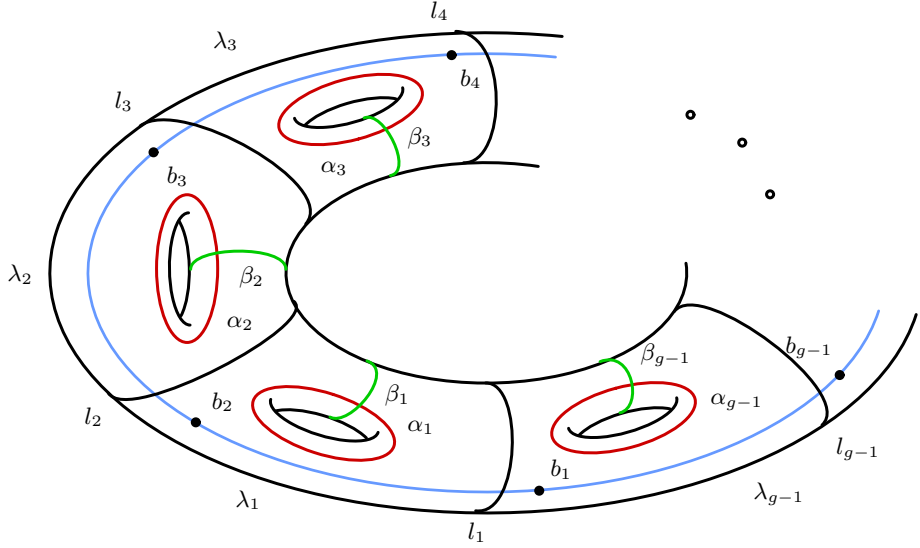

 We now give an example where $\gamma^*o_{1}$ has order $g-1$. Let $\Sigma_1^2$ be a genus one surface with two boundary components $m_1, m_2$. Consider the surface obtained by taking the quotient $$\left(\bigcup_{i\in \mathbb{Z}/(g-1)\mathbb{Z}} \Sigma_1^2\right)/\sim$$ where $\sim$ is the equivalence relation identifying the copy of $m_1$ on the $i$-th copy of $\Sigma_1^2$ with $m_2$ on the $i+1$-th copy of $\Sigma_1^2$. This is a surface of genus $g$ with $g-1$ marked loops (namely, the images of the $m_i$), which we denote $l_1, \cdots, l_{g-1}$. This marked surface is pictured in Figure \ref{fig:marked-surface-g-1}. 
 
 Let $S$ be the automorphism of $\Sigma_g$ that rotates the surface clockwise $\frac{2\pi}{g-1}$ radians (i.e.~it is induced by cyclically permuting the components of the disjoint union in the definition of our surface). Let $T_i \in \text{Mod}(g)$ be the Dehn twist around the loop $l_i \in \Sigma_g$, as indicated in Figure \ref{fig:marked-surface-g-1}.
Then we define
\begin{align*}
   \gamma(a) = T= \prod_{i=1}^{g-1}T_i,\quad
    \gamma(b) = S.
\end{align*}
Note that the dual graph of the marked surface constructed above is the stable graph $C_{g-1}$ described in the introduction.
\begin{theorem}\label{T:nontrivialMorita}
For $\gamma$ as above, the order of 
$\gamma^*o_{1}$ is  $g-1$.
\end{theorem}
\begin{proof}
We start by constructing lifts $\widetilde T, \widetilde S$ of $T$ and $S$ to $\Aut (\pi_1(\Sigma_g))$, following Construction \ref{remark:liftsST}, so that we may apply Corollary~\ref{cor:reppullback} to compute $\gamma^* o_{1}$. Fix a base point $B=b_1$ on $\Sigma_g$ that does not lie on $l_i,i=1,\ldots,g-1$. Let $\{b_1,\ldots,b_{g-1}\}$ be the $S$-orbit of $B$, such that $S(b_i) = b_{i+1}$ for $i=1,\ldots,g-2$ and $S(b_{g-1})=b_1$. The  map $S$ induces an isomorphism $\pi_1(\Sigma_g,b_1) \simeq \pi_1(\Sigma_g,b_2)$ and $S^{-1}$ induces $\pi_1(\Sigma_g,b_1) \to \pi_1(\Sigma_g,b_{g-1})$. 

Let $\lambda_1$ be a path from $b_1$ to $b_2$, and let $\lambda_i=S^{i-1}(\lambda_1)$, as indicated in blue in Figure \ref{fig:marked-surface-g-1}. Now we have isomorphisms $\pi_1(\Sigma_g,b_2) \simeq \pi_1(\Sigma_g,b_1)$ and $\pi_1(\Sigma_g,b_{g-1}) \simeq \pi_1(\Sigma_g,b_1)$ induced by conjugation by $\lambda_1$ and $\lambda_{g-1}^{-1}$. As in Construction \ref{remark:liftsST}, conjugation by these isomorphisms gives us $\widetilde{S}$ and $\widetilde{S}^{-1} : \pi_1(\Sigma_g,b_1) \to \pi_1(\Sigma_g,b_1)$. 

By construction, Dehn twists $T_i$ for $i > 1$ acts as identity on $\lambda_1$ and therefore $\widetilde{T}(\lambda_1) = T_1(\lambda_1)$. So by Corollary~\ref{cor:reppullback}, it suffices to show that the image of $ T_1(\lambda_1) \lambda_1^{-1}$ in $H_1(\Sigma_g,\mathbb{Z})_{\langle S, T \rangle}$ has order $g-1$.

We now compute the $\langle S, T \rangle$ action on $H_1(\Sigma_g,\mathbb{Z})$. For an element $g \in \pi_1(\Sigma_g,B)$, we denote by $[g]$ the class it represents in $H_1(\Sigma_g,\mathbb{Z})$. Then there is a symplectic basis $$\{ [\lambda_1\ldots\lambda_{g-1}], [l_1], \alpha_1, \ldots,\alpha_{g-1}, \beta_1,\ldots,\beta_{g-1}\}$$ of $H_1(\Sigma_g,\mathbb{Z})$ (pictured in Figure \ref{fig:marked-surface-g-1}).  For simplicity, we denote by $\lambda$ the class $[\lambda_1\ldots\lambda_{g-1}]$. Note that the image of $T_1(\lambda_1)\lambda_1^{-1}$ in $H^1(\Sigma_g, \mathbb{Z})$ is simply $[l_1]$; we wish to show that the image of this class in $H_1(\Sigma_g, \mathbb{Z})_{\langle S, T\rangle}$ has order $g-1$.

Explicitly, we have:
\begin{align*}
    S(\lambda) &= \lambda,S([l_1]) = [l_1],\quad
    S(\alpha_i) = \alpha_{i+1},\quad S(\beta_i) = \beta_{i+1},\\
    T(\lambda) &= \lambda+(g-1)[l_1],\quad
    T([l_1]) = [l_1],\quad
    T(\alpha_i)=\alpha_i,\quad T(\beta_i) = \beta_i.
\end{align*}
As $(T-1)\lambda=(g-1)[l_1]$, the class $[l_1]$ has order dividing $(g-1)$ in $H^1(\Sigma_g, \mathbb{Z})_{<S, T>}$. We now show it has order divisible by $g-1$, as desired.

Indeed, consider the map $H_1(\Sigma_g, \mathbb{Z})\to \mathbb{Z}/(g-1)\mathbb{Z}$ sending $[l_1]$ to $1$ and $\lambda, \alpha_i, \beta_j$ to zero. This map is evidently $S, T$-equivariant (for the trivial $S, T$-action on $\mathbb{Z}/(g-1)\mathbb{Z}$), and factors through $H_1(\Sigma_g, \mathbb{Z})_{\langle S, T\rangle}$, which completes the proof.
\end{proof}

\begin{remark}
The strategy of the proof of Theorem 
\ref{T:nontrivialMorita} can be used to compute the Morita class of many other marked surfaces with a given automorphism. For example, the graph in Figure \ref{Fig:Genus4Max}, where each vertex has genus $0$, is a genus $4$ stable graph, dual to the marked surface on the right.  Let $T$ be the Dehn multitwist about the marked loops and $S$ the order $3$ automorphism given by rotating $\frac{2\pi}{3}$ degrees clockwise. By direct computation, we conclude the Morita class of the associated $\Sigma_4$-bundle over the torus is nontrivial of order $3$.
\end{remark}

\begin{figure}[h]
\includegraphics[scale=0.6]{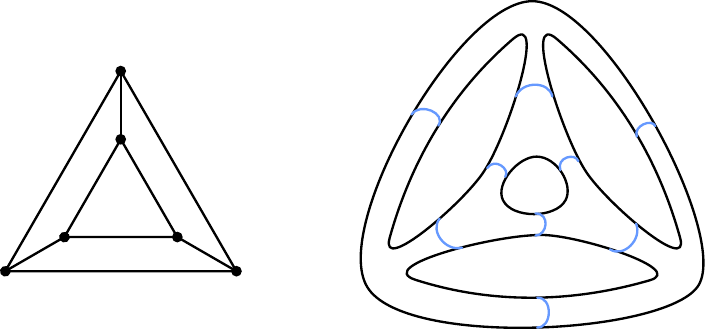}
\caption{A genus $4$ curve with maximal degeneration}\label{Fig:Genus4Max}
\end{figure}
We now observe that the result of Theorem \ref{T:nontrivialMorita} can be used to cheaply give many other examples of surface bundles with non-trivial Morita class $o_1$. We do not attempt to give an exhaustive list here, but we indicate some strategies and examples.
\begin{proposition}\label{prop:fromE1ToE2}
Let $\Sigma_{g_1} \to E_1 \to \Sigma_h$ and $\Sigma_{g_2} \to E_2 \to \Sigma_h$ be two fibrations, corresponding to maps $\gamma:\pi_1(\Sigma_h) \to \Mod(g_1)$ and $\xi:\pi_1(\Sigma_h) \to \Mod(g_2)$. Let $\rho:E_1\to E_2$ be a map over $\Sigma_h$. The induced map $H_1(\Sigma_{g_1}, \mathbb{Z})\to H_1(\Sigma_{g_2}, \mathbb{Z})$ on the homology of the fibers gives a map $$\rho_*: H^2(\pi_1(\Sigma_h),H_1(\Sigma_{g_1}, \mathbb{Z}))\to H^2(\pi_1(\Sigma_h),H_1(\Sigma_{g_2}, \mathbb{Z}))$$ satisfying $$\rho_*(\gamma^*o_1)=\xi^*o_1.$$
\end{proposition}

\begin{proof}
Immediate from the functoriality of $o_1$.
\end{proof}

Using Proposition \ref{prop:fromE1ToE2}, we can extend the result of Theorem \ref{T:nontrivialMorita} to many other graphs, described in the following corollary.

\begin{corollary}\label{cor:main-corollary-o1-pinwheel}
Let $\Sigma_{g}^{2r}$ be a genus $g$ surface with $2r$ boundary components, labeled $m_1, \cdots, m_r, n_1, \cdots, n_r$, and let $l_1, \cdots, l_j$ be disjoint simple closed loops on $\Sigma_{g}^{2r}$. Suppose that we are given a continuous map $$f: \Sigma_g^{2r}\to \Sigma_1^2$$ sending the $m_i$ isomorphically onto one of the boundary components of $\Sigma_1^2$ and the $n_i$ isomorphically onto the other boundary component, and sending the $l_i$ to points.

Let $$\Xi:= \left(\bigsqcup_{i\in \mathbb{Z}/d\mathbb{Z}} \Sigma_{g}^{2r}\right)/\sim$$ be the surface obtained by taking $d$ copies of  $\Sigma_{g}^{2r}$, indexed by $\mathbb{Z}/d\mathbb{Z}$, and identifying the boundary component $m_j$ on the $i$-th copy with the boundary component $n_j$ on the $i+1$-th copy. Let $S$ be the automorphism of this surface obtained by cyclically permuting the components and let $T$ be the Dehn multitwist about the curves $l_i$ (on all copies of  $\Sigma_{g}^{2r}$ in the disjoint union) and the images of the $m_i, n_j$. Let $\Gamma$ be the dual graph of $\Xi$ (with all of these marked curves). Then if $G$ is the genus of $\Xi$, the induced map $$\gamma: \langle S, T\rangle =\pi_1(\Sigma_1)\to \on{Mod}(G)$$ has the property that $\gamma^*o_1$ has order divisible by $d$.
\end{corollary}
See Figure \ref{Fig:C34} for an illustration of the dual graphs $\Gamma$ of the marked surfaces $\Xi$ constructed as above.

\begin{proof}
The map $f$ induces a map of surface bundles from the surface bundle in the statement of the Corollary to the one considered in Theorem \ref{T:nontrivialMorita}, over $\Sigma_1$. The result is immediate from Proposition \ref{prop:fromE1ToE2}. 
\end{proof}

\begin{figure}[h]
\labellist
\tiny
\pinlabel $1$ [bl] at 93.53 291.94
\pinlabel $1$ [tl] at 121.88 235.25
\pinlabel $1$ [tr] at 51.02 235.25
\pinlabel $1$ [bl] at 93.53 96.37
\pinlabel $1$ [bl] at 121.88 45.35
\pinlabel $1$ [br] at 51.02 45.35
\pinlabel $\epsilon_2$ [r] at 405.32 266.43
\pinlabel $\epsilon_1$ [r] at 263.60 266.43
\pinlabel $n_1$ [r] at 280.60 56.69
\pinlabel $m_1$ [r] at 419.49 56.69
\endlabellist
\includegraphics[scale=.5]{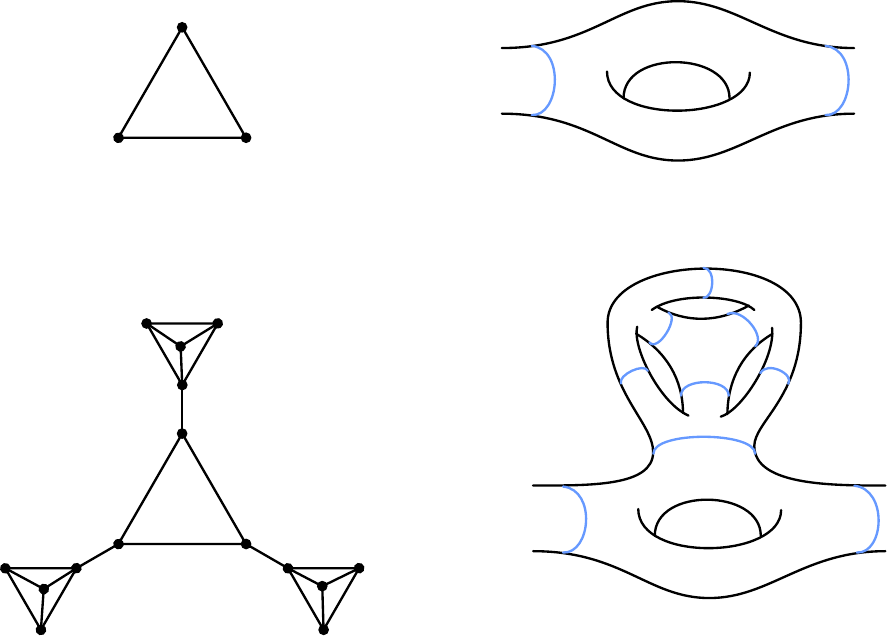}
\caption{The dual graph of a marked surface as in Corollary \ref{cor:main-corollary-o1-pinwheel}, with $r=1$. There is a map from the lower surface to the upper surface satisfying the conditions of the corollary, given by collapsing the upper subsurface to a point. The unlabeled vertices have genus $0$.}\label{Fig:C34}
\end{figure}
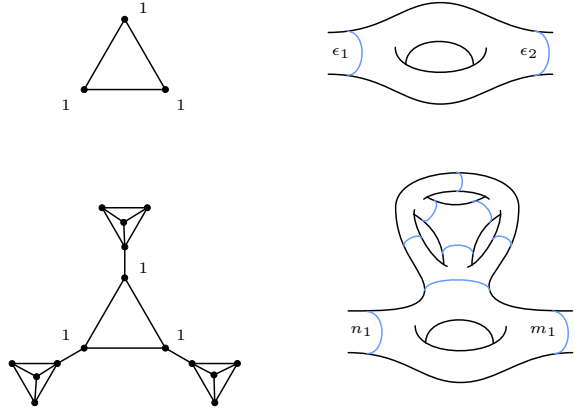

\subsection{Computing the secondary Morita classes for surface groups}\label{sec:computation-secondary-Morita}

We now consider situations where $o_1$ vanishes. Parallel to our analysis of $o_1$ in Section~\ref{S:MorSurfGrp}, we will now study pullbacks of the secondary Morita class $\widetilde{o_2}$ to surface groups.

Suppose we are given a homomorphism $\gamma_2: \pi_1(\Sigma_h) \to  \pi_1(\mathbf{Pic}^1_{\mathscr{C}_g/\mathscr{M}_g})$. 
\begin{remark}\label{rem:splitting-and-independence}
Note that as the Abel-Jacobi map induces a canonical isomorphism  $$\Mod(g,1)/L^2\pi_1(\Sigma_g)\simeq \pi_1(\mathbf{Pic}^1_{\mathscr{C}_g/\mathscr{M}_g})$$ so the data of a map $\gamma_2$ as above is the same as a map $\gamma: \pi_1(\Sigma_h)\to \text{Mod}(g)$ and a choice of splitting of the induced sequence $$1\to H_1(\Sigma_g, \mathbb{Z})\to \pi_1(\Sigma_h)\times_{\text{Mod}(g)} \Mod(g,1)/L^2\pi_1(\Sigma_g) \to \pi_1(\Sigma_h)\to 1.$$ Such a splitting exists if and only if $\gamma^*o_1=0$, by definition. By Proposition \ref{prop:obs-independent-of-section} and the definition of $\widetilde{o_2}$, the pullback $\gamma_2^*\widetilde{o_2}$ is independent of the given splitting. That is, $\gamma_2^*\widetilde{o_2}$ only depends on $\gamma$, not the choice of lift $\gamma_2$. Hence if we wish to be agnostic of the choice of lift we will denote it $\gamma^*\widetilde{o_2}$.
\end{remark}
Choose lifts $\widetilde{\gamma_2(a_1)},\widetilde{\gamma_2(b_1)}\cdots, \widetilde{\gamma_2(a_h)}, \widetilde{\gamma_2(b_h)}$ of $\gamma_2(a_i), \gamma_2(b_i)$ from $ \pi_1(\mathbf{Pic}^1_{\mathscr{C}_g/\mathscr{M}_g}) \simeq  \Mod(g,1)/L^2\pi_1(\Sigma_g)$ to $\text{Mod}(g,1)$. Then $$\widetilde{R}_2= \prod_{i=1}^h [\widetilde{\gamma_2(a_i)}, \widetilde{\gamma_2(b_i)}]$$ is an inner automorphism of $\pi_1(\Sigma_g)$, and hence can be written as conjugation by some element $\tilde r_2 \in \pi_1(\Sigma_g)$ as in Proposition \ref{P:MoritaSurfaceGroup}. Since $\gamma_2$ was a homomorphism, in fact $\tilde r_2 \in L^2 \pi_1(\Sigma_g)$.

To reduce notational clutter, we will write $\pi=\pi_1(\Sigma_g)$.

\begin{proposition}[The secondary Morita class for surface groups]\label{P:SecondaryMoritaSurfaceGroup} 
The secondary Morita class $\gamma_2^*\widetilde{o_2}$ is the image of $\tilde r_2$ under the map $$L^2\pi\to (L^2\pi/L^3\pi)_{\pi_1(\Sigma_h)}\overset{\sim}{\to} H^2(\pi_1(\Sigma_h), L^2\pi/L^3\pi)\to M(\pi_1(\Sigma_h), L^2\pi/L^3\pi),$$ where the first map is the natural quotient, the isomorphism comes from Corollary \ref{h2-corollary}, and the last map is the quotient map.
\end{proposition}

\begin{proof}
As in \ref{P:MoritaSurfaceGroup} we have the following commutative diagram:

$$\xymatrix{
1 \ar[r]  & 
\langle R \rangle \ar[r]\ar[d] &
F^{2h} \ar[r]\ar[d] &
\pi_1(\Sigma_h)\ar[r]\ar[d]^{\gamma_2}& 1 \\
1 \ar[r]  & 
L^2\pi \ar[r]\ar[d]^{\on{Hur}_2} &
\text{Mod}(g,1) \ar[r]\ar[d] &
\pi_1(\Pic^1) \ar[r]\ar[d]^{\simeq}& 1 \\
1 \ar[r] & 
L^2\pi/L^3\pi \ar[r] &
\text{Mod}(g,1)/L^3\pi \ar[r] &
\pi_1(\Pic^1) \ar[r]&
1.}$$
We conclude as in Proposition \ref{P:MoritaSurfaceGroup} that the class $\gamma_2^*o_2$ can be represented by $\Hur_2(\tilde{r}_2)$. By Remark \ref{rem:o2-functoriality}, the class $\gamma_2^*o_2 \in H^2(\pi_1(\Sigma_h),L^2\pi/L^3\pi)$ maps to $\gamma_2^*\widetilde{o_2} \in M(\pi_1(\Sigma_h),L^2\pi/L^3\pi)$.
\end{proof}


Now for each $g \in \mathbb{Z}_{\ge 1}$, we give a map $$\gamma: \pi_1(\Sigma_1)\to \Mod(2g)$$ such that the Morita class $\gamma^*o_1$ is trivial, and such that the secondary Morita class $\gamma_2^*\widetilde{o_2}$ is nontrivial, where $$\gamma_2: \pi_1(\Sigma_1)\to \pi_1(\mathbf{Pic}^1_{\mathscr{C}_g/\mathscr{M}_g})$$ is a lift of $\gamma$ as in Remark \ref{rem:splitting-and-independence}.

    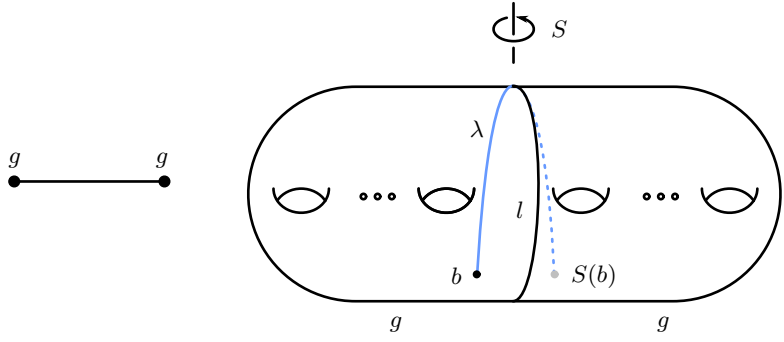
\begin{figure}[h]
    \labellist
    \small\pinlabel $g$ [b] at 2.83 53.85
\pinlabel $g$ [b] at 65.19 53.85
\pinlabel $g$ [t] at 164.39 -2.83
\pinlabel $g$ [t] at 277.77 -2.83
\pinlabel $S$ [l] at 226.75 116.21
\pinlabel $b$ [r] at 195.57 11.34
\pinlabel $S(b)$ [l] at 235.25 11.34
\pinlabel $\lambda$ [br] at 204.08 68.03
\pinlabel $l$ [tl] at 212.92 45.35
\endlabellist
	\centering
	\includegraphics[height=40mm]{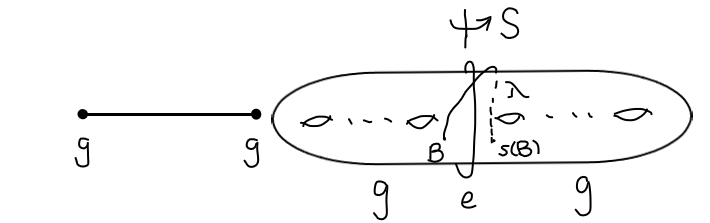}
	\caption{The surface discussed in Theorem \ref{P:nontrivialsecondaryMorita}}\label{fig:D_2g}
\end{figure}

Let $\Sigma_{2g}$ be a closed Riemann surface of genus $2g$. Let $l$ be a null-homologous closed curve which separates $\Sigma_{2g}$ into two subsurfaces each of genus $g$. Let $S$ be the order $2$ orientation-preserving mapping class that preserves $l$ and interchanges the two subsurfaces shown in Figure \ref{fig:D_2g}. Let $T$ be the Dehn twist around the loop $l$. Observe that $T$ and $S$ commute, and hence give rise to a map $$\gamma \colon \pi_1(\Sigma_1) = \langle a,b \ | \ [a,b] \rangle \rightarrow \text{Mod}(g),$$ defined by $a \mapsto T$ and $b \mapsto S$. This map gives rise to a $\Sigma_{2g}$-bundle over the torus $\Sigma_1$; we denote its total space by $E$ so that we have a fiber sequence $$\Sigma_{2g}\to E\to \Sigma_1.$$ We will denote by $G$ the group $\pi_1(\Sigma_1)$. 

\begin{lemma}\label{lemma:Morita-trivial-Dg}
The Morita class $\gamma^*o_{1}$ is trivial.
\end{lemma}

\begin{proof}
Let $B$ be a point in $\Sigma_{2g}\setminus l$ and let $\lambda$ be a path connecting $B$ and $S(B)$. We define lifts $\widetilde{T},\widetilde{S}$ of $T,S$ as described in Construction \ref{remark:liftsST}. By Lemma~\ref{commutatorcomputation}, the commutator $[\widetilde{T},\widetilde{S}]$ is an inner automorphism of $\pi_1(\Sigma_{2g},B)$ given as conjugation by $T(\lambda)\lambda^{-1}$.

By Corollary \ref{cor:reppullback}, the Morita class $\gamma^*o_{1}$ is represented by $\Hur(T(\lambda)\lambda^{-1})$, that is, the homology class of $l$. Since $l$ is null-homologous, it follows that that $\gamma^*o_{1}$ is trivial.
\end{proof}
\begin{remark}
We could also prove this by imitating the proof of Proposition \ref{prop:o1-vanishing-cpct-type}, i.e.~by contracting the loop $l$ above to a point.
\end{remark}
As observed in the proof above, the loop $T(\lambda)\lambda^{-1}$ is null-homologous. Thus the commutator
 $[\widetilde{T},\widetilde{S}] \in \Mod(g,1)$ is in the kernel of the natural map $\iota: \Mod(g,1) \to \pi_1(\Pic^1)$. So we have already constructed a map  $$\gamma_2: G\to \pi_1(\mathbf{Pic}^1_{\mathscr{C}_g/\mathscr{M}_g})$$ as in Remark \ref{rem:splitting-and-independence}, via $$\gamma_2(a) = \iota(\widetilde{T})$$ and $$\gamma_2(b) = \iota(\widetilde{S}).$$ In particular, $\widetilde{T},\widetilde{S}$ are lifts of $\gamma_2(a),\gamma_2(b)$.

    \begin{figure}[h]
	\centering
	\labellist
	\tiny
\pinlabel $\alpha_1$ [tr] at 124.71 85.03
\pinlabel $\beta_1$ [br] at 99.20 62.36
\pinlabel $\alpha_g$ [bl] at 42.52 76.53
\pinlabel $\beta_g$ [bl] at 6.17 48.18
\pinlabel $S(\alpha_1)$ [bl] at 178.57 79.36
\pinlabel $S(\beta_1)$ [bl] at 192.74 59.52
\pinlabel $S(\alpha_g)$ [tr] at 243.93 79.36
\pinlabel $S(\beta_g)$ [tl] at 239.43 14.17
\pinlabel $B$ [tr] at 124.71 11.34
\pinlabel $S(B)$ [tl] at 161.56 11.34
\pinlabel $l$ [b] at 144.55 93.53
\pinlabel $\lambda$ [l] at 138.88 48.18
\endlabellist
	\includegraphics[scale=1]{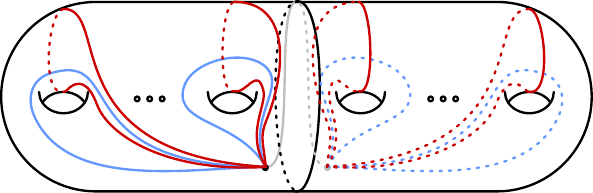}
	\caption{$\Sigma_{2g}$ with a set of generators of $\pi_1(\Sigma_{2g})$}\label{fig:D_2gRiemann}
\end{figure}
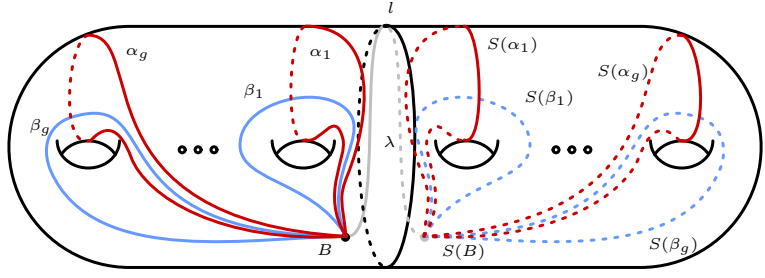

Let $B$ be a point in $\Sigma_{2g}\setminus l$ and let $\lambda$ be a path connecting $B$ and $S(B)$, as  shown in Figure \ref{fig:D_2gRiemann}. As in the picture, we choose $\lambda$ such that the loop $\lambda S(\lambda)$ is nullhomotopic. We choose a set of generators of $\pi_1(\Sigma_{2g})$, denoted $$\{\alpha_1,\beta_1, \ldots,\alpha_{2g},\beta_{2g} \},$$ in which $\alpha_1,\beta_1,\ldots,\alpha_g,\beta_g$ are drawn in Figure \ref{fig:D_2gRiemann}, and in which $\alpha_i = \lambda S(\alpha_{2g+1-i})\lambda^{-1},\beta_i = \lambda S(\beta_{2g+1-i})\lambda^{-1}$, for $g+1 \le i \le 2g$. We will use this basis for $\pi_1(\Sigma_{2g})$ for the rest of the section. The homology classes represented by these elements will be denoted by $x_1,y_1 \ldots, x_{2g},y_{2g}$ for $\alpha_1,\beta_1,\ldots,\alpha_{2g},\beta_{2g}$; they form a symplectic basis for $H_1(\Sigma_{2g},\mathbb{Z})$. 
 
 \begin{lemma}\label{lemma:gamma_2o_2of-order2}
 The class $\gamma_2^*o_2 \in H^2(G,L^2\pi_1(\Sigma_{2g})/L^3\pi_1(\Sigma_{2g}))$ has order $2$ (with $o_2$ defined as in Equation \eqref{eq:o2}).
 \end{lemma}
 
 \begin{proof}
By Proposition \ref{P:SecondaryMoritaSurfaceGroup}, the class $\gamma_2^*o_2$ is represented by $\Hur_2(T(\lambda)\lambda^{-1})$ (where $\Hur_2: L^2\pi_1(\Sigma_{2g})\to L^2\pi_1(\Sigma_{2g})/L^3\pi_1(\Sigma_{2g})$ is the natural quotient map).
Observe that the null-homologous loop $T(\lambda)\lambda^{-1}$ is represented by word $\prod_{i=1}^g [\alpha_i,\beta_i]$. By Lemma 
\ref{lem:free-Lie-alg-2nd-piece}, there is an isomorphism $$L^2\pi_1(\Sigma_{2g})/L^3\pi_1(\Sigma_{2g}) \simeq \wedge^2 H_1(\Sigma_{2g},\mathbb{Z})/\left\langle\sum_{i=1}^{2g} x_i\wedge y_i\right\rangle$$ induced by the map $[\alpha,\beta] \mapsto \Hur(\alpha)\wedge \Hur(\beta)$. The image of $\Hur_2(T(\lambda)\lambda^{-1})$ under  this isomorphism is $\sum_{i=1}^{g} x_i \wedge y_i$, or equivalently $ -\sum_{i=g+1}^{2g} x_i\wedge y_i$. 

Since $l$ is null-homologous, the Dehn twist $T$ acts trivially on $H_1(\Sigma_{2g},\mathbb{Z})$ and hence on $L^2\pi/L^3\pi$. Since $S$ exchanges the two subsurfaces into which $l$ separates $\Sigma_{2g}$, we have
$S(x_{i}) = x_{2g+1-i}$ and $S(y_{i}) = y_{2g+1-i}$. In particular, we have $$S\left(\sum_{i=1}^g x_i\wedge y_i\right)=\sum_{i=g+1}^{2g} x_i\wedge y_i=-\sum_{i=1}^{g} x_i \wedge y_i,$$ so the image of $\sum_{i=1}^g x_i\wedge y_i$ in $(L^2 \pi_1(\Sigma_{2g})/L^3 \pi_1(\Sigma_{2g}))_G$ is $2$-torsion.

We now show this element has order exactly $2$ in $(L^2 \pi_1(\Sigma_{2g})/L^3 \pi_1(\Sigma_{2g}))_G$, by constructing a $G$-equivariant map $L^2\pi_1(\Sigma_{2g})/L^3\pi_1(\Sigma_{2g}) \to \mathbb{Z}/2\mathbb{Z}$ such that the image of $\sum_{i=1}^g x_i\wedge y_i$ is nontrivial. 

As a $\mathbb{Z}$-module, $\wedge^2 H_1(\Sigma_{2g},\mathbb{Z})$ has basis $\langle x_i \wedge x_j, y_i \wedge y_j, 1\le i<j \le 2g,x_i\wedge y_j, 1\le i,j \le 2g\rangle$. 
Define a map $$\rho: \wedge^2 H_1(\Sigma_{2g},\mathbb{Z})/\langle\sum_{i=1}^{2g} x_i\wedge y_i\rangle \to \mathbb{Z}/2\mathbb{Z}$$ by $\rho(x_1\wedge y_1) = \rho(x_{2g}\wedge y_{2g}) = 1$ and $\rho(x_i\wedge y_j)=0$ for the other basis elements. Since $S(x_1\wedge y_1)=x_{2g}\wedge y_{2g} $, this map is $G$-equivariant. As $\rho(\sum_{i=1}^g x_i\wedge y_i)=1$,
we've shown that the image of $\Hur_2(T(\lambda)\lambda^{-1})$ in $(L^2 \pi_1(\Sigma_{2g})/L^3 \pi_1(\Sigma_{2g}))_G$ has order exactly $2$ as desired.
 \end{proof}
 
\begin{lemma}\label{lem:H^G-to-H_G}
The image of $H_1(\Sigma_{2g},\mathbb{Z})^G$ in 
$H_1(\Sigma_{2g}, \mathbb{Z})_G \otimes (\mathbb{Z}/2\mathbb{Z})$ is zero.
\end{lemma}

\begin{proof}
The action of $G$ on $H_1(\Sigma_{2g},\mathbb{Z})$ factors through $\mathbb{Z}/2\mathbb{Z}$. Now the result  follows from direct computation from the fact that $H_1(\Sigma_{2g},\mathbb{Z})$ is a direct sum of free  $\mathbb{Z}[\mathbb{Z}/2\mathbb{Z}]$-modules.
\end{proof}

\begin{theorem}\label{P:nontrivialsecondaryMorita} 
The secondary Morita class $\gamma_2^*\widetilde{o_2} \in M(G,L^2 \pi_1(\Sigma_{2g})/L^3 \pi_1(\Sigma_{2g}))$ has order exactly $2$.
\end{theorem}

\begin{proof}
In this proof, we will let $H=\pi_1(\Sigma_{2g})/L^2\pi_1(\Sigma_{2g})$ and $\omega = \sum_{i=1}^{2g}x_i\wedge y_i$. Then by Lemma \ref{lem:free-Lie-alg-2nd-piece} we have $L^2\pi_1(\Sigma_{2g})/L^3\pi_1(\Sigma_{2g}) = \wedge^2 H /\langle \omega \rangle$. From the proof of Lemma \ref{lemma:gamma_2o_2of-order2} we have that the element $[e]:= \sum_{i=1}^{g}x_i\wedge y_i \in \wedge^2 H$ represents the class $\gamma_2^*o_2 \in H^2(G,L^2\pi_1(\Sigma_{2g})/L^3\pi_1(\Sigma_{2g})) \simeq (\wedge^2 H /\langle \omega \rangle)_G$, and its image under the map $\rho$ defined in the proof of Lemma \ref{lemma:gamma_2o_2of-order2} is nontrivial.

Let the maps $g,h$ be as defined in Lemma \ref{lem:pass-to-H_G}.
We have the following commutative diagram, where the horizontal solid arrows  are the natural quotient maps:

\[
\xymatrix{ H^G \wedge H \ar[d] & & \\
\wedge^2 H \ar[d]^{g'}\ar[r] & \wedge^2H/\langle\omega\rangle \ar[d]^{g}\ar[r]^{} & (\wedge^2H/\langle\omega\rangle)_G \ar[ld]^{h}\ar[ldd]^{\rho} \ar[dd]\\
\wedge^2(H_G) \ar[r]\ar[d]^{w} & \wedge^2(H_G)/\langle \overline{\omega} \rangle \ar[d] &\\
\mathbb{Z}/2\mathbb{Z} \ar@{=}[r]& \mathbb{Z}/2\mathbb{Z} & \overline{H^2(G, \wedge^2H/\langle\omega\rangle)}. \ar@{.>}[l]^-{\rho'}
}
\]
We will define $\rho'$ later. Here $H^G\wedge H$ denotes the subspace of $\wedge^2H$ spanned by elements of the form $a\wedge b$ with $a\in H^G, b\in H$.

To show the image of $\gamma_2^*o_2$ in $M(G,L^2\pi_1(\Sigma_{2g})/L^3\pi_1(\Sigma_{2g}))$ is nontrivial, we will give explicit descriptions of the images of $m$ and $\delta$ in $(\wedge^2 H/\langle \omega\rangle)_G$ and show they go to $0$ in $\mathbb{Z}/2\mathbb{Z}$ under the map $\rho$ in the commutative diagram above. See Section \ref{subsec:nonabelian-cohomology} for the definition of the maps $m, \delta$.

\textbf{Analysis of the image of $m$:}
We start by analyzing the image of map $m$, defined as in Definition \ref{def:m}, using Lemma \ref{lem:compute-m}; we will freely use the notation and results of Section \ref{subsubsec:genus-1-computations}.
As $T$ acts on $H$ trivially, Equation \eqref{eq:cocycle-condition} implies that
for $\phi \in \Hom(P_1,H)$ to be in $\ker d_2^*$, we must have $\phi(e) \in H^G$.
Take $\phi,\psi \in \ker d_2^*$; then a representative of $m([\phi]\otimes[\psi])$ is given by equation \eqref{equ:m-for-h=1}:
\[ m(\phi \otimes \psi)(u)=\phi(e) \wedge \psi(f) - \phi(f) \wedge \psi(e). \]
This implies that any class in $\im m$ can be represented by an element of
$H^G \wedge H$. By Lemma \ref{lem:H^G-to-H_G}, the image of $H^G \wedge H \to \wedge^2(H_G)$ lies in $2(\wedge^2(H_G))$. Thus $H^G \wedge H$ goes to $0$ under the map $w \circ g'$, as desired.

In particular, we see that $\rho$ induces a natural map $$\rho': \overline{H^2(G, \wedge^2H/\langle \omega\rangle)}\to \mathbb{Z}/2\mathbb{Z}$$ fitting into the commutative diagram above.

\textbf{Analysis of the image of $\delta$:}
Now we analyze the map $$\delta: H^1(G, H)\to \overline{H^2(G,L^2\pi/L^3\pi)}$$ of Definition \ref{defn:independent-bdry-map}, using Lemma \ref{lem:compute-delta}; we wish to show that $\rho'\circ \delta=0$.  Note that using the resolution of the trivial $\mathbb{Z}[G]$-module described in Section \ref{subsubsec:genus-1-computations}, we have a natural surjection $$\ker(d_2^*:\Hom(P_1, H)\to \Hom(P_2, H))\to H^1(G,H).$$

Because the map $\rho'\circ \delta$ is linear by Proposition \ref{prop:non-linear-bdry-map}, it suffices to check that it vanishes on a set of generators of $\ker(d_2^*)$. As $T$ acts trivially on $H$ the condition that $\phi\in \Hom(P_1, H)$ lies in $\ker(d_2)$ is exactly the condition that $\phi(e)\in H^G$.

Now let $\phi \in \ker d_2^*$ represent a class $[\phi] \in H^1(G,H)$ such that $\phi(e)=x_i+x_{2n+1-i}$ (or $y_i+y_{2n+1-i}$) for some $1 \le i \le g$ and $\phi(f) = x_j$ (or $y_j$) for some $1 \le j \le 2g$. Such $\phi$ generate $\ker(d_2^*)$ by the previous paragraph. Then we may choose $\widetilde{\phi(e)} = \alpha_i\lambda S(\alpha_{i})\lambda^{-1}$ (resp. $\beta_i\lambda S(\beta_i)\lambda^{-1}$) and $\widetilde{\phi(f)}=\alpha_j$ (resp. $\beta_j$) for $1 \le j \le g$ or $\widetilde{\phi(f)}=\lambda S(\alpha_{2g+1-j})\lambda^{-1}$ (resp. $\lambda S(\beta_{2g+1-j})\lambda^{-1}$) for $g+1 \le j \le 2g$ to be lifts of $\phi(e),\phi(f)$ to $\pi_1(\Sigma_{2g})/L^3\pi_1(\Sigma_{2g})$.

By Lemma \ref{lem:compute-delta},  a representative in $L^2\pi_1(\Sigma_{2g})/L^3\pi_1(\Sigma_{2g})$ of the class $\delta([\phi])$ is given by
\[[\phi(e)^{-1}, \phi(f)]\widetilde{\phi(f)}^T(\widetilde{\phi(f)})^{-1} \widetilde{\phi(e)}(\widetilde{\phi(e)}^S)^{-1}. \]

Let $\widetilde{l} \in \pi_1(\Sigma_{2g},B)$ be a loop based at $B$ homologous to $l$. If $\widetilde{\phi(f)}=\alpha_j$ or $\beta_j$, then $T$ acts trivially on it and the element $\widetilde{\phi(f)}^T(\widetilde{\phi(f)})^{-1}$ is trivial. If $\widetilde{\phi(f)}=\lambda S(\alpha_{2g+1-j})\lambda^{-1}$ or $\lambda S(\beta_{2g+1-j})\lambda^{-1}$, then
$$\widetilde{\phi(f)}^T(\widetilde{\phi(f)})^{-1} =[\widetilde{\phi(f)},\widetilde{l}] \in L^3\pi_1(\Sigma_{2g})$$
as $\widetilde{l}\in L^2\pi_1(\Sigma_{2g})$. In both cases we have \[[\phi(e)^{-1}, \phi(f)]\widetilde{\phi(f)}^T(\widetilde{\phi(f)})^{-1} \widetilde{\phi(e)}(\widetilde{\phi(e)}^S)^{-1}=[\phi(e)^{-1}, \phi(f)]\widetilde{\phi(e)}(\widetilde{\phi(e)}^S)^{-1} \bmod L^3\pi_1(\Sigma_{2g}).\]

Now suppose $\widetilde{\phi(e)} = \alpha_i\lambda S(\alpha_{i})\lambda^{-1}$. Then 
\[\widetilde{\phi(e)}^S =\lambda S(\alpha_i) S(\lambda) \alpha_{i}S(\lambda^{-1})\lambda^{-1} = \lambda S(\alpha_i) S(\lambda) \alpha_{i}. \]
Thus 
\[\widetilde{\phi(e)}(\widetilde{\phi(e)}^S)^{-1} = \alpha_i\lambda S(\alpha_{i})\lambda^{-1} \alpha_i^{-1} S(\lambda)^{-1}S(\alpha_i)^{-1} \lambda^{-1}= [\alpha_i,\lambda S(\alpha_{i})\lambda^{-1}]. \]
So we have that
$$[\phi(e)^{-1}, \phi(f)]\widetilde{\phi(e)}(\widetilde{\phi(e)}^S)^{-1} =[\phi(e)^{-1}, \phi(f)] [\alpha_i,\lambda S(\alpha_{i})\lambda^{-1}] \bmod L^3\pi_1(\Sigma_g)$$
is a representative for $\delta([\phi])$. In additive notation, we've found that $$- \phi(e) \wedge \phi(f) + x_i\wedge x_{2g+1-i}\in \wedge^2 H$$ is a representative for $\delta([\phi])$. Since $\phi(e) \wedge \phi(f) \in H^G \wedge H$, it is sent to $0$ under the map $w \circ g'$. Since $S(x_i)=x_{2g+1-i}$, we have $g'(x_i\wedge x_{2g+1-i})=0$. Thus $\rho'(\delta([\phi]))=0$.  Now an identical argument works in the case  $\widetilde{\phi(e)}=\beta_i\lambda S(\beta_i)\lambda^{-1}$, from which we conclude the result.
\end{proof}

Note that the only property of the Dehn twist $T$ used in the proof of Theorem \ref{P:nontrivialsecondaryMorita} is it lies in the Torelli subgroup of $\Mod(g)$, i.e, it acts trivially on $H_1(\Sigma_{2g},\mathbb{Z})$. So an identical (if notationally more involved) proof yields:

\begin{theorem}\label{thm:tree-o2-nonvanishing}
Let $\Sigma_{2g}$ be a surface of genus $g$, and let  $l_1, \cdots, l_N$ be disjoint simple closed curves on $\Sigma_{2g}$. Suppose that the dual graph $\Gamma$ of this marked surface is a stable tree. Suppose moreover that $\Sigma_{2g}$ admits an involution $S$ permuting the $l_i$, such that the induced automorphism of $\Gamma$ has the following property: $S$ fixes no vertices and stabilizes exactly one edge. Let $E$ be the surface bundle over the torus induced by the automorphism $S$ and the Dehn multitwist about the $l_i$, and let $\gamma: \pi_1(\Sigma_1)\to \on{Mod}(g)$ be the induced map. Then $\gamma^*o_1=0$, and $\gamma^*\widetilde{o_2}$ has order exactly $2$.
\end{theorem}
\begin{remark}
Suppose $\Gamma$ is any graph admitting an involution $S$ which fixes no vertices and stabilizes a unique edge. Suppose moreover that this edge is separating. Let $\Sigma_\Gamma$ be the corresponding marked surface, and consider the $\Sigma_\Gamma$-bundle over the torus obtained from the involution $S$ and the Dehn multitwist about the marked curves. Then we expect the methods above will show that $\widetilde{o_2}$ obstructs sections for this surface bundle.
\end{remark}
\subsection{Analytic consequences}\label{M_g-consequences}


We now deduce from the computations of Sections \ref{S:MorSurfGrp} and \ref{sec:computation-secondary-Morita} the non-vanishing of certain Gysin images of the primary and secondary Morita classes. In this section, we work with complex-analytic stacks.
\subsubsection{Degeneration of the primary Morita class}
Let $$\gamma: \pi_1(\Sigma_1)=\langle a, b\mid [a,b]\rangle\to \text{Mod}(g)$$ be a homomorphism such that $\gamma(a)$ is a Dehn multi-twist along a collection of disjoint loops $l_1, \cdots, l_n\in \Sigma_g,$ and such that $\gamma(b)$ is (in the isotopy class of) a self-homeomorphism of $\Sigma_g$ permuting the $l_i$. Let $\Gamma$ be the dual graph of the marked surface $(\Sigma_g, l_1, \cdots, l_n)$. That is, $\Gamma$ is the labeled graph with one vertex for each component of $\Sigma_g\setminus \{l_1, \cdots, l_n\}$, labeled with the genus of this subsurface, and an edge between adjacent components for each shared boundary component $l_i$ (see Figure \ref{fig:graph-surface}). Suppose $\Gamma$ is a stable graph, and let $Z_\Gamma$ be the corresponding stratum of the boundary of $\mathscr{M}_g$. 

Let $$E_\Gamma\subset \on{Bl}_{\overline{Z_\Gamma}}\overline{\mathscr{M}_g}$$ be the exceptional divisor of the blowup of $\overline{\mathscr{M}_g}$ at $\overline{Z_\Gamma}$. Let $S_g$ be the set of stable graphs of genus $g$ with a single edge, and $S_{g,\Gamma}$ the set of stable graphs of genus $g$ with a single edge which specialize to $\Gamma$. That is, $S_g$ corresponds to the set of boundary divisors of $\overline{\mathscr{M}_g}$, and $S_{g, \Gamma}$ corresponds to the set of stable graphs with one edge obtainable from $\Gamma$ via contraction, or equivalently the set of boundary divisors of $\overline{\mathscr{M}_g}$ containing $Z_\Gamma$. Let $$\overline{\mathscr{M}_{g,\Gamma}}=\on{Bl}_{\overline{Z_\Gamma}}\overline{\mathscr{M}_g}\setminus\bigcup_{\Gamma'\in S_g\setminus S_{g,\Gamma}} \widetilde{D_{\Gamma'}}$$
and
$$E_\Gamma^\circ=E_\Gamma\cap \overline{\mathscr{M}_{g,\Gamma}},$$
where $D_{\Gamma'}$ is the boundary divisor of $\mathscr{M}_g$ corresponding to $\Gamma'$, and $\widetilde{D_{\Gamma'}}$ is its proper transform in $\on{Bl}_{\overline{Z_\Gamma}}\overline{\mathscr{M}_g}.$ That is, $\overline{\mathscr{M}_{g,\Gamma}}$ is the complement in $\on{Bl}_{\overline{Z_\Gamma}}\overline{\mathscr{M}_g}$ of the proper transforms of the boundary components not containing $Z_\Gamma$, and $E_\Gamma^\circ$ is the part of the exceptional divisor contained in this complement.

There is a natural inclusion $$\iota: \mathscr{M}_g\hookrightarrow \overline{\mathscr{M}_{g, \Gamma}},$$ with complement $E_\Gamma^\circ$. Now let $\mathscr{F}$ be a locally constant sheaf of Abelian groups on $\mathscr{M}_g$. Let $\widetilde{E_\Gamma^\circ}$ be a deleted neighborhood of $E_\Gamma^\circ$ in $\overline{\mathscr{M}_{g, \Gamma}}$ and $$\pi: \widetilde{E_\Gamma^\circ}\to E_\Gamma^\circ$$ the corresponding circle bundle. Recall from Section \ref{subsec:gysin} that there is a natural Gysin map $$g_\Gamma: H^2(\mathscr{M}_g, \mathscr{F})\to H^1(E_\Gamma^\circ, R^1\pi_*\mathscr{F}|_{\widetilde{E_\Gamma^\circ}}).$$
\begin{proposition}\label{prop:gysin-image-o1}
With the above notation, suppose $\gamma^*o_1\not=0$. Then $g_\Gamma(o_1)$ is non-zero of order divisible by that of $\gamma^*o_1.$
\end{proposition}
\begin{proof}
Recall that $$o_1\in H^2(\mathscr{M}_g, \mathbb{V}_1);$$ we will abuse notation and also denote by $\mathbb{V}_1$ the $\text{Mod}(g)$-representation corresponding to this local system.

By Lemmas \ref{lem:dehn-twist-inertia} and \ref{lem:pi1-factoring}, the map $$\gamma: \pi_1(\Sigma_1)\to \text{Mod}(g)$$ factors through $\pi_1(\widetilde{E_\Gamma^\circ})$, with $\gamma(a)$ a generator of the inertia subgroup $I_\Gamma\subset \pi_1(\widetilde{E_\Gamma^\circ})$ (that is, the subgroup of $\pi_1(\widetilde{E_\Gamma^\circ})$ generated by a fiber of $\pi$ --- namely, a Dehn multitwist).

We have a commutative diagram of short exact sequences of groups
$$\xymatrix{
0 \ar[r] & \langle a\rangle \ar[r] \ar[d] & \pi_1(\Sigma_1)\ar[r]\ar[d]^\gamma & \pi_1(\Sigma_1)/\langle a \rangle \ar[d] \ar[r] & 0\\
0\ar[r] & I_\Gamma\ar[r] & \pi_1(\widetilde{E_\Gamma^\circ}) \ar[r]^{\pi_*} & \pi_1(E_\Gamma^\circ)\ar[r] & 1.
}$$
Writing $\mathbb{Z}$ for $\pi_1(\Sigma_1)/\langle a \rangle$, this diagram induces a morphism of long exact sequences arising from the Hochschild-Serre spectral sequence (see Section \ref{subsec:gysin}) as follows :
$$\xymatrix{
\cdots \ar[r] & H^2(\pi_1(E_\Gamma^\circ), \mathbb{V}_1^{I_\Gamma}) \ar[r] \ar[d] & H^2(\pi_1(\widetilde{E_\Gamma^\circ}), \mathbb{V}_1) \ar[r]^{g_\Gamma} \ar[d]^{\gamma^*} & H^1(\pi_1(E_\Gamma^\circ), (\mathbb{V}_1)_{I_\Gamma}) \ar[r] \ar[d]^{\gamma^*} & H^3(\pi_1(E_\Gamma^\circ), \mathbb{V}_1^{I_\Gamma}) \ar[r] \ar[d]& \cdots  \\
\cdots \ar[r] & H^2(\mathbb{Z}, \gamma^*\mathbb{V}_1^{\langle a\rangle}) \ar[r]  & H^2(\pi_1(\Sigma_1), \gamma^*\mathbb{V}_1) \ar[r]^{p} & H^1(\mathbb{Z}, \gamma^*(\mathbb{V}_1)_{\langle a \rangle}) \ar[r] & H^3(\mathbb{Z}, \gamma^*\mathbb{V}_1^{\langle a\rangle})\ar[r] & \cdots .
}$$
Now $$H^2(\mathbb{Z}, \gamma^*\mathbb{V}_1^{\langle a\rangle})=H^3(\mathbb{Z}, \gamma^*\mathbb{V}_1^{\langle a\rangle})=0$$ for degree reasons, so the map $p$ in the diagram above is an isomorphism. Hence $$\gamma^*g_\Gamma(o_1)=p(\gamma^*o_1),$$ has the same order of $\gamma^*o_1$, which completes the proof.
\end{proof}
The following is immediate:
\begin{corollary}\label{cor:o1-gysin-corollary}
For the graphs $\Gamma$ in Corollary \ref{cor:main-corollary-o1-pinwheel}, $g_\Gamma(o_1)$ is non-zero.
\end{corollary}
For example, we have that $g_\Gamma(o_1)$ is non-zero for the graph $C_{g-1}$ described in Theorem \ref{thm:main-geometric-thm}, as well as the graphs depicted in Figures \ref{Fig:Genus4Max} and \ref{Fig:C34}.
\subsubsection{Degeneration of the secondary Morita class}\label{subsubsec:o2-degeneration}
We now analyze the analogous situation with $\widetilde{o_2}$, in cases that $\gamma^*o_1=0$. Our results and proofs are almost identical to those above, except insofar as there are additional complications arising from the fact that $\gamma^*\widetilde{o_2}$ resides in a (non-trivial) quotient of $H^2(\pi_1(\Sigma_1), \gamma^*\mathbb{V}_2),$ namely $M(\pi_1(\Sigma_1), \gamma^*\mathbb{V}_2),$ and from the fact that we are forced to work with the scheme $P_{1,g}$ rather than $\mathbf{Pic}^1_{\mathscr{C}_g/\mathscr{M}_g},$ due to the inadequacy of existing compactifications of this latter stack for our purposes. (Recall that $P_{d,g}$ coarsely represents the degree $d$ Picard functor over the locus of automorphism-free curves $M_g^0$; see Section \ref{subsec:pic-preliminaries} for details.)

As before, let $$\gamma: \pi_1(\Sigma_1)=\langle a, b\mid [a,b]\rangle\to \text{Mod}(g)$$ be a homomorphism such that $\gamma(a)$ is a Dehn multi-twist along a collection of disjoint loops $l_1, \cdots, l_n\in \Sigma_g,$ and such that $\gamma(b)$ is (in the isotopy class of) a self-homeomorphism of $\Sigma_g$ permuting the $l_i$. Let $\Gamma$ be the dual graph of the marked surface $(\Sigma_g, l_1, \cdots, l_n)$. Suppose $\Gamma$ is a stable tree, and let $Z_\Gamma$ be the corresponding stratum of the boundary of $\overline{M_g}^0$ (the locus in $\overline{M_g}$ parametrizing automorphism-free stable curves). Suppose now that $\gamma^*o_1=0$, so by Remark \ref{rem:o2-defn-remark} or Remark \ref{rem:splitting-and-independence}, we may define $\gamma^*\widetilde{o_2}$.

Let $q: \overline{P_{1, g}}\to \overline{M_g}$ be the canonical forgetful map, and let $$\overline{P_{1, g}}^0=q^{-1}(\overline{M_g}^0)$$ be the preimage of $\overline{M_g}^0$. Let $$\overline{P_{1, g, \Gamma}}:= \on{Bl}_{q^{-1}(\overline{Z_\Gamma})} \overline{P_{1, g}}^0$$ and let $F_\Gamma$ be the exceptional divisor. Let $\overline{P_{1, g, \Gamma}}^r$ be the regular locus of $\overline{P_{1, g, \Gamma}}^0$; as $\overline{P_{1, g}}$ is Cohen-Macaulay and $\overline{Z_\Gamma}$ is  an lci subscheme of $\overline{M_g}^0$, $F_\Gamma$ has non-empty intersection with $\overline{P_{1, g, \Gamma}}^r$ by Lemma \ref{lem:cohen-macaulay}. Note that $F_\Gamma$ is irreducible by Proposition \ref{prop:irreducible-fibers}. The map $q$ lifts to a natural map $$p: \overline{P_{1, g, \Gamma}}^r\to \on{Bl}_{\overline{Z_\Gamma}\cap \overline{M_g}^0}\overline{M_g}^0.$$ Let $F_\Gamma^\circ = p^{-1}(E_\Gamma^\circ),$ and let $$P_\Gamma=p^{-1}(\overline{M_{g, \Gamma}}\cap \on{Bl}_{\overline{Z_\Gamma}\cap \overline{M_g}^0}\overline{M_g}^0)\setminus F_{\Gamma}^{\circ, \text{sing}},$$ where $\overline{M_{g, \Gamma}}$ is the coarse space of $\overline{\mathscr{M}_{g, \Gamma}}$. Note that we have deleted the singular locus of $F_\Gamma^\circ$ from this scheme.

Observe that $P_{1,g}$ is a smooth open subscheme of $P_\Gamma$, and $F_{\Gamma}^{\circ, \text{ns}}$ is its complement; it is regular by definition. Let $\widetilde{F_{\Gamma}^{\circ, \text{ns}}}$ be a deleted neighborhood of $F_{\Gamma}^{\circ, \text{ns}}$, and let $$\pi: \widetilde{F_{\Gamma}^{\circ, \text{ns}}}\to {F_{\Gamma}^{\circ, \text{ns}}}$$ be the corresponding circle bundle. There is a natural Gysin map $$g_\Gamma: H^2(P_{1,g}, \mathbb{V}_2|_{P_{1,g}})\to H^1(F_\Gamma^{\circ, \text{ns}}, R^1\pi_* \mathbb{V}_2|_{\widetilde{F_{\Gamma}^{\circ, \text{ns}}}}).$$

Let $$h_\Gamma: M(P_{1,g}, \mathbb{V}_2|_{P_{1,g}})\to N(F_\Gamma^{\circ, \text{ns}}, R^1\pi_* \mathbb{V}_2|_{\widetilde{F_{\Gamma}^{\circ, \text{ns}}}})$$ be the induced map, where $N(F_\Gamma^{\circ, \text{ns}}, R^1\pi_* \mathbb{V}_2|_{\widetilde{F_{\Gamma}^{\circ, \text{ns}}}})$ is the maximal quotient of $H^1(F_\Gamma^{\circ, \text{ns}}, R^1\pi_* \mathbb{V}_2|_{\widetilde{F_{\Gamma}^{\circ, \text{ns}}}})$ such that such a factorization of $g_\Gamma$ through $M(P_{1,g}, \mathbb{V}_2|_{P_{1,g}})$ exists, as described in detail in Section \ref{subsubsec:defn-of-N}.
\begin{proposition}\label{prop:topological-o2-gysin}
With the above notation, suppose $$\gamma^*\widetilde{o_2}\in M(\pi_1(\Sigma_1), \gamma^*\mathbb{V}_2)$$ is non-zero. Then $h_\Gamma(\widetilde{o_2})\in N(F_\Gamma^{\circ, \text{ns}}, R^1\pi_* \mathbb{V}_2|_{\widetilde{F_{\gamma}^{\circ, \text{ns}}}})$ is non-zero of order divisible by that of $\gamma^*\widetilde{o_2}$.
\end{proposition}
\begin{proof}
The proof is essentially the same as that of Proposition \ref{prop:gysin-image-o1}. Indeed, let $$\tilde\gamma: \pi_1(\Sigma_1)\to \pi_1(\mathbf{Pic}^1_{\mathscr{C}_g/\mathscr{M}_g})$$ be a lift of $\gamma$; such a lift exists as $\gamma^*o_1=0$ by assumption. As before, $\tilde\gamma$ factors through $\pi_1(\widetilde{F_\Gamma^{\circ, \text{ns}}})$ by Lemmas \ref{lem:dehn-twist-inertia} and \ref{lem:pi1-factoring}. So we have a commutative square (arising from a map of Gysin sequences)
$$\xymatrix{
M(\pi_1(\widetilde{F_\Gamma^{\circ, \text{ns}}}), \mathbb{V}_2) \ar[r]^{h_\Gamma} \ar[d]^{\tilde\gamma^*} & N(\pi_1(F_\Gamma^{\circ, \text{ns}}), (\mathbb{V}_2)_{I_\Gamma})  \ar[d]^{\tilde\gamma^*} \\
M(\pi_1(\Sigma_1), \tilde\gamma^*\mathbb{V}_2) \ar[r]^{p} & N(\mathbb{Z}, \tilde\gamma(\mathbb{V}_2)_{\langle a \rangle}) 
}$$
But the map $p$ above is an isomorphism as $$H^2(\mathbb{Z}, \tilde\gamma^*\mathbb{V}_2^{\langle a\rangle})=H^3(\mathbb{Z}, \tilde\gamma^*\mathbb{V}_2^{\langle a\rangle})=0$$ for degree reasons. Now $$\tilde\gamma^*h_\Gamma(\widetilde{o_2})=p(\tilde\gamma^*\widetilde{o_2})$$ is non-zero by assumption, which completes the proof.
\end{proof}
We record a function-field analogue of this statement. Let $\widehat{S_\Gamma}$ be the fraction field of the complete local ring of $P_\Gamma$ at the generic point of $F_\Gamma^{\circ, \text{ns}}$, and let $T_\Gamma=\mathbb{C}(F_\Gamma^{\circ, \text{ns}})$ be its residue field. As before we have a Gysin map $$h_\Gamma: M(\widehat{S}_\Gamma, \widehat{V_2}|_{S_\Gamma})\to N(T_\Gamma, \widehat{\mathbb{V}_2}(-1)|_I)$$ (see \ref{subsec:gysin} for details on Galois-cohomological Gysin maps and Section \ref{subsubsec:defn-of-N} for a recollection of the group $N$). 
\begin{proposition}
For $\Gamma$ as in Proposition \ref{prop:topological-o2-gysin}, we have $$h_\Gamma(\widetilde{o_{2, \text{\'et}}})\not=0.$$
\end{proposition}
\begin{proof}
The proof is identical to that of Proposition \ref{prop:topological-o2-gysin}, replacing the use of Lemma \ref{lem:pi1-factoring} with Lemma \ref{lem:galois-factoring}.
\end{proof}
We immediately deduce:
\begin{corollary}\label{cor:o2-gysin-nonvanishing}
Let $\Gamma$ be as in Theorem \ref{thm:tree-o2-nonvanishing}. Then $h_\Gamma(\widetilde{o_2})$ (resp.~$h_\Gamma(\widetilde{o_{2, \text{\'et}}})$) is non-zero.
\end{corollary}
In particular, this non-vanishing holds for the graphs $T_g$ described in Theorem \ref{thm:main-geometric-thm}.
\section{Consequences for the section conjecture}\label{SSecConjApp}

\subsection{Geometric results}
We now deduce the main geometric results stated in the introduction from the Gysin computations performed in the previous section. Let $k$ be a field, possibly of positive characteristic, and let $g>2$ be an integer.

\subsubsection{Consequences arising from non-vanishing of $o_1$}
Recall that if $\Gamma$ was a stable graph of genus $g$, we defined $Z_\Gamma$ to be the corresponding stratum of the boundary of $\overline{\mathscr{M}_{g, k}}$ and $E_\Gamma$ to be the exceptional divisor of the blowup $\on{Bl}_{\overline{Z_\Gamma}}\overline{\mathscr{M}_{g,k}}$. We let $\widehat{L_\Gamma}$ be the fraction field of the complete local ring of $\on{Bl}_{\overline{Z_\Gamma}}\overline{\mathscr{M}_{g,k}}$ at the generic point of $E_\Gamma$. Recall from Proposition \ref{prop:gysin-image-o1} that we denoted the Gysin map into the cohomology of $E_\Gamma^\circ$ by $g_\Gamma$.
\begin{proposition}\label{prop:blowup-o1-generic-point}
Let $\Gamma$ be a stable graph of genus $g$ such that, over the complex numbers, $g_\Gamma(o_1)$ has order $d>1$. Then over a field of characteristic $0$, $o_{1, \text{\'et}}|_{\widehat{L_\Gamma}}$ has order divisible by $d$. Over a field of characteristic $p>0$ with $p$ not dividing $d$, $o_{1, \text{\'et}}^{(p)}|_{\widehat{L_\Gamma}}$  has order divisible by $d$.
\end{proposition}
\begin{proof}
The idea of the proof is to use the fact that for normal varieties the first \'etale cohomology group injects into the Galois cohomology of the generic point; the same is true for normal Deligne-Mumford stacks. We use the notation from Section \ref{M_g-consequences}. Let $$\iota:\mathscr{M}_g\hookrightarrow \overline{\mathscr{M}_{g, \Gamma}}$$ be the natural open embedding and let $$j: E_\Gamma^\circ\hookrightarrow \overline{\mathscr{M}_{g, \Gamma}}$$ its closed complement. 
In characteristic $0$, we have a commutative diagram of Gysin maps 
$$\xymatrix{
H^2(\mathscr{M}_g, \widehat{\mathbb{V}_1})\ar[r]^{g_\Gamma}\ar[d] & H^1(E_\Gamma^\circ, j^*R^1\pi_*\widehat{\mathbb{V}_1}) \ar@{^(->}[d]\\
H^2(\widehat{L_\Gamma}, \widehat{\mathbb{V}_1}|_{\widehat{L_\Gamma}}) \ar[r] & H^1(k(E_\Gamma^{\circ}), \widehat{\mathbb{V}_1}|_{\widehat{L_\Gamma}}(-1)|_I).
}$$
By comparison with the analytic setting $g_\Gamma(o_{1, \text{\'et}})$ is non-zero; now we conclude by the injectivity of the right-hand vertical arrow (as first \'etale cohomology of a smooth Deligne-Mumford stack injects into the Galois cohomology of its generic point).

The proof in characteristic $p>0$ is identical; by the argument above it suffices to show that for the Gysin map $$g_\Gamma: H^2(\mathscr{M}_{g,k}, \widehat{\mathbb{V}_1}^{(p)})\to H^1(E_\Gamma^{\circ}, j^*R^1\pi_*\widehat{\mathbb{V}_1}^{(p)}),$$ we have $g_\Gamma(o_{1, \text{\'et}}^{(p)})$ is non-zero. Now let $W(k)$ be the Witt vectors of $k$ and let $K$ be the fraction field of $W(k)$. We have a commutative diagram
$$\xymatrix{
H^2(\mathscr{M}_{g,k}, \widehat{\mathbb{V}_1}^{(p)})\ar[r]^{g_\Gamma}\ar[d] & H^1(E_{\Gamma, k}^\circ, j^*R^1\pi_*\widehat{\mathbb{V}_1}^{(p)}) \ar[d]\\
H^2(\mathscr{M}_{g,K}, \widehat{\mathbb{V}_1}^{(p)})\ar[r]^{g_\Gamma}& H^1(E_{\Gamma, K}^\circ, j^*R^1\pi_*\widehat{\mathbb{V}_1}^{(p)}) 
}$$
where the vertical arrows are cospecialization maps, whence the result follows from the characteristic $0$ situation.
\end{proof}
\begin{corollary}\label{cor:o1-tropical-section-conjecture}
Let $\Gamma'$ be a graph specializing to one of the graphs $\Gamma$ appearing in Corollary \ref{cor:main-corollary-o1-pinwheel}, and let $k$ be a field. Let $d$ be as in Corollary \ref{cor:main-corollary-o1-pinwheel}. Then the tropical section conjecture (Conjecture \ref{conj:tropical-section-conjecture}) is true for $\Gamma', k$ as long as $\text{char}(k)$ does not divide $d$. In fact the \emph{abelianized} fundamental exact sequence $$1\to \pi_1^{\text{\'et}}(\mathscr{C}_{\overline{\widehat{K_{\Gamma'}}}})^{\text{ab}}\to \pi_1^{\text{\'et}}(\mathscr{C}_{\widehat{K_{\Gamma'}}})/L^2\pi_1^{\text{\'et}}(\mathscr{C}_{\overline{\widehat{K_{\Gamma'}}}})\to \on{Gal}(\overline{\widehat{K_{\Gamma'}}}/\widehat{K_{\Gamma'}})\to 1$$ does not split.
\end{corollary}
\begin{proof}
It suffices to show that $o_{1, \text{\'et}}|_{\widehat{K_{\Gamma'}}}$ is non-zero in characteristic $0$, and $o_{1, \text{\'et}}^{(p)}|_{\widehat{K_{\Gamma'}}}$ is non-zero in characteristic $p$. By assumption there is a natural map $\gamma: \widehat{K_{\Gamma'}}\to \widehat{L_\Gamma},$ with $\widehat{L_\Gamma}$ defined as in Proposition \ref{prop:blowup-o1-generic-point}. By functoriality we have $$\gamma^*o_{1, \text{\'et}}|_{\widehat{K_{\Gamma'}}}=o_{1, \text{\'et}}|_{\widehat{L_\Gamma}}$$ in characteristic $0$ and $\gamma^*o_{1, \text{\'et}}^{(p)}|_{\widehat{K_{\Gamma'}}}=o_{1, \text{\'et}}^{(p)}|_{\widehat{L_\Gamma}}$ in characteristic $p$, so it suffices to show that $o_{1, \text{\'et}}|_{\widehat{L_\Gamma}}$ (resp.~$o_{1, \text{\'et}}^{(p)}|_{\widehat{L_\Gamma}}$) is non-zero. But this follows by combining Proposition \ref{prop:blowup-o1-generic-point} with Corollary \ref{cor:o1-gysin-corollary}. 
\end{proof}
\begin{corollary}\label{cor:generic-point-o1-nonvanishing}
Suppose $\on{char}(k)=0$ or $\on{char}(k)> g-1.$ Let $L/k(\mathscr{M}_g)$ be an extension of degree not divisible by $g-1$. Then the class $o_{1,\text{\'et}}|_{{L}}$ is non-zero. That is, the abelianized fundamental exact sequence $$1\to \pi_1^{\text{\'et}}(\mathscr{C}_{g, \overline{k(\mathscr{M}_g)}})^{\text{ab}}\to \pi_1^{\text{\'et}}(\mathscr{C}_{g,k(\mathscr{M}_g)})/[\pi_1^{\text{\'et}}(\mathscr{C}_{g, \overline{k(\mathscr{M}_g)}},\pi_1^{\text{\'et}}(\mathscr{C}_{g, \overline{k(\mathscr{M}_g)}}]\to \text{Gal}(\overline{k(\mathscr{M}_g)}/L)\to 1$$ does not split. In particular, the section conjecture is ``trivially true" for the base change of the generic curve to $L$.
\end{corollary}
\begin{proof}
This is immediate from Corollary \ref{cor:o1-tropical-section-conjecture} applied in the case where $\Gamma'$ is the graph consisting of a single vertex of genus $g$.
\end{proof}
\begin{remark}
In \cite{hain2011rational}, Hain proves (among other things) that the section conjecture is true for the generic curve of genus $g\geq 5$ over a field of characteristic $0$. Our result Corollary \ref{cor:generic-point-o1-nonvanishing} refines this result by showing that in fact the \emph{abelianized} analogue of the fundamental exact sequence does not split (and that the result is in fact true for all $g\geq 3$).
\end{remark}
\begin{remark}\label{R:periodindex}
In fact the estimates of Theorem \ref{T:nontrivialMorita} imply that the class $o_{1, \text{\'et}}|_{k(\mathscr{M}_g)}$ has order divisible by $g-1$. Because of the relationship between the class $o_{1, \text{\'et}}$ and the class $[\mathbf{Pic}^1_{\mathscr{C}_g/\mathscr{M}_g}]\in H^1(\mathscr{M}_{g,S,\text{\'et}}, J_{\mathscr{C}_g/\mathscr{M}_g, S})$ discussed in \ref{subsubsec:Pic-construction-o1}, this estimate gives a lower bound on the \emph{period} of the generic curve of genus $g$; namely the period of the generic curve is divisible by $g-1$. In fact it is known (by the main result of \cite{schroer2003strong}) that the period of the generic curve over any field is $2g-2$; see \cite{qixiao-period} for an explicit statement. Over fields for which it applies, Corollary \ref{cor:generic-point-o1-nonvanishing} implies that, if $C_{\text{gen}}$ is the generic curve, the class of $[\mathbf{Pic}^1_{C_{\text{gen}}}]$ is not divisible in the Weil-Chatelet group of $\mathbf{Pic}^0_{C_{\text{gen}}}$, which was not to our knowledge previously known. It is natural from the point of view of the section conjecture to ask if $[\mathbf{Pic}^1_{C_{\text{gen}}}]$ in fact generates the quotient of the Weil-Chatelet group by its divisible part, and what its order is in this group. Of course it is also natural to ask what is the true order of $o_{1, \text{\'et}}|_{k(M_g)}$. Our degeneration methods are related to those of Ma \cite{qixiao-period}. See e.g.~\cite{lichtenbaum1968period, lichtenbaum1969duality} for a discussion of the period-index problem for curves. 

We get similar bounds for the period of the curves $\mathscr{C}_{\widehat{K_\Gamma}}$ above. Analogously, it would be interesting to study the Picard groups of the curves  $\mathscr{C}_{\widehat{K_\Gamma}}$.
\end{remark}
\subsubsection{Consequences of the non-vanishing of $\widetilde{o_2}$}
We now perform a similar analysis with the class $\widetilde{o_2}$. The arguments are almost identical.
\begin{proposition}\label{prop:o2-generic-point-of-blowup}
Let $k$ be a field of characteristic different from $2$. Let $\widehat{S_\Gamma}$ be the fraction field of the complete local ring of $P_{\Gamma}$ at the generic point of $F_\Gamma^{\circ, \text{ns}}$, defined as in Section \ref{subsubsec:o2-degeneration}. Then if $k$ has characteristic $0$, $\widetilde{o_{2, \text{\'et}}}|_{\widehat{S_\Gamma}}$ is non-zero for $\Gamma$ as in Theorem \ref{thm:tree-o2-nonvanishing}; if $k$ has characterstic $p>2$, then $\widetilde{o_{2, \text{\'et}}}^{(p)}$ is non-zero.
\end{proposition}\begin{proof}
The statement in characteristic $0$ is immediate from Corollary \ref{cor:o2-gysin-nonvanishing}; if $k$ has characteristic $p>0$ it follows as in the proof of Proposition \ref{prop:blowup-o1-generic-point}.
\end{proof}
\begin{corollary}\label{cor:tropical-section-on-pic}
Let $k$ be a field of characteristic different from $2$. Let $\Gamma'$ be a graph specializing to one of the graphs appearing in Theorem \ref{thm:tree-o2-nonvanishing}. Let $\widehat{Q_{\Gamma'}}$ be the fraction field of the complete local ring of $\overline{P_{1,g}}$ at $\pi^{-1}(\overline{Z_{\Gamma'}})$. Then if $\mathscr{C}_{g,\widehat{Q_{\Gamma'}}}$ is the base change of the universal curve to $\widehat{Q_{\Gamma'}}$, the sequence $$1\to \pi_1^{\text{\'et}}(\mathscr{C}_{g,\overline{\widehat{Q_{\Gamma'}}}})/L^3\pi_1^{\text{\'et}}(\mathscr{C}_{g,\overline{\widehat{Q_{\Gamma'}}}})\to \pi_1^{\text{\'et}}(\mathscr{C}_{g,{\widehat{Q_{\Gamma'}}}})/L^3\pi_1^{\text{\'et}}(\mathscr{C}_{g,\overline{\widehat{Q_{\Gamma'}}}})\to \on{Gal}(\overline{\widehat{Q_{\Gamma'}}}/{\widehat{Q_{\Gamma'}}})\to 1$$ does not split.
\end{corollary}
\begin{proof}
It suffices to show $\widetilde{o_{2,\text{\'et}}}|_{\widehat{Q_{\Gamma'}}}$ is non-zero. But by Proposition \ref{prop:o2-generic-point-of-blowup}, its pullback to $\widehat{S_\Gamma}$ is non-zero, as desired.
\end{proof}
We immediately deduce:
\begin{corollary}
Let $\Gamma'$ be a graph specializing to one of the graphs appearing in Theorem \ref{thm:tree-o2-nonvanishing}. Then the tropical section conjecture (Conjecture \ref{conj:tropical-section-conjecture}) is true for $\Gamma'$ over fields of characteristic different from $2$. In fact the sequence $$1\to \pi_1^{\text{\'et}}(\mathscr{C}_{g,\overline{\widehat{K_{\Gamma'}}}})/L^3\pi_1^{\text{\'et}}(\mathscr{C}_{g,\overline{\widehat{K_{\Gamma'}}}})\to \pi_1^{\text{\'et}}(\mathscr{C}_{g,{\widehat{K_{\Gamma'}}}})/L^3\pi_1^{\text{\'et}}(\mathscr{C}_{g,\overline{\widehat{K_{\Gamma'}}}})\to \on{Gal}(\overline{\widehat{K_{\Gamma'}}}/{\widehat{K_{\Gamma'}}})\to 1$$ does not split.
\end{corollary}
\begin{proof}
We have just shown that $\widetilde{o_{2, \text{\'et}}}$ (resp.~$\widetilde{o_{2, \text{\'et}}}^{(p)}$) does not vanish after pulling back to $\widehat{Q_{\Gamma'}}$; the result is immediate.
\end{proof}
Finally, we have the following simple corollary:
\begin{corollary}\label{cor:generic-point-of-pic}
Let $k$ be a field of characteristic different from $2$, and let $Q$ be the field of meromorphic functions on $P_{1,g,k}$. Then if $g>2$ is even and if $\mathscr{C}_{g,Q}$ is the base change of the universal curve to $Q$, the sequence 
$$1\to \pi_1^{\text{\'et}}(\mathscr{C}_{g,\overline{Q}})/L^3\pi_1^{\text{\'et}}(\mathscr{C}_{g,\overline{Q}})\to \pi_1^{\text{\'et}}(\mathscr{C}_{g,{Q}})/L^3\pi_1^{\text{\'et}}(\mathscr{C}_{g,\overline{Q}})\to \on{Gal}(\overline{Q}/{Q})\to 1$$
does not split.
\end{corollary}
\begin{proof}
This is immediate from Corollary \ref{cor:tropical-section-on-pic} for the case of the graph consisting only of a single vertex.
\end{proof}

\subsection{Arithmetic results} 
We now use the results above to show the existence of \emph{arithmetic} examples of curves trivially satisfying the section conjecture, over $p$-adic fields and then number fields. In this section we work over $\mathbb{Z}$. We first show the existence of examples such that $o_{1, \text{\'et}}$ obstructs sections. For a closed point $z$ of a scheme we denote by $\kappa(z)$ its residue field. 
\begin{theorem} \label{thm:examples-o1-p-adic}
Let $\Gamma$ be a graph as in Corollary \ref{cor:main-corollary-o1-pinwheel}. There exists a Zariski-dense set $S$ of closed points of $Z_\Gamma$ such that: for each $s\in S$, there exists a $\on{Frac}(W(\kappa(s)))$-point $s'$ of $\mathscr{M}_g$ specializing to $s$, such that the corresponding curve $\mathscr{C}_{g, s'}$ trivially satisfies the section conjecture (indeed $o_{1,\text{\'et}}|_{s'}$ is non-vanishing). 
\end{theorem}
\begin{proof}
Notation is as in Proposition \ref{prop:blowup-o1-generic-point}. We first observe that for $p\gg g$, we have $$g_\Gamma(o_{1, \text{\'et}}^{(p)})\not=0\in H^1(E_{\Gamma, \mathbb{F}_p}^{\circ}, j^*R^1\pi_* \widehat{\mathbb{V}_1}^{(p)}),$$ by Proposition \ref{prop:blowup-o1-generic-point} and Corollary \ref{cor:main-corollary-o1-pinwheel}. The restriction of this class to any open subscheme of $E_{\Gamma, \mathbb{F}_p}^{\circ}$ is non-zero by Proposition \ref{prop:blowup-o1-generic-point} as well. Hence by Theorem \ref{modified-chebotarev} (applied after replacing $E_{\Gamma, \mathbb{F}_p}^{\circ}$ with an open substack representable by a scheme, and the map $\gamma^*$ defined in the proof of Proposition \ref{prop:gysin-image-o1}), there exists a Zariski-dense set of points $S_p$ of $E_{\Gamma, \mathbb{F}_p}^{\circ}$ such that $o_{1, \text{\'et}}^{(p)}|_s$ is non-zero for $s\in S_p$. Let $S$ be the union of the images of the $S_p$ (over all $p$) in $Z_\Gamma$.

Now for $s\in S_p$, let $s'$ be any deformation of $s$ into $\mathscr{M}_g$, over $W(\kappa(s))$, which is transverse to $E_{\Gamma}^{\circ}$ (i.e.~a local equation for $E_{\Gamma}^{\circ}$ pulls back to a uniformizer of $W(\kappa(s))$). Such a lift exists as $\on{Bl}_{\overline{Z_\Gamma}}\overline{\mathscr{M}_g}$ is smooth. Now we have a commutative diagram of Gysin maps
$$\xymatrix{
H^2({\mathscr{M}_{g, W(\kappa(s))}}, \widehat{\mathbb{V}_1}^{(p)})\ar[r]^{g_\Gamma} \ar[d] & H^1(E_{\Gamma, \mathbb{F}_p}^{\circ}, j^*R^1\pi_* \widehat{\mathbb{V}_1}^{(p)})\ar[d] \\
H^2(\on{Frac}(W(\kappa(s))), \widehat{\mathbb{V}_1}^{(p)}|_{\on{Frac}(W(\kappa(s)))})\ar[r] & H^1(\kappa(s), j^*R^1\pi_* \widehat{\mathbb{V}_1}^{(p)}|_{\kappa(s)})
}$$
where the vertical arrows are restriction. By our choice of $s$, we have $g_\Gamma(o_{1, \text{\'et}}^{(p)})|_{\kappa(s)}$ non-zero. Hence $o_{1, \text{\'et}}^{(p)}|_{s'}$ is non-zero, as desired.
\end{proof}
\begin{remark}
The method above could be used equally well to show the existence of examples of curves over e.g.~$\mathbb{F}_q((t))$ satisfying the section conjecture or indeed to show the existence of examples over any complete discrete valuation ring with finite residue field.
\end{remark}
We now use an essentially identical argument to product examples where $\widetilde{o_{2,\text{\'et}}}$ obstructs sections, with some mild complications arising from the fact that $\widetilde{o_{2,\text{\'et}}}$ is not a cohomology class, but rather a coset of such:
\begin{theorem}\label{thm:example-o2-p-adic}
Let $\Gamma$ be a graph as in Theorem \ref{thm:tree-o2-nonvanishing}. There exists a Zariski-dense set $S$ of closed points of $F_{\Gamma}^{\circ,\text{ns}}$ such that: for each $s\in S$, there exists a $\on{Frac}(W(\kappa(s)))$-point $s'$ of $P_{1,g}$ specializing to $s$, such that the corresponding curve $\mathscr{C}_{g, s'}$ trivially satisfies the section conjecture (indeed $o_{2,\text{\'et}}|_{s'}$ is non-vanishing). 
\end{theorem}
\begin{proof}
By Corollary \ref{cor:o2-gysin-nonvanishing}, we have that for $p\gg 0$, $$h_\Gamma(\widetilde{o_{2, \text{\'et}}}^{(p)})\in N(F_{\Gamma, \mathbb{F}_p}^{\circ, \text{ns}}, i^*R^1j_*\widehat{\mathbb{V}_2}^{(p)})$$ is non-zero, where $j: P_{1,g}\to P_\Gamma$ is the natural inclusion and $i: F_\Gamma^{\circ, \text{ns}}\to P_\Gamma$ is the inclusion of its complement. Hence the same is true for $$g_\Gamma(o_{2, \text{\'et}}^{(p)})\in H^1(F_{\Gamma, \mathbb{F}_p}^{\circ, \text{ns}}, i^*R^1j_*\widehat{\mathbb{V}_2}^{(p)}),$$ where $o_{2, \text{\'et}}^{(p)}$ is defined as in Section \ref{subsubsec:etale-o2}. Let $S_p$ be the set of closed points of $F_\Gamma^{\circ,\text{ns}}$ such that for $s\in S$, $$g_\Gamma(o_{2, \text{\'et}}^{(p)})|_{\kappa(s)}\not=0;$$ by Theorem \ref{modified-chebotarev} applied to the map $\Tilde{\gamma}^*$ defined in the proof of Proposition \ref{prop:topological-o2-gysin}, this set is Zariski-dense in $F_{\Gamma, \mathbb{F}_p}^{\circ, \text{ns}}.$ As before let $S$ be the union of the $S_p$ and, for each $s\in S$, let $s'$ be a deformation of $s$ to a $\on{Frac}(W(\kappa(s)))$-point of $P_{1,g}$, transverse to $F_\Gamma^{\circ, \text{ns}}$. Such a deformation exists by the smoothness of $P_\Gamma$.

Now let $L=\on{Frac}(W(\kappa(s)))$ and $\mathscr{O}_L=W(\kappa(s))$. We wish to show that for an $L$-point $s'$ as above, we have $\widetilde{o_{2, \text{\'et}}}^{(p)}|_{s'}\not=0$. We have a commutative diagram of Gysin maps 
$$\xymatrix{
M(P_{1,g, \mathscr{O}_L}, \widehat{\mathbb{V}_2}^{(p)})\ar[r]^{h_\Gamma} \ar[d] & N(F_{\Gamma, \mathscr{O}_L}^{\circ, \text{ns}}, j^*R^1\pi_* \widehat{\mathbb{V}_2}^{(p)})\ar[d] \\
M(L, \widehat{\mathbb{V}_2}^{(p)}|_{L})\ar[r] & N(\kappa(s), j^*R^1\pi_* \widehat{\mathbb{V}_2}^{(p)}|_{\kappa(s)})
}$$
By assumption, $g_\Gamma(o_{2, \text{\'et}}^{(p)})|_{\kappa(s)}\not=0;$ it thus suffices to show that this class is not annihilated in the passage from $H^1(\kappa(s), j^*R^1\pi_* \widehat{\mathbb{V}_2}^{(p)}|_{\kappa(s)})$ to $N(\kappa(s), j^*R^1\pi_* \widehat{\mathbb{V}_2}^{(p)}|_{\kappa(s)})$.

As $h_\Gamma(o_{2, \text{\'et}}^{(p)})\not=0$ by assumption, it suffices to show that the natural map $$H^1(P_{1,g, \mathscr{O}_L}, \widehat{\mathbb{V}_1}^{(p)})\to H^1(L, \widehat{\mathbb{V}_1}^{(p)}|_L)$$ is surjective, by the definition of $N$ (see Section \ref{subsubsec:defn-of-N}). Now $$H^1(P_{1,g, \mathscr{O}_L}, \widehat{\mathbb{V}_1}^{(p)})\to H^1(\mathscr{O}_L, \widehat{\mathbb{V}_1}^{(p)}|_{\mathscr{O}_L})$$ is surjective because it has a section induced by the structure map $P_{1, g, \mathscr{O}_L}\to \on{Spec}(\mathscr{O}_L)$. Thus it suffices to show that the natural map $$H^1(\mathscr{O}_L, \widehat{\mathbb{V}_1}^{(p)}|_{\mathscr{O}_L})\to H^1(L, \widehat{\mathbb{V}_1}^{(p)}|_{L})$$ is surjective. But this follows from the inflation-restriction exact sequence; if $I\subset G_L:=\on{Gal}(\overline{L}/L)$ is the inertia subgroup, the cokernel of the above map injects into $$H^1(I,\widehat{\mathbb{V}_1}^{(p)})^{G_L/I}.$$ By assumption $\Gamma$ is a stable tree, so $I$ acts trivially on $\widehat{\mathbb{V}_1}^{(p)}.$ Thus this group is simply $\on{Hom}_{\text{cts}}(I^{\text{ab}}, \widehat{\mathbb{V}_1}^{(p)})^{G_L/I}.$ But this last vanishes for weight reasons; $I^{\text{ab}}$ has weight $-2$ and $\widehat{\mathbb{V}_1}^{(p)}$ has weight $-1$ (again as $\Gamma$ is a tree).
\end{proof}
\begin{remark}\label{rem:approximation-remark}
One may immediately use the above theorems to construct examples of curves over number fields for which the existence of $\pi_1$-sections is obstructed by $o_{1, \text{\'et}}$ (resp.~$\widetilde{o_{2,\text{\'et}}}$) by algebraization and Artin approximation. See e.g.~\cite[7.5]{stix2010period} for an explanation.
\end{remark}
\section{Appendix: Group cohomology constructions and computations}\label{sec:Non-abelian-cohomology}
\subsection{Obstructions arising from extensions by a $2$-nilpotent group}\label{subsec:nonabelian-cohomology}
Suppose we are given a short exact sequence of continuous maps of (not necessarily commutative) discrete or pro-finite groups
\[ 1 \to \pi \to \tilde{\pi} \to G \to 1.  \]
Then conjugation induces an outer action of $G$ on $\pi$. 

Let $\pi = L^1\pi \supset L^2\pi \supset \ldots$, where $L^{k+1}\pi = \overline{[\pi,L^k\pi]}$, be the lower central series of $\pi$. 
\subsubsection{Non-abelian cohomology computations}
Consider the sequences
\begin{equation}\label{seq:nilpotent_seq} 0 \to L^2\pi/L^3\pi \to  \pi/L^3\pi \to  \pi/L^2\pi \to 0 ,\end{equation}
\begin{equation} \label{seq:abelian_seq} 0\to \pi/L^2\pi\to \tilde\pi/L^2\pi\to G\to 1, \end{equation}
and 
\begin{equation} \label{seq:l2modl3extension}
0\to L^2\pi/L^3\pi\to \tilde\pi/L^3\pi\to \tilde\pi/L^2\pi\to 1.\end{equation}

\begin{definition}\label{def:EquivSections}
Let $$1 \to A \to B \to C \to 1$$ be a split exact sequence of groups, with $A$ abelian.
For two sections $s_1,s_2: C \to B$, we say they are equivalent if there exists some element $a \in A$ such that $s_1(c) = a s_2(c) a^{-1}$ for any $c\in C$.
\end{definition}

\begin{lemma}\label{lem:splittings-torsor}
Let $$1 \to A \to B \to C \to 1$$ be a exact sequence of (discrete or pro-finite) groups, with $A$ abelian. Then the set of continuous sections $s:C \to B$ up to equivalence is, if non-empty, canonically a torsor for $H^1(C, A)$.
\end{lemma}

\begin{proof}
This is \cite[I.5, Exercise 4]{neukirch2013cohomology}, taking $f=\text{id}, G=G'$.
\end{proof}

\begin{proposition}\label{prop:section-H1-torsor}
The set of continuous sections to sequence \ref{seq:abelian_seq} up to equivalence is, if non-empty, canonically a torsor for $H^1(G, \pi/L^2\pi)$.
\end{proposition}

\begin{proof}
Immediate from Lemma \ref{lem:splittings-torsor}.
\end{proof}

\begin{definition}\label{def:boundary}
Suppose sequence \ref{seq:abelian_seq} admits a splitting $s$, inducing an action of $G$ on $\pi/L^3\pi$. Let $$\delta_s: H^1(G, \pi/L^2\pi)\to H^2(G, L^2\pi/L^3\pi)$$ be the boundary map in non-abelian cohomology arising from sequence \ref{seq:nilpotent_seq}. Concretely, for a cocycle $x: G \to \pi/L^2\pi$, we lift it to a continuous map $\tilde{x}:G \to \pi/L^3\pi$ and define a cocycle representing the class $\delta_s([x])$ as 
\begin{equation}\label{eq:cocycle-for-boundary}
    \delta_s(x)(a,b) = \tilde{x}(a)(\tilde{x}(b))^{s(a)}(\tilde{x}(ab))^{-1}
\end{equation}
where here $a\in G$ acts on $\tilde{x}(b) \in \pi/L^3\pi$ via the splitting $s$. For reference, see \cite[Section 5.6, 5.7]{Serre97book}.
\end{definition}

\begin{definition}\label{def:m}
We define a map $m: H^1(G, \pi/L^2\pi)^{\otimes 2} \to H^2(G, L^2\pi/L^3\pi)$ as the  composition of the cup product with the map on $H^2$ induced by the commutator map:
$$m: H^1(G, \pi/L^2\pi)^{\otimes 2} \overset{\cup}{\longrightarrow} H^2(G, (\pi/L^2\pi)^{\otimes 2})\overset{[-,-]}{\longrightarrow} H^2(G, L^2\pi/L^3\pi).$$  Explicitly, the second map is induced by the following map between coefficients: $$(\pi/L^2\pi)^{\otimes 2}\to L^2\pi/L^3\pi$$ $$\alpha \otimes \beta \mapsto \tilde{\alpha}\tilde{\beta}\tilde{\alpha}^{-1}\tilde{\beta}^{-1}$$
where $\tilde{\alpha},\tilde{\beta}$ are lifts of $\alpha,\beta \in \pi/L^2\pi$ to $\pi/L^3\pi$.
\end{definition}
Note that the map $m$ above is defined independent of any choice of section to sequence \ref{seq:abelian_seq}.

\begin{proposition}[Compare to {\cite[Proposition 1]{ellenberg2nilpotent}}]\label{prop:non-linear-bdry-map}
Let $\delta_s$ and $m$ be the maps defined in Definitions \ref{def:boundary}, \ref{def:m}.  Then we have $$\delta_s(x+y)-\delta_s(x)-\delta_s(y)=m(x\otimes  y).$$ 
\end{proposition}

\begin{proof}
Let $x,y: G \to \pi/L^2\pi$ be cocycles representing classes in $H^1(G,\pi/L^2\pi)$ and let $\tilde{x},\tilde{y}: G \to \pi/L^3\pi$ be continuous set-theoretic lifts of $x,y$ from Sequence \ref{seq:nilpotent_seq}.

By definition, a cocycle representing the class $ x \otimes y \in H^2(G,(\pi/L^2\pi)^{\otimes 2})$ is
\[ x \otimes y: a,b \mapsto x(a) \otimes y(b)^a.  \]
So $m(x \otimes y)$ can be represented by 
\[ m(x \otimes y): a,b \mapsto [\tilde{x}(a), \tilde{y}(b)^a]  \]
where the choice of the action of $a \in G$ on $\tilde{y}(b)$ does not affect this commutator.

Now we have 
\begin{align*}
    &(\delta_s(x+y) - \delta_s(x) -\delta_s(y))(a,b)\\ 
    =& \tilde{x}(a)\tilde{y}(a)\tilde{x}(b)^a\tilde{y}(b)^a \tilde{y}^{-1}(ab) \tilde{x}^{-1}(b)^a \tilde{x}^{-1}(a)\tilde{y}(ab)\tilde{y}^{-1}(b)^a \tilde{y}^{-1}(a)\\
     =& \tilde{x}(a)\tilde{y}(a)\tilde{x}(b)^a
     \tilde{y}^{-1}(a)\tilde{y}(a)
     \tilde{y}(b)^a \tilde{y}^{-1}(ab) \tilde{x}^{-1}(b)^a \tilde{x}^{-1}(a)\tilde{y}(ab)\tilde{y}^{-1}(b)^a \tilde{y}^{-1}(a)\\
    =& \tilde{x}(a)\tilde{y}(a)\tilde{x}(b)^a \tilde{y}^{-1}(a) \tilde{x}^{-1}(b)^a \tilde{x}^{-1}(a)
    \tilde{y}(a)\tilde{y}(b)^a \tilde{y}^{-1}(ab)
    \tilde{y}(ab)\tilde{y}^{-1}(b)^a \tilde{y}^{-1}(a)\\
    =& \tilde{x}(a) [\tilde{y}(a),\tilde{x}(b)^a]\tilde{x}^{-1}(a)\\
    =&[\tilde{y}(a),\tilde{x}(b)^a].
\end{align*}
Here, we used the fact that $\tilde{y}(a)\tilde{y}(b)^a \tilde{y}^{-1}(ab), [\tilde{y}(a),\tilde{x}(b)^a] \in L^2\pi/L^3\pi$ is in the center of $\pi/L^3\pi$. And we conclude by noticing $\delta_s(x+y)-\delta_s(x)-\delta_s(y)$ does not change if we switch $x,y$.
\end{proof}

In particular $\delta_s$ is a homomorphism of abelian groups modulo the image of $m$.

\begin{proposition}\label{prop: bdry-independence-of-section}
Let $s_1, s_2$ be sections to sequence \ref{seq:abelian_seq}. Following Definition \ref{def:boundary}, each section induces a $G$-action on $\pi/L^3\pi$ and hence a boundary map in cohomology $$\delta_{s_i}: H^1(G, \pi/L^2\pi)\to H^2(G, L^2\pi/L^3\pi).$$ Then $$\delta_{s_1}(x)- \delta_{s_2}(x)= m([s_1-s_2]\otimes x),$$ where $m$ is the map defined in Definition \ref{def:m} and $[s_1-s_2] \in H^1(G, \pi/L^2\pi)$ is the element classifying the difference between $s_1, s_2$ from Proposition \ref{prop:section-H1-torsor}. This class is represented by the cocycle $c_{s_2s_1}: g \mapsto s_1(g)s_2(g)^{-1}$. 
\end{proposition}

\begin{proof}
    Following Definition \ref{def:boundary}, if we denote by $\tilde{x}: G \to \pi/L^3\pi$ a continuous lift of $x$, then
    \[ \delta_{s_i}(x)(a,b) = \tilde{x}(a)\tilde{x}(b)^{s_i(a)}\tilde{x}^{-1}(ab) \]
    where $a \in G$ acts on $\tilde{x}(b) \in \pi/L^3\pi$ through conjugation by $\tilde{s_i}(a)$, where $\tilde{s_i}: G \to \tilde{\pi}/L^3\pi$ is a continuous lift of $s_i$ to $\tilde\pi/L^3\pi$, i.e. \[ \tilde{x}(b)^{s_i(a)}=\tilde{s_i}(a)\tilde{x}(b)\tilde{s_i}(a)^{-1}. \]
    Note that since different choices of lifts differ by an element in $L^2\pi/L^3\pi$, which is in the center of $\pi/L^3\pi$, this action is independent of $\tilde s_i$, justifying the notation. So now we have
    \begin{align*}
    (\delta_{s_1}(x)-\delta_{s_2}(x))(a,b)
    =& \tilde{x}(a)\tilde{s_1}(a)\tilde{x}(b)\tilde{s_1}^{-1}(a)\tilde{x}^{-1}(ab)\tilde{x}(ab)\tilde{s_2}(a)\tilde{x}^{-1}(b)\tilde{s_2}^{-1}(a)\tilde{x}^{-1}(a)\\
    =& \tilde{x}(a)\tilde{s_1}(a)\tilde{x}(b)\tilde{s_1}^{-1}(a)\tilde{s_2}(a)\tilde{x}^{-1}(b)\tilde{s_2}^{-1}(a)\tilde{x}^{-1}(a)\\
    =& \tilde{x}(a)\tilde{s_1}(a)\tilde{s_2}^{-1}(a)\tilde{s_2}(a)
    \tilde{x}(b)\tilde{s_2}^{-1}(a)\tilde{s_2}(a)
    \tilde{s_1}^{-1}(a)\tilde{s_2}(a)\tilde{x}^{-1}(b)\tilde{s_2}^{-1}(a)\tilde{x}^{-1}(a)\\
    =& \tilde{x}(a)[\tilde{s_1}(a)\tilde{s_2}^{-1}(a),\tilde{s_2}(a)
    \tilde{x}(b)\tilde{s_2}^{-1}(a)]\tilde{x}^{-1}(a)\\
    =& m(c_{s_2s_1}\otimes x)(a,b).
\end{align*}

Here we used the fact that $m(c_{s_2s_1}\otimes x)(a,b)=[\tilde{c}_{s_2s_2}(a),\tilde{x}(b)^a]$ is defined independent of the choice of the $G$-action on $\pi/L^3\pi$, so we chose to take the action to be conjugation by $\tilde{s_2}(a)$.  Also, the elements $\tilde{s_1}(a)\tilde{s_2}^{-1}(a), \tilde{s_2}(a)\tilde{x}(b)\tilde{s_2}^{-1}(a) \in \pi/L^3\pi$, which implies their commutator lies in $L^2\pi/L^3\pi$ and hence commutes with $\tilde{x}(a) \in \pi/L^3\pi$.

\end{proof}

\subsubsection{Construction of the class $\widetilde{o_2}$}
\begin{definition}\label{defn:independent-bdry-map}
Let $$\overline{H^2(G, L^2\pi/L^3\pi)}:=H^2(G, L^2\pi/L^3\pi)/\on{im}(m).$$ By Proposition \ref{prop: bdry-independence-of-section}, for any two sections $s_1, s_2$ to sequence \ref{seq:abelian_seq}, the composite maps $$H^1(G, \pi/L^2\pi)\overset{\delta_{s_i}}{\longrightarrow} H^2(G, L^2\pi/L^3\pi)\to \overline{H^2(G, L^2\pi/L^3\pi)}$$ are canonically identified. We denote this (canonical) composite map by $\delta$. Note that by Proposition \ref{prop:non-linear-bdry-map}, $\delta$ is linear.
\end{definition}
As $L^2\pi/L^3\pi$ is abelian, sequence \ref{seq:l2modl3extension} gives rise to a class in $b\in H^2(\tilde \pi/L^2\pi, L^2\pi/L^3\pi)$. Given a splitting $s$ of sequence \ref{seq:abelian_seq}, there is an induced class $s^*b\in H^2(G, L^2\pi/L^3\pi)$; this class depends on $s$.
\begin{proposition}\label{prop:obs-independent-of-section}
Suppose $s_1, s_2$ are two sections to sequence \ref{seq:abelian_seq}, with difference $$[s_1-s_2]\in H^1(G, \pi/L^2\pi).$$ Denote by $\overline{s_1^*b-s_2^*b}$ the image of $s_1^*b-s_2^*b$ in $\overline{H^2(G, L^2\pi/L^3\pi)}.$ Then we have
$$\overline{s_1^*b-s_2^*b}= \delta([s_1-s_2]), \text{ where }$$ $$\delta: H^1(G, \pi/L^2\pi)\to \overline{H^2(G, L^2\pi/L^3\pi)}$$ is the map from Definition \ref{defn:independent-bdry-map}.
\end{proposition}

\begin{proof}
By definition, the extension class $b \in H^2(\tilde{\pi}/L^2\pi,L^2\pi/L^3\pi)$ is represented by a cocycle $g_1,g_2 \mapsto \widetilde{g_1}\widetilde{g_2}\widetilde{g_1g_2}^{-1}$ where $\widetilde{g_1},\widetilde{g_2},\widetilde{g_1g_2}$ are lifts of $g_1,g_2,g_1g_2 \in \tilde{\pi}/L^2\pi$ to $\tilde{\pi}/L^3\pi$. Thus, the induced class $s_i^*b$ is represented by 
\[x,y \mapsto \tilde{s_i}(x)\tilde{s_i}(y)\tilde{s_i}^{-1}(xy)\]
where $\tilde{s_1},\tilde{s_2}:G \to \tilde{\pi}/L^3\pi$ are continuous lifts of $s_1,s_2$.

We will prove the desired statement by showing 
\[ s_1^*b-s_2^*b = \delta_{s_2}([s_1-s_2]) \in H^2(G,L^2\pi/L^3\pi). \]
Let $c_{s_2s_1}: g \mapsto s_1(g)s_2(g)^{-1}$ be a cocycle representing the class $[s_1-s_2]$. Then $g \to \tilde{s_1}(g)\tilde{s_2}^{-1}(g)$ is a continuous lift of $c_{s_2s_1}$ and by Definition \ref{def:boundary}, we have 
\begin{align*}
\delta_{s_2}(c_{s_2s_1})(x,y) =& \tilde{s_1}(x)\tilde{s_2}^{-1}(x) (\tilde{s_1}(y)\tilde{s_2}^{-1}(y))^{s_2(x)} \tilde{s_2}(xy)\tilde{s_1}^{-1}(xy)\\
=& \tilde{s_1}(x)\tilde{s_2}^{-1}(x)\tilde{s_2}(x) \tilde{s_1}(y)\tilde{s_2}^{-1}(y)\tilde{s_2}^{-1}(x) \tilde{s_2}(xy)\tilde{s_1}^{-1}(xy)\\
=& \tilde{s_1}(x)\tilde{s_1}(y)\tilde{s_2}^{-1}(y)\tilde{s_2}^{-1}(x) \tilde{s_2}(xy)\tilde{s_1}^{-1}(xy)\\
=& (\tilde{s_1}(x)\tilde{s_1}(y)\tilde{s_1}^{-1}(xy))(\tilde{s_1}(xy)\tilde{s_2}^{-1}(xy))(\tilde{s_2}(xy)\tilde{s_2}^{-1}(y)\tilde{s_2}^{-1}(x))( \tilde{s_2}(xy)\tilde{s_1}^{-1}(xy))\\
=&s_1^*b(x,y)-s_2^*b(x,y)
\end{align*}
In the last step, we used that the element $\tilde{s_1}(xy)\tilde{s_2}^{-1}(xy)$ lies in $\pi/L^3\pi$ and the element $\tilde{s_2}(xy)\tilde{s_2}^{-1}(y)\tilde{s_2}^{-1}(x)$ lies in $L^2\pi/L^3\pi$. So they commute with each other. 
\end{proof}

\begin{definition}\label{def:obs-independent-of-section}
Let $$M(G, L^2\pi/L^3\pi)=\on{coker}(\delta:  H^1(G, \pi/L^2\pi)\to \overline{H^2(G, L^2\pi/L^3\pi)}).$$ Then for any section $s: G\to \tilde\pi/L^2\pi$ of sequence \ref{seq:abelian_seq}, we let $\widetilde{o_2}\in M(G, L^2\pi/L^3\pi)$ be the image of $s^*b$. This class is independent of $s$ by Proposition \ref{prop:obs-independent-of-section}.
\end{definition}

\begin{remark}\label{rem:o2-functoriality}
Let $\gamma: G' \to G$ be a group homomorphism. Pulling back the sequence $$1\to \pi\to \tilde\pi\to G\to 1$$ along this line, we obtain a sequence $$1\to \pi\to \tilde\pi\times_GG'\to G'\to 1$$ with analogous properties to the corresponding sequence for $G$; hence we may define $M(G', L^2\pi/L^3\pi)$. There exists a unique map $\gamma^*:M(G,L^2\pi/L^3\pi) \to M(G',L^2\pi/L^3\pi)$ which makes the following diagram commute:
\[
\xymatrix{
H^2(G,L^2\pi/L^3\pi) \ar[r]^{\gamma^*}\ar[d] &H^2(G',L^2\pi/L^3\pi) \ar[d]\\
M(G,L^2\pi/L^3\pi) \ar[r]^{\gamma^*} &M(G',L^2\pi/L^3\pi)}
\]
From the definition it is clear that $\gamma^*\widetilde{o_2}=\widetilde{o_2}$.
\end{remark}
\begin{proposition}
The obstruction $\widetilde{o_2}\in M(G, L^2\pi/L^3\pi)$ vanishes if the sequence \[ 1 \rightarrow \pi/L^3 \pi \rightarrow \tilde{\pi}/L^3 \pi \rightarrow G \rightarrow 1\] splits.
\end{proposition}
\begin{proof}
Immediate from the definition.
\end{proof}
\subsubsection{Gysin images of $\widetilde{o_2}$}\label{subsubsec:defn-of-N}
Throughout this paper we will consider various Gysin images of $\widetilde{o_2}$, as we now explain. Suppose that we are in the situation described above, and moreover $G$ sits in a short exact sequence $$1\to I\to G\to H\to 1$$ with $I=\mathbb{Z}$ (in the discrete setting) or $I=\widehat{\mathbb{Z}}$ (in the pro-finite setting). Then as in Section \ref{subsec:gysin}, we have a natural (Gysin) map $$g: H^2(G, L^2\pi/L^3\pi)\to H^1(H, (L^2\pi/L^3\pi)_I)$$ arising from the Hochschild-Serre spectral sequence.

We let $$\overline{H^1(H, (L^2\pi/L^3\pi)_I)}:=\on{coker}(g\circ m: H^1(G, \pi/L^2\pi)^{\otimes 2}\to H^1(H, (L^2\pi/L^3\pi)_I))$$ and define $$N(H, (L^2\pi/L^3\pi)_I):=\on{coker}(g\circ \delta: H^1(G, \pi/L^2\pi)\to \overline{H^1(H, (L^2\pi/L^3\pi)_I)}).$$ Then by definition the natural Gysin map $g$ above descends to a map $$h: M(G, L^2\pi/L^3\pi)\to N(H, (L^2\pi/L^3\pi)_I).$$

We now discuss the functoriality properties of the map $h$ defined above. Suppose we have a map of short exact sequences of groups
$$\xymatrix{
1\ar[r] & I' \ar[r] \ar[d]^\sim & \ar[r] G'\ar[d]^\gamma & H' \ar[r] \ar[d]^{\eta} & 1\\
1\ar[r] & I \ar[r] & \ar[r] G & H \ar[r] & 1\\
}$$
inducing an isomorphism $I'\overset{\sim}{\to} I$ as above.

Pulling back the sequence $$1\to \pi\to \tilde\pi\to G\to 1$$ along this line, we obtain a sequence $$1\to \pi\to \tilde\pi\times_GG'\to G'\to 1$$ with analogous properties to the corresponding sequence for $G$; hence we have an analogous map $$h':  M(G', L^2\pi/L^3\pi)\to N(H', (L^2\pi/L^3\pi)_{I'}).$$ It is immediate from the functoriality of the Hochschild-Serre spectral sequence that the evident square
$$\xymatrix{
 M(G, L^2\pi/L^3\pi)\ar[r]^h \ar[d]^{\gamma^*} &  N(H, (L^2\pi/L^3\pi)_I) \ar[d]^{\eta^*}\\
M(G', L^2\pi/L^3\pi)\ar[r]^{h'} & N(H', (L^2\pi/L^3\pi)_{I'})
}$$
commutes.
\subsection{Cohomological preliminaries for surface groups}
We now specialize to the case where $G$ is a surface group, i.e.~we
let $G= \pi_1(\Sigma_h)$ for some $h \ge 1$. $G$ has presentation $$G=\langle a_1, b_1 \cdots, a_h, b_h \mid \prod_{i=1}^h[a_i, b_i]\rangle.$$ We will make the computations described in Section \ref{subsec:nonabelian-cohomology} explicit in this case. 

Let $G$ act on $\mathbb{Z}[G]$ by multiplication on the left. We begin by introducing a finite free resolution of $\mathbb{Z}$ as a (trivial) $\mathbb{Z}[G]$-module. Let $R_i = [a_i,b_i]$ and $R=\prod_{i=1}^{h}R_i$. As a convention, we set $a_0=b_0=1$.

\begin{proposition}[\cite{GM15}, Proposition 2.1]\label{prop:resolution-of-Z}
There is an exact sequence of $\mathbb{Z}[G]$-modules 
\begin{equation}\label{equ:resolution-of-Z}
0 \rightarrow P_2 \xrightarrow{d_2} P_1 \xrightarrow{d_1} P_0 \xrightarrow{\epsilon} \mathbb{Z} \rightarrow 0
\end{equation}
where $P_0 = \mathbb{Z}[G]$ with generator $v$, $P_1 = (\mathbb{Z}[G])^{2h}$ with generators $e_i,f_i,i=1,\ldots,h$ and $P_2= \mathbb{Z}[G]$ with generator $u$.
Then, on the generators the maps are defined as
\begin{align*}
\epsilon &\colon v \mapsto 1, \\
d_1 &\colon e_i \mapsto (a_i-1)v,\\
&\ \  f_i \mapsto (b_i-1)v, \\
d_2 &\colon u \mapsto \sum_{i=1}^{h}\left( \frac{\partial R}{\partial a_i}e_i+ \frac{\partial R}{\partial b_i}f_i \right),
\end{align*}
where the partial derivatives are the Fox derivatives.
\end{proposition}

Let $M$ be a left $G$-module and apply the functor $\Hom (\bullet,M)$ to Sequence \eqref{equ:resolution-of-Z}. We get a sequence
\[0 \longleftarrow \Hom(P_2, M) \overset{d_2^*}{\longleftarrow} \Hom(P_1, M) \overset{d_1^*}{\longleftarrow} \Hom(P_0, M) \longleftarrow \ldots  \]
whose cohomology is precisely $H^*(G,M)$. Since $P_2$ is a free $\mathbb{Z}[G]$-module generated by a single element $u$, the group $\Hom(P_2,M)$ is isomorphic to $M$ via the map $\phi\mapsto \phi(u)$. Thus, each class in $H^2(G,M)$ can be represented by an element of $M$. The following corollary gives an explicit description of $H^2(G,M)$ in these terms.

\begin{corollary}[\cite{GM15} Corollary 3.1, \cite{Lyndon} Corollary 11.2]\label{cor:H2forG-one-relation}
For any left $G$-module $M$,
\[H^2(G,M) \simeq M/\left(\frac{\partial R}{\partial a_1},\frac{\partial R}{\partial b_1},\ldots,\frac{\partial R}{\partial a_h},\frac{\partial R}{\partial b_h} \right)M.\]
\end{corollary}

\begin{corollary}\label{h2-corollary} Let $I \colonequals \ker(\epsilon)$. For any left $G$-module $M$,
$$H^2(G, M) \simeq M/IM = M_{G}.$$ 
\end{corollary}

\begin{proof}
Using Corollary 
\ref{cor:H2forG-one-relation}, the only thing to show is $\left(\frac{\partial R}{\partial a_1},\frac{\partial R}{\partial b_1},\ldots,\frac{\partial R}{\partial a_h},\frac{\partial R}{\partial b_h} \right)M=IM$. Since 
\begin{align*}
    \frac{\partial R}{\partial a_i} &= \left(\prod_{j=0}^{i-1}R_j\right)a_i(1-b_i)a_i^{-1},
    \frac{\partial R}{\partial b_i} = \left(\prod_{j=0}^{i-1}R_j\right) a_ib_i(1-a_i^{-1})b_i^{-1},
\end{align*}
we get an inclusion $\left(\frac{\partial R}{\partial a_1},\frac{\partial R}{\partial b_1},\ldots,\frac{\partial R}{\partial a_h},\frac{\partial R}{\partial b_h} \right)M\subset IM$.

The $\mathbb{Z}[G]$-module $IM$ is generated by elements of the form $(1-a_i)m,(1-b_i)m, i=1,\ldots,h,m\in M$. For any $m \in M$, let $m_i'=a_im$ and $m_i''=b_im$. Then we have 
\begin{align*}
    (1-a_i)m=& -a_ib_i^{-1}a_i^{-1}\left(\prod_{j=0}^{i-1}R_j\right)^{-1}\frac{\partial R}{\partial b_i} m_i'',
       (1-b_i)m= a_i^{-1}\left(\prod_{j=0}^{i-1}R_j\right)^{-1}\frac{\partial R}{\partial a_i} m_i'.
\end{align*}
So we have the reverse inclusion and conclude the statement.
\end{proof}
Note that the abstract isomorphism in the statement of Corollary \ref{h2-corollary} follows from Poincar\'e duality for surfaces; the purpose of the corollary is to make this isomorphism explicit in terms of the free resolution of Proposition \ref{prop:resolution-of-Z}.

\begin{proposition}\label{prop:cup-product-surface-group}
Let $\phi \otimes \psi \in (\ker d_2^*)^{\otimes 2} \subset \Hom(P_1,M)^{\otimes 2}$ represent the class $[\phi] \otimes [\psi] \in H^1(G,M)^{\otimes 2} $. Then the class $[\phi]\cup[\psi] \in H^2(G,M^{\otimes 2})$ is represented by a cocycle $\phi \cup \psi \in \Hom(\mathbb{Z}[G],M^{\otimes 2})$ defined by $u \mapsto \phi \otimes \psi (\Delta_{11}(u))$ where $\Delta_{11}:P_2 \to P_1 \otimes P_1$ is a part of a diagonal approximation $\Delta: P \to P \otimes P$ constructed in \cite[Theorem 2.2]{GM15}.
\end{proposition}

\begin{proof}
This immediately follows from the definition of cup product. See for example \cite[Section 5.3]{Brown}.
\end{proof}

\begin{lemma}\label{lem:free-Lie-alg-2nd-piece}
The map
\begin{align*}
    L^2G/L^3G &\to (\wedge^2 (G/L^2G))/\langle \sum_{i=1}^ga_i\wedge b_i \rangle\\
[x,y] &\mapsto [x] \wedge [y]
\end{align*} 
is an isomorphism where $[x]$ is the image of $x \in G$ under the natural map $G \to G/L^2G$. 
\end{lemma}

\begin{proof}
Let $F = \langle a_1,b_1,\ldots,a_g,b_g \rangle$ be the free group. By statement (4) in \cite[Section 4.3, page 20]{SerreLie}, the graded Lie algebra $(F/L^2F) \oplus (L^2F/L^3F)$ with the bracket operation given by the commutator is isomorphic to $(F/L^2F) \oplus (\wedge^2 (F/L^2F))$. Then by \cite[Section 1, Theorem on page 17]{Labute}, the statement in our lemma for a group $G$ with a single defining relation follows.
\end{proof}

\subsubsection{Computation for the surface group $\pi_1(\Sigma_1)$}\label{subsubsec:genus-1-computations}

Now we specialize to the case $h=1$ and carry out some computations to be used in Section \ref{sec:computation-secondary-Morita}. 

Let 
$$G = \pi_1(\Sigma_1) = \langle S,T|[S,T]\rangle$$
with an outer action on $\pi_1(\Sigma_g)$ given by a sequence 
\[1 \to \pi_1(\Sigma_g) \to \pi_1(E) \to G \to 1,\] where $E$ is a $\Sigma_g$-bundle over $\Sigma_1$.
In this section, we will be considering the two $G$-modules $\pi_1(\Sigma_g)/L^2\pi_1(\Sigma_g)$ and $L^2\pi_1(\Sigma_g)/L^3\pi_1(\Sigma_g)$.

The free resolution of $\mathbb{Z}$ as a $\mathbb{Z}[G]$-module in Proposition \ref{prop:resolution-of-Z} in this case is given as follows:
\begin{equation}\label{eq:Finite-resolution-of-h=1}
    0 \longrightarrow P_2 \overset{d_2}{\longrightarrow} P_1 \overset{d_1}{\longrightarrow} P_0 \overset{\epsilon}{\longrightarrow} \mathbb{Z} \longrightarrow 0  
\end{equation}
where $P_2 = \mathbb{Z}[G]$ with generator $u$, $P_1=\mathbb{Z}[G]^2$ with generators $e,f$ and $P_0=\mathbb{Z}[G]$ with generator $v$. The maps are: 
\begin{align*}
    d_2(u) =& (1-S)e - (1-T)f,\\
    d_1(e) =& (T-1)v,\quad d_1(f) = (S-1)v.
\end{align*}
Appling the functor $\Hom(\bullet,\pi_1(\Sigma_g)/L^2\pi_1(\Sigma_g))$ to this resolution, we see that the groups $H^*(G,\pi_1(\Sigma_g)/L^2\pi_1(\Sigma_g))$ are given by the cohomology of the sequence
\[0 \longleftarrow \Hom(P_2, \pi_1(\Sigma_g)/L^2\pi_1(\Sigma_g)) \overset{d_2^*}{\longleftarrow} \Hom(P_1, \pi_1(\Sigma_g)/L^2\pi_1(\Sigma_g)) \overset{d_1^*}{\longleftarrow} \Hom(P_0, \pi_1(\Sigma_g)/L^2\pi_1(\Sigma_g)) \longleftarrow \ldots  \]

Cocycles in $\ker d_2^*$ are given by $\phi \in \Hom (P_1,\pi_1(\Sigma_g)/L^2\pi_1(\Sigma_g))$ satisfying
\begin{equation}\label{eq:cocycle-condition}
    (1-S)\phi(e)-(1-T)\phi(f)=0.
\end{equation}

\begin{lemma}\label{lem:compute-m}
Let $\phi,\psi \in \ker d_2^*$ represent classes $[\phi],[\psi] \in H^1(G,\pi_1(\Sigma_g)/L^2\pi_1(\Sigma_g))$. Then a cocycle in $\Hom(P_2,L^2\pi_1(\Sigma_g)/L^3\pi_1(\Sigma_g))$ representing the class $m([\phi]\otimes[\psi])$ is given by
\begin{equation}\label{equ:m-for-h=1}
    m(\phi \otimes \psi)(u)=\phi(e) \wedge T\psi(f) - \phi(f) \wedge S\psi(e).
\end{equation}
\end{lemma}

\begin{proof}

By Proposition \ref{prop:cup-product-surface-group}, an element in $\Hom(P_2,\pi/L^2\pi)$ which represents the class $$[\phi] \cup [\psi] \in H^2(G,(\pi_1(\Sigma_g)/L^2\pi_1(\Sigma_g))^{\otimes 2})$$ is given by
\[ (\phi \otimes \psi \circ \Delta_{11} )(u) = \phi(e)\otimes \psi(T f) - \phi(f) \otimes \psi(S e) \]
where the map 
$$\Delta_{11}: u \mapsto e\otimes Tf - f\otimes Se  $$
is from \cite[Theorem 2.2]{GM15}.
Thus by Definition \ref{def:m} and Lemma \ref{lem:free-Lie-alg-2nd-piece}, 
\begin{equation*}
    m(\phi \otimes \psi)(u)=\phi(e) \wedge T\psi(f) - \phi(f) \wedge S\psi(e)
\end{equation*}
represents the class $m([\phi]\otimes[\psi]) \in H^2(G,L^2\pi_1(\Sigma_g)/L^3\pi_1(\Sigma_g))$.
\end{proof}

In the following, we consider the case where the sequence 
\[ 1 \to \pi_1(\Sigma_g)/L^2\pi_1(\Sigma_g) \to \pi_1(E)/L^2\pi_1(\Sigma_g)\to G \to 1 \]
splits. So we fix a section $s: G \to \pi_1(E)/L^2\pi_1(\Sigma_g)$ and let $S,T$ act on $\pi_1(\Sigma_g)/L^3\pi_1(\Sigma_g)$ via the section $s$. We compute the boundary map $\delta_s$ from Definition \ref{def:boundary} explicitly in the following lemma.

\begin{lemma}\label{lem:compute-delta}
Let $\phi \in \ker d_2^*$ represent the class $[\phi] \in H^1(G,\pi_1(\Sigma_g)/L^2\pi_1(\Sigma_g))$. Then a cocycle in $\Hom(P_2,L^2\pi_1(\Sigma_g)/L^3\pi_1(\Sigma_g))$ representing the class $\delta_s([\phi])$ is given by
\[ \delta_s(\phi)(u) = [\widetilde{\phi(e)}^{-1},\widetilde{\phi(f)}] \left(\widetilde{\phi(f)}^T\right)\widetilde{\phi(f)}^{-1}\widetilde{\phi(e)}\left(\widetilde{\phi(e)}^S \right)^{-1} \]
where $\widetilde{\phi(e)},\widetilde{\phi(f)} \in \pi_1(\Sigma_g)/L^3\pi_1(\Sigma_g)$ are lifts of $\phi(e),\phi(f) \in \pi_1(\Sigma_g)/L^2\pi_1(\Sigma_g)$.
\end{lemma}

\begin{proof}
We start with giving an embedding of the resolution \eqref{eq:Finite-resolution-of-h=1} into the standard resolution of $\mathbb{Z}$ as a $\mathbb{Z}[G]$-module as follows:
$$\xymatrix{
0 \ar[r] & \mathbb{Z}[G] \ar[r]^{d_2} \ar[d]^{\iota_2} & \mathbb{Z}[G]^2 \ar[r]^{d_1} \ar[d]^{\iota_1} & \mathbb{Z}[G] \ar[r]^{\epsilon} \ar[d]^{\simeq} & \mathbb{Z} \ar[r] \ar[d]^{\simeq}& 0  \\
\cdots \ar[r] &\mathbb{Z}[G^3] \ar[r]^{D_2}  & \mathbb{Z}[G^2] \ar[r]^{D_1} & \mathbb{Z}[G] \ar[r]^{\epsilon} & \mathbb{Z} \ar[r] & 0
}$$
where the maps are given by $$\iota_2(u) = (ST,S,T)-(S,T,1),$$
\[ \iota_1(e)=-(T,1), \iota_1(f)=-(S,1). \]
Since $\iota_1$ is injective and
$\mathbb{Z}[G^2] = \oplus_{g \in G} \mathbb{Z}[G](1,g) $ is a free $\mathbb{Z}[G]$-module, there exists 
$$\phi' \in \ker D_2^{*} \subset \Hom_{\mathbb{Z}[G]}(\mathbb{Z}[G^2],\pi_1(\Sigma_g)/L^2\pi_1(\Sigma_g))$$ 
such that $\phi'((1,T^{-1})) = (\phi(e)^{-1})^{T^{-1}}$, $\phi'((1,S^{-1})) =(\phi(f)^{-1})^{S^{-1}}$, or equivalently  $\iota_1 \circ \phi' = \phi.$ Now direct computation shows that an inhomogeneous cocycle  $x: G \to \pi_1(\Sigma_g)/L^2\pi_1(\Sigma_g)$ representing the class $[\phi] \in H^1(G,\pi_1(\Sigma_g)/L^2\pi_1(\Sigma_g))$ is given by
\[x(T^{-1}) = (\phi(e)^{-1})^{T^{-1}} ,  x(S^{-1}) = (\phi(f)^{-1})^{S^{-1}}\] and hence $$x(T)=\phi(e),~x(S^{-1}T)=(\phi(f)^{-1}\phi(e))^{S^{-1}}$$
Now choose a lift $\tilde x$ of $x$ to $\pi_1(\Sigma_g)/L^3\pi_1(\Sigma_g)$ such that $$\tilde x(T^{-1})=(\widetilde{\phi(e)}^{-1})^{T^{-1}}, \tilde{x}(S^{-1}T)=(\widetilde{\phi(f)}^{-1}\widetilde{\phi(e)})^{S^{-1}}$$

Using Equation \eqref{eq:cocycle-for-boundary} in Definition \ref{def:boundary}, we may compute an inhomogeneous cocycle representing $\delta_s([x])=\delta_s([\phi])$. Translating back into homogeneous cocycles and evaluating on $\iota_2(u)$ gives:
\begin{align*}
 & (\delta_s(x)(S^{-1}T,T^{-1}))^S (\delta_s(x)(T^{-1},S^{-1}T))^{ST})^{-1}\\
    =& \left(\left(\widetilde{\phi(e)}^{-1}\widetilde{\phi(f)}\right)^{S^{-1}} \left(\widetilde{\phi(e)}^{T^{-1}}\right)^{S^{-1}T} \left(\widetilde{\phi(f)}^{S^{-1}}\right)^{-1}\right)^S
    \left(\left(\widetilde{\phi(e)}^{T^{-1}}\left(\widetilde{\phi(e)}^{-1}\widetilde{\phi(f)}\right)^{S^{-1}}\left(\widetilde{\phi(f)}^{S^{-1}}\right)^{-1}\right)^{ST}\right)^{-1}\\
    =&[\widetilde{\phi(e)}^{-1},\widetilde{\phi(f)}] \left(\widetilde{\phi(f)}^T\right)\widetilde{\phi(f)}^{-1}\widetilde{\phi(e)}\left(\widetilde{\phi(e)}^{S} \right)^{-1},
\end{align*}
as desired.
\end{proof}

\begin{remark}
It might not seem obvious that the element $\left(\widetilde{\phi(f)}^T\right)\widetilde{\phi(f)}^{-1}\widetilde{\phi(e)}\left(\widetilde{\phi(e)}^{S} \right)^{-1}$ is in $L^2\pi_1(\Sigma_g)$.
But if we apply the natural map $\pi_1(\Sigma_g)/L^3\pi_1(\Sigma_g) \to \pi_1(\Sigma_g)/L^2\pi_1(\Sigma_g)$ to it, we get $T (\phi(f))-\phi(f)+\phi(e)-S(\phi(e))$, which is $0$ by Equation \eqref{eq:cocycle-condition}. Thus this claim follows from the fact that $\phi$ was a cocycle.
\end{remark}

Finally, we record the following diagram, which will be useful for our computations in Section \ref{sec:computation-secondary-Morita}.
\begin{lemma}\label{lem:pass-to-H_G}
Let $H=\pi_1(\Sigma_g)/L^2\pi_1(\Sigma_g)$, $\omega = \sum_{i=1}^g a_i \wedge b_i$, and let $\overline{\omega}$ be the image of $\omega$ under the natural map $\wedge^2 H \to \wedge^2 H_G$. Then there exists a unique map $h$ which makes the following diagram commute.
\[
\xymatrix{
(\wedge^2 H)/\langle \omega \rangle \ar[r]^{f} \ar[d]^{g} & ((\wedge^2 H)/\langle \omega \rangle )_G \ar[ld]^{h}\\
(\wedge^2 H_G)/\langle \overline{\omega} \rangle  &
}
\]
\end{lemma}

\begin{proof}
The only thing we need to show is that $\ker f \subset \ker g$. The group $\ker f$ is generated by elements of the form $g\alpha \wedge g\beta - \alpha\wedge\beta$, for $g\in G$. Observe that elements of the form $(g-1)\alpha\wedge\beta$ are contained in $\ker g$. Since $g\alpha \wedge g\beta - \alpha\wedge\beta = g\alpha\wedge(g-1)\beta+(g-1)\alpha\wedge\beta$, we conclude our lemma.
\end{proof}

\bibliographystyle{alpha}
\bibliography{SectionConjecture}

\end{document}